%% file: DiracOperatorsTwistedByRamifiedLineBundles.tex
\documentclass[leqno,11pt]{article}
\input{preamble/preamble}

\newcommand{\dom}{\mathrm{dom}}
\newcommand{\Dmin}{D_{\mathrm{min}}}
\newcommand{\Dmax}{D_{\mathrm{max}}}
\newcommand{\mathringDmax}{{\mathring{D}_{\mathrm{max}}}}
\newcommand{\mathringDmin}{{\mathring{D}_{\mathrm{min}}}}
\DeclareMathOperator{\ErrorTerm}{\mathrm{Err}}
\newcommand{\APS}{\mathrm{APS}}
\newcommand{\bag}{\mathrm{bag}}
\newcommand{\GelfandRobbinQuotient}{\check\bH}
\newcommand{\DualGelfandRobbinQuotient}{\hat\bH}
\newcommand{\ResidueCondition}{R}
\newcommand{\AbstractBag}{B}
\newcommand{\Spectrum}{\sigma}
\newcommand{\DiffOp}{\mathrm{DiffOp}}
\newcommand{\adm}{a}

\author{%
  Gorapada Bera
  \and
  Thomas Walpuski
}
\date{2026-09-14}
\title{Dirac operators twisted by ramified Euclidean line bundles}

\hypersetup{
  pdfauthor={Gorapada Bera, Thomas Walpuski},
  pdftitle={Dirac operators twisted by ramified Euclidean line bundles},
  pdflang={English}}

\begin{document}

\maketitle

\begin{abstract}
  This article is concerned with the analysis of Dirac operators $D$ twisted by ramified Euclidean line bundles $(Z,\fl)$---%
  motivated by their relation with $\Z/2\Z$ harmonic spinors,
  which have appeared in various contexts in gauge theory and calibrated geometry.
  The closed extensions of $D$ are described in terms of the Gelfand--Robbin quotient $\GelfandRobbinQuotient$.  
  Assuming that the branching locus $Z$ is a closed codimension two submanifold,
  a geometric realisation of $\GelfandRobbinQuotient$ is constructed.
  This, in turn, leads to an $L^2$ regularity theory.
\end{abstract}

\begin{center}
  \emph{Dedicated to the memory of Dietmar Salamon.}
\end{center}

\tableofcontents

\input{Introduction}
\input{GelfandRobbinQuotientAbstract}
\input{GelfandRobbinQuotientGeometric}
\input{RegularityTheory}

\appendix
\input{NonCoorientableBranchingLoci}

\printbibliography

\end{document}

%%% Local Variables:
%%% mode: latex
%%% TeX-master: t
%%% ispell-local-dictionary: "british"
%%% End:

%% file: preamble/preamble.tex
\usepackage[utf8]{inputenc}
\usepackage[T1]{fontenc}
\usepackage{microtype}

\usepackage[dvipsnames]{xcolor}
\usepackage{xcolor-solarized}

\usepackage{calc}
\usepackage[a4paper]{geometry}

\ifdefined\screenLayout
  \pagecolor{solarized-base3}
  \color{solarized-base03}
  \usepackage{fancyhdr}
\fi

\usepackage{amsmath}
\usepackage{amsthm}

\usepackage{tikz-cd}
\usetikzlibrary{arrows} 
\tikzset{
  commutative diagrams/.cd, 
  arrow style=tikz, 
  diagrams={>=stealth}
}
\usetikzlibrary{matrix,decorations.pathreplacing,calc}
\usetikzlibrary{graphs,graphs.standard}
\usepackage{pgfplots}
\usepgfplotslibrary{external}
\usepackage{afterpage}
\usepackage{pdflscape}

\usepackage{colortbl}
\usepackage{adforn}
\usepackage{nameref}

\usepackage{textcomp}
\usepackage[sb]{libertine}
\usepackage[varqu,varl]{zi4}%
\usepackage[libertine,bigdelims,vvarbb]{newtxmath}

\IfPackageAtLeastTF{superiors}{2.0}{%
  \usepackage[supsfam=LibertinusSerif-Sup,supscaled=1.2,raised=-.13em]{superiors}
}{%
  \usepackage[supstfm=libertinesups,supscaled=1.2,raised=-.13em]{superiors}
}

\useosf
\usepackage[scr=boondox,cal=euler]{mathalfa}
\usepackage{marvosym}
\usepackage{slashed}
\usepackage{esint} % SLASHED INTEGRAL

\usepackage[ngerman,english]{babel}
\usepackage{imakeidx}
\makeindex[intoc]
\usepackage{csquotes}
\usepackage[
  backend=biber,
  hyperref=true,
  backref=true,
  isbn=false,
  doi=true,
  natbib=true,
  eprint=true,
  useprefix=true,
  maxcitenames=99,
  maxbibnames=99,  
  maxalphanames=99, 
  minalphanames=99,
  safeinputenc,
  style=alphabetic,
  citestyle=alphabetic,
  block=space,
  datamodel=preamble/ext-eprint,
  sorting=nyt
]{biblatex}
\usepackage[
  bookmarksnumbered = true,
  hypertexnames = false,
  colorlinks    = true,
  citecolor     = solarized-blue,
  linkcolor     = solarized-blue,
  urlcolor      = solarized-blue,
  breaklinks
]{hyperref}

\DeclareFieldFormat{url}{%
  \href{#1}{\ComputerMouse}
}
\DeclareFieldFormat{doi}{%
  \mkbibacro{DOI}\addcolon\space\href{https://doi.org/#1}{#1}
}
\makeatletter
\DeclareFieldFormat{arxiv}{%
  arXiv\addcolon\space\href{http://arxiv.org/\abx@arxivpath/#1}{#1}
}
\makeatother
\DeclareFieldFormat{mr}{%
  MR\addcolon\space\href{http://www.ams.org/mathscinet-getitem?mr=MR#1}{#1}
}
\DeclareFieldFormat{zbmath}{%
  zbMATH\addcolon\space\href{http://zbmath.org/?q=an:#1}{#1}
}
\renewbibmacro*{eprint}{%
  \printfield{arxiv}%
  \newunit\newblock
  \printfield{mr}%
  \newunit\newblock
  \printfield{zbmath}%
  \newunit\newblock
  \iffieldundef{eprinttype}
  {\printfield{eprint}}
  {\printfield[eprint:\strfield{eprinttype}]{eprint}}
}

\AtEveryBibitem{%
  \clearlist{address}%
}
\DeclareFieldFormat[article,inproceedings,inbook,incollection,phdthesis,thesis]{title}{\textit{#1}}
\renewbibmacro{in:}{}
\usepackage[inline,shortlabels]{enumitem}

\usepackage{subcaption}

\usepackage[yyyymmdd]{datetime}

\usepackage{etoolbox}
\ifundef{\abstract}{}{\patchcmd{\abstract}%
    {\quotation}{\quotation\noindent\ignorespaces}{}{}}

\usepackage[super]{nth}

\input{preamble/envs.tex}
\input{preamble/macros.tex}
\input{preamble/letters.tex}

%% file: preamble/envs.tex
\usepackage{thmtools}
\usepackage[framemethod=TikZ]{mdframed}

\numberwithin{equation}{section}

\renewcommand{\qedsymbol}{$\blacksquare$}

\newcommand{\CorollaryQED}{\qedsymbol}
\newcommand{\ConjectureQED}{$\square$}
\newcommand{\SituationQED}{$\times$}
\newcommand{\DefinitionQED}{$\bullet$}
\newcommand{\NotationQED}{$\circ$}
\newcommand{\ExampleQED}{$\spadesuit$}
\newcommand{\RemarkQED}{$\clubsuit$}
\newcommand{\ExerciseQED}{?!}

\ifdefined\screenLayout
  \declaretheoremstyle[
  bodyfont=\itshape,
  mdframed={    
    backgroundcolor=solarized-base3!90!solarized-blue,
    linewidth=0,
    innerleftmargin=.5em,
    innerrightmargin=.5em,
    innertopmargin=.5em,
    innerbottommargin=.5em,
    leftmargin=-.5em,
    rightmargin=-.5em,
  }
]{theorem}

\declaretheoremstyle[
mdframed={
  backgroundcolor=solarized-base3!90!solarized-green,
  linewidth=0,
  innerleftmargin=.5em,
  innerrightmargin=.5em,
  innertopmargin=.5em,
  innerbottommargin=.5em,
  leftmargin=-.5em,
  rightmargin=-.5em,
  }
]{definition}

\declaretheoremstyle[
  mdframed={
    backgroundcolor=solarized-base3!90!solarized-yellow,
    linewidth=0,
    innerleftmargin=.5em,
    innerrightmargin=.5em,
    innertopmargin=.5em,
    innerbottommargin=.5em,
    leftmargin=-.5em,
    rightmargin=-.5em,
  }
]{example}

\declaretheoremstyle[
  mdframed={
    backgroundcolor=solarized-base3!90!solarized-orange,
    linewidth=0,
    innerleftmargin=.5em,
    innerrightmargin=.5em,
    innertopmargin=.5em,
    innerbottommargin=.5em,
    leftmargin=-.5em,
    rightmargin=-.5em,
  }
]{remark}
\else
  \declaretheoremstyle[
  bodyfont=\itshape
  ]{theorem}
  \declaretheoremstyle[]{definition}
  \declaretheoremstyle[]{example}
  \declaretheoremstyle[]{remark}
\fi

\declaretheorem[numberlike=equation,style=theorem]{theorem}
\declaretheorem[numbered=no,name=Theorem,style=theorem]{theorem*}
\declaretheorem[numberlike=equation,name=Lemma,style=theorem]{lemma}
\declaretheorem[numberlike=equation,name=Proposition,style=theorem]{prop}
\declaretheorem[numberlike=equation,name=Corollary,qed=\CorollaryQED,style=theorem]{cor}

\declaretheorem[numberlike=equation,name=Hypothesis]{hypothesis}

\declaretheorem[numberlike=equation,name=Definition,style=definition,qed=\DefinitionQED]{definition}
\declaretheorem[numbered=no,name=Definition,style=definition,qed=\DefinitionQED]{definition*}

\declaretheorem[numberlike=equation,style=definition,qed=\ExampleQED]{example}

\declaretheorem[numberlike=equation,style=remark,qed=\RemarkQED]{remark}
\declaretheorem[numbered=no,style=remark,name=Remark,qed=\RemarkQED]{remark*}

\def\makeautorefname#1#2{\AtBeginDocument{\expandafter\def\csname#1autorefname\endcsname{#2}}}
\makeautorefname{table}{Table}        
\makeautorefname{chapter}{Chapter}
\makeautorefname{section}{Section}
\makeautorefname{subsection}{Section}
\makeautorefname{subsubsection}{Section}
\makeautorefname{footnote}{Footnote}
\AtBeginDocument{\def\itemautorefname~#1\null{(#1)\null}}
\AtBeginDocument{\def\equationautorefname~#1\null{(#1)\null}}

\numberwithin{substep}{step}
\makeautorefname{step}{Step}
\makeautorefname{substep}{Step}

\makeautorefname{case}{Case}
\makeautorefname{substep}{Step}
\AddToHook{env/proof/begin}{%
  \setcounter{step}{0}%
  \setcounter{substep}{0}%
}

\setlist[description]{leftmargin=!,labelindent=1em}
\setlist[enumerate]{label={\rm (\arabic*)},ref=\arabic*}
\setlist[enumerate,2]{label={\rm (\alph*)},ref=\theenumi.\alph*}
\setlist[enumerate,3]{label={\rm (\roman*)},ref=\theenumii.\roman*}

%% file: preamble/macros.tex
\let\C\undefined
\let\U\undefined

\usepackage{bm}
\usepackage{mathtools} % FOR PAIREDDELIMITERS
\usepackage{stmaryrd} % FOR SUPER LIE BRACKET

\DeclareFontFamily{U}{mathx}{\hyphenchar\font45}
\DeclareFontShape{U}{mathx}{m}{n}{
      <5> <6> <7> <8> <9> <10>
      <10.95> <12> <14.4> <17.28> <20.74> <24.88>
      mathx10
      }{}
\DeclareSymbolFont{mathx}{U}{mathx}{m}{n}
\DeclareFontSubstitution{U}{mathx}{m}{n}
\DeclareMathAccent{\widecheck}{0}{mathx}{"71}
\DeclareMathAccent{\wideparen}{0}{mathx}{"75}

\DeclareMathOperator{\End}{End}

\DeclareMathOperator{\ext}{ext}

\DeclareMathOperator{\HF}{\HF}

\DeclareMathOperator{\Hom}{Hom}

\DeclareMathOperator{\Tor}{Tor}

\DeclareMathOperator{\coker}{coker}

\DeclareMathOperator{\im}{im}
\DeclareMathOperator{\ind}{index}

\DeclareMathOperator{\res}{res}
\DeclareMathOperator{\rk}{rk}

\DeclareMathOperator{\spec}{spec}
\DeclareMathOperator{\supp}{supp}

\DeclareMathOperator{\str}{str}

\DeclarePairedDelimiter\floor{\lfloor}{\rfloor}
\DeclarePairedDelimiter\paren{\lparen}{\rparen}

\DeclarePairedDelimiter{\Abs}{\|}{\|}

\DeclarePairedDelimiter{\Inner}{\langle}{\rangle}

\DeclarePairedDelimiter{\abs}{\lvert}{\rvert}
\DeclarePairedDelimiter{\bracket}{\langle}{\rangle}

\DeclarePairedDelimiter{\set}{\lbrace}{\rbrace}
\def\({\left(}
\def\){\right)}
\def\<{\left\langle}
\def\>{\right\rangle}

\newcommand{\Cl}{\mathrm{C\ell}}
\newcommand{\C}{{\mathbf{C}}}

\newcommand{\LC}{\mathrm{LC}}

\newcommand{\N}{{\mathbf{N}}}

\newcommand{\PSL}{\P\SL}

\newcommand{\Ric}{\mathrm{Ric}}
\newcommand{\R}{\mathbf{R}}

\newcommand{\SL}{\mathrm{SL}}

\newcommand{\U}{\mathrm{U}}
\newcommand{\Vect}{\mathrm{Vect}}

\newcommand{\Z}{\mathbf{Z}}

\newcommand{\andq}{\text{and}\quad}

\newcommand{\co}{\mskip0.5mu\colon\thinspace}

\newcommand{\defined}[2][\key]{\def\key{#2}\textbf{#2}\index{#1}}

\newcommand{\del}{\partial}
\newcommand{\ev}{\mathrm{ev}}

\newcommand{\id}{\mathrm{id}}
\newcommand{\incl}{\hookrightarrow}
\newcommand{\emb}{\hookrightarrow}

\newcommand{\iso}{\cong}

\newcommand{\loc}{\mathrm{loc}}

\newcommand{\one}{\mathbf{1}}
\newcommand{\onto}{\twoheadrightarrow}

\newcommand{\pr}{\mathrm{pr}}
\newcommand{\qandq}{\quad\text{and}\quad}

\newcommand{\qand}{\quad\text{and}}

\newcommand{\qwithq}{\quad\text{with}\quad}

\newcommand{\vol}{\mathrm{vol}}

\renewcommand{\O}{\mathrm{O}}
\renewcommand{\P}{\mathbf{P}}

\renewcommand{\div}{\operatorname{div}}
\renewcommand{\emptyset}{\varnothing}
\renewcommand{\epsilon}{\varepsilon}
\renewcommand{\setminus}{{\backslash}}

\renewcommand{\leq}{\leqslant}
\renewcommand{\geq}{\geqslant}

\newcommand{\Wedge}{\Lambda}

\makeatletter
\renewcommand*\env@matrix[1][*\c@MaxMatrixCols c]{%
  \hskip -\arraycolsep
  \let\@ifnextchar\new@ifnextchar
  \array{#1}}

\renewcommand\xleftrightarrow[2][]{%
  \ext@arrow 9999{\longleftrightarrowfill@}{#1}{#2}}
\newcommand\longleftrightarrowfill@{%
  \arrowfill@\leftarrow\relbar\rightarrow}
\makeatother

\newcommand{\tw}{\mathrm{tw}}

%% file: preamble/letters.tex
\newcommand{\rd}{{\rm d}}

\newcommand{\rI}{{\rm I}}
\newcommand{\rII}{{\rm II}}

\newcommand{\bA}{{\mathbf{A}}}

\newcommand{\bH}{{\mathbf{H}}}

\newcommand{\bL}{{\mathbf{L}}}

\newcommand{\sF}{\mathscr{F}}

\newcommand{\sL}{\mathscr{L}}
\newcommand{\sM}{\mathscr{M}}

\newcommand{\sP}{\mathscr{P}}

\newcommand{\sS}{\mathscr{S}}

\newcommand{\fl}{{\mathfrak l}}

\newcommand{\fo}{{\mathfrak o}}

\newcommand{\ubR}{{\underline \R}}

%% file: Introduction.tex
\section{Introduction}
\label{Sec_Introduction}

Let $(X,g)$ be a closed Riemannian manifold of dimension $n$.

\begin{definition}[{cf.~\cites[Chapter II Definition 5.2]{Lawson1989}[§1(b)]{Bismut1989:LocalIndexNonKahler}}] %[Dirac bundle with skew torsion]
  % \label{Def_DiracBundleWithSkewTorsion}
  A \defined{Dirac bundle with skew torsion} on $(X,g)$ consists of:
  \begin{enumerate}
  \item
    a Euclidean vector bundle $S$ over $X$ equipped with a skew-adjoint Clifford multiplication $\gamma \co TX \to \fo(S)$;
    that is:
    \begin{equation*}
      \gamma(v)^2 = -\abs{v}^2 \, \one_S
    \end{equation*}
    for every $v \in TX$; and
  \item
    an orthogonal covariant derivative $\nabla$ on $S$ and a $3$--form $\Tor \in \Omega^3(X)$ such that $\gamma$ is parallel with respect to $\nabla$ and the orthogonal affine connection $\nabla^T$ on $(X,g)$ defined by
    \begin{equation*}
      \Inner{\nabla_u^Tv,w} = \Inner{\nabla_u^\LC v,w} + \tfrac12\Tor(u,v,w).
    \end{equation*}
    Here $\nabla^\LC$ denotes the Levi-Civita connection of $(X,g)$.
    \qedhere
  \end{enumerate}
\end{definition}

\begin{definition}%[Ramified Euclidean line bundle]
  \label{Def_RamifiedEuclideanLineBundle}
  A \defined{ramified Euclidean line bundle} over $X$ consists of:
  \begin{enumerate}
  \item
    a closed subset $Z \subset X$, the \defined{branching locus}, and
  \item
    a Euclidean line bundle $\fl$ over $X \setminus Z$
  \end{enumerate}
  such that
  \begin{enumerate}[resume]
  \item
    if $W \subset Z$ is closed and $\fl$ extends over $X \setminus W$,
    then $W = Z$.
    \qedhere
  \end{enumerate}
\end{definition}

This article is concerned with the analysis of the Dirac operator associated with a Dirac bundle with skew torsion $(S,\gamma,\nabla,\Tor)$ twisted by a ramified Euclidean line bundle $\fl$
\begin{equation*}
  D \co H^1\Gamma\paren{X\setminus Z, S\otimes\fl} \to L^2\Gamma\paren{X\setminus Z, S\otimes\fl}
\end{equation*}
and its extensions.

\smallskip

% MOTIVATION
The authors' motivation for this stems from the following.
Taubes has observed that the failure of compactness for a wide variety of generalised Seiberg--Witten equations---e.g.:
stable flat $\PSL_2(\C)$--connections in dimension three \cite{Taubes2012},
anti-self-dual $\SL_2(\C)$--connections in dimension four \cite{Taubes2013},
the Seiberg--Witten equation with multiple spinors \cite{Taubes2016:SeibergWittenCompactness},
the Vafa--Witten equation \cite{Taubes2017}, and
the Kapustin--Witten equation \cite{Taubes2022:NahmPoleKapustinWitten}%
---
leaves behind evidence in the form of a \defined{$\Z/2\Z$ harmonic spinor}.
The latter is a pair $(Z,\fl;\Phi)$ consisting of a ramified Euclidean line bundle $(Z,\fl)$ and a harmonic spinor $\Phi \in \ker D$.
$\Z/2\Z$ harmonic spinors also appear in \citeauthor{Donaldson2016}'s work on adiabatic limits of coassociative Kovalev--Lefschetz fibrations of $G_2$--manifolds \cite{Donaldson2016} and \citeauthor{He2022:DeformBranchedSLags}'s work on branched double covers of special Lagrangian submanifolds \cite{He2022:DeformBranchedSLags}.

In light of this, it is important to understand the universal moduli space of $\Z/2\Z$ harmonic spinors (allowing for $g$, $\gamma$, and $\nabla$ to vary).
The fundamental issue is that $D$ is only left semi-Fredholm (under mild assumptions; see \autoref{Hyp_BorderlineHardyInequality}), but not Fredholm---except in edge cases, e.g., if $Z = \emptyset$ or $n=2$.
The naive expectation is that the $\infty$--dimensional cokernel of $D$ can be compensated by wiggling the branching locus $Z$.
In his PhD thesis,
\citet{Takahashi2015,Takahashi2017} has made some initial progress in this direction.
\citet{Donaldson2021} and \citet{Parker2023:Deformation} have developed a (partial) deformation theory for $\Z/2\Z$ harmonic $1$--forms and $\Z/2\Z$ harmonic spinors on spin $3$--manifolds respectively.
There is work in progress by He, Parker and Walpuski to address this problem a bit more systematically.
The present article should be considered infrastructure for this project \cite{HeParkerWalpuski2026} (and, hopefully, other applications as well).

\smallskip

% OVERVIEW OF THE RESULTS
Here is a summary of the results contained in this article.
\autoref{Sec_GelfandRobbinQuotient} considers $D$ as an unbounded operator $\Dmin$ on  $L^2\Gamma\paren{X\setminus Z, S\otimes\fl}$, the \defined{minimal extension}, and systematically studies its closed extensions.
The adjoint $\Dmax \coloneq \Dmin^*$ is the \defined{maximal extension} of $\Dmin$.
The closed extensions $D_\ResidueCondition$ of $\Dmin$ are classified by \defined{residue conditions},
that is: closed subspaces $\ResidueCondition \subset \GelfandRobbinQuotient$ of the \defined{Gelfand--Robbin quotient}
\begin{equation*}
  \GelfandRobbinQuotient
  \coloneq
  \frac{\dom(\Dmax)}{\dom(\Dmin)}.
\end{equation*}
Moreover, $\GelfandRobbinQuotient$ is equipped with a symplectic form $G$, the \defined{Green's form}, which controls the formation of adjoints.
Within this framework it is also possible to describe which extensions $D_\ResidueCondition$ are Fredholm.
The entire discussion only relies on $\Dmin$ being closed, densely defined, and symmetric as well as left semi-Fredholm.
It is confined to the realm of abstract functional analysis and its purpose is to separate what is true for formal reasons from what is true for geometric reasons.
Most of the observations in \autoref{Sec_GelfandRobbinQuotient} can be found in \cites[Exercise 2.17]{McDuff1998}[§3]{BoossBavnbekFurutani1998:Maslov}[Appendix B]{Salamon2008}[Exercises 6.3.3 and 6.5.11]{BuhlerSalamon2018:FunctionalAnalysis} in some shape or form.

Assuming that $Z \subset X$ is a closed (cooriented) submanifold of codimension two,
\autoref{Sec_GeometricRealisation} constructs an isomorphism of symplectic Hilbert spaces
\begin{equation*}
  \res \co
  \paren{\GelfandRobbinQuotient,G}
  \iso  
  \paren{\check H\Gamma\paren{Z,\check S},\check\Omega},
\end{equation*}
the \defined{residue map}, between the Gelfand--Robbin quotient and a Hilbert space of sections of a symplectic vector bundle over $Z$.
The residue map extracts the leading order behavior of $\phi \in \dom(\Dmax)$ which is shown to be (at worst) comparable to $\bar z^{-1/2}$ transversely to $Z$.
With the help of the above it is possible to define spectral residue conditions,
analogous to the APS boundary condition \cite{Atiyah1975a},
as well as local residue conditions.
As a by-product this yields a variant of the bordism theorem,
whose significance remains somewhat mysterious to the authors.
Evidently,
the construction in \autoref{Sec_GeometricRealisation} is inspired by \citeauthor{BaerBallmann2012:BVP}'s magnificent article \cite{BaerBallmann2012:BVP} on boundary value problems for Dirac operators.

\autoref{Sec_RegularityTheory} develops an $L^2$ regularity theory on the scale of \defined{adapted Sobolev spaces} $\paren{H_a^{k+1}\Gamma\paren{X\setminus Z,S\otimes\fl}}_{k \in \N_0}$.
This scale is defined via conormal differential operators and the Dirac operator $D$.
It gives rise to a graded Fréchet space $H_a^\infty\Gamma\paren{X\setminus Z,S\otimes\fl}$ which is tame in the sense of \citet[Part II Definition 1.3.2]{Hamilton1982:NashMoser}---a prerequisite for using Nash--Moser theory.
It is proved that if a residue condition $\ResidueCondition \subset \check H\Gamma\paren{Z,\check S}$ is \defined{$\infty$--regular},
then the extension $D_\ResidueCondition$ satisfies a variant of elliptic regularity together with elliptic estimates.
For local residue conditions $\ResidueCondition$,
$\infty$--regularity can be verified using a symbolic criterion.
In particular, this criterion applies to the Lagrangian local residue condition which is secretly at the heart of \cite{Takahashi2015,Parker2023:Deformation}.
It is quite plausible that these results can be cobbled together using the powerful machines developed by \citet{Mazzeo1991,MazzeoVertman2014,AlbinGellRedman2016,AlbinGellRedman2023}.
However, the arguments in \autoref{Sec_RegularityTheory} are almost elementary and there should be some value in that.

\smallskip

It should be possible, with suitable modifications, to extend the work in the present article to higher rank ramified Euclidean local systems;
in particular: to flat Hermitian line bundles.
In fact, Ammann--Große have work in progress in this direction and some instances of this appear in \citeauthor{PortmannSokSolovej2018:SelfAdjoint}'s work on magnetic links \cite{PortmannSokSolovej2018:SelfAdjoint,PortmannSokSolovej2018:ZeroModes,PortmannSokSolovej2020}.

\paragraph{Acknowledgements.}
The authors thank Siqi He and Greg Parker for helpful discussions,
Jacek Rzemieniecki and Thibault Langlais for their careful proofreading of earlier versions of this article, and
Bernd Ammann for discussions regarding his joint work with Große and pointing out the work of \citeauthor{PortmannSokSolovej2018:SelfAdjoint}.
TW is indebted to Dietmar Salamon for teaching him about the Gelfand--Robbin quotient almost two decades ago.

This material is based in part upon work carried out while the authors were in residence at the Simons Laufer Mathematical Sciences Institute (previously known as MSRI), Berkeley, California, during the Fall 2022 semester.

\paragraph{Conventions.}
\emph{Choose} a cut-off function $\chi \in C^\infty([0,\infty),[0,1])$ with $\chi|_{[0,1/2]} = 1$ and $\supp(\chi) \subset [0,3/4)$.
The notation $x \lesssim_p y$ means the inequality $x \leq c(p) y$ with the constant $c(p)$ depending on the parameter(s) $p$.
The \defined{bracket} $\bracket{-} \co \R \to [1,\infty)$ is defined by $\bracket{x} \coloneq \paren{1+x^2}^{1/2}$.

%%% Local Variables:
%%% mode: latex
%%% TeX-master: "DiracOperatorsTwistedByRamifiedLineBundles"
%%% End:

%% file: GelfandRobbinQuotientAbstract.tex
\section{The Gelfand--Robbin quotient, I: abstract theory}
\label{Sec_GelfandRobbinQuotient}

This section studies the closed extensions of the Dirac operator $D \co H^1\Gamma\paren{X\setminus Z, S\otimes\fl} \to L^2\Gamma\paren{X\setminus Z, S\otimes\fl}$,
considered as an unbounded operator,
following
\cites[Exercise 2.17]{McDuff1998}[§3]{BoossBavnbekFurutani1998:Maslov}[Appendix B]{Salamon2008}[Exercises 6.3.3 and 6.5.11]{BuhlerSalamon2018:FunctionalAnalysis}.
Throughout,
assume the following analytic condition on the branching locus $Z$.

\begin{hypothesis}
  \label{Hyp_BorderlineHardyInequality}
  There is an $r \in C^\infty(X\setminus Z,(0,\infty))$, uniformly comparable to the Riemannian distance to $Z$, such that 
  the following \defined{borderline Hardy inequality} holds:
  for every $\phi \in H^1\Gamma\paren{X\setminus Z, S\otimes\fl}$,
  $r^{-1}\phi \in L^2\Gamma\paren{X\setminus Z, S\otimes\fl}$ and
  \begin{equation*}
    \Abs{r^{-1}\phi}_{L^2} \lesssim \Abs{\phi}_{H^1}.
  \end{equation*}
\end{hypothesis}

\begin{remark}
  \label{Rmk_BorderlineHardyInequality}
  \autoref{Hyp_BorderlineHardyInequality} holds if $Z$ is a codimension two submanifold;
  see \citet[Lemma 2.6]{Takahashi2015} or \autoref{Lem_BorderlineHardyInequality}.
  Moreover, it holds in the situation considered by \citet{HaydysTakahashiMazzeo2023:Index} where $Z \subset X$ is a graph embedded in a $3$--manifold.
\end{remark}

\input{MinimalAndMaximalExtension}
\input{ClosedExtensionsAndBranchingConditions}
\input{GreensFormAndAdjointExtensions}
\input{FredholmExtensions}
\input{Chirality}

%%% Local Variables:
%%% mode: latex
%%% TeX-master: "DiracOperatorsTwistedByRamifiedLineBundles"
%%% ispell-local-dictionary: "british"
%%% End:

%% file: MinimalAndMaximalExtension.tex
\subsection{The minimal and maximal extension}
\label{Sec_MinimalExtension}

\begin{prop}
  \label{Prop_D_LeftSemiFredholm}
  The bounded operator
  $D \co H^1\Gamma\paren{X\setminus Z, S\otimes\fl} \to L^2\Gamma\paren{X\setminus Z, S\otimes\fl}$
  is left semi-Fredholm;
  that is: $\ker D$ is finite-dimensional and $\im D$ is closed.
\end{prop}

The proof relies on the following consequences of the borderline Hardy inequality.

\begin{lemma}
  \label{Lem_BorderlineHardyInequality_Consequences}
  The following hold:
  \begin{enumerate}
  \item
    \label{Lem_BorderlineHardyInequality_Consequences_H1=H10}
    $H^1\Gamma\paren{X\setminus Z, S\otimes \fl} = H_0^1\Gamma\paren{X\setminus Z, S\otimes \fl}$.
  \item
    \label{Lem_BorderlineHardyInequality_Consequences_H1L2Compact}
    The inclusion $H^1\Gamma\paren{X\setminus Z, S\otimes \fl} \incl L^2\Gamma\paren{X\setminus Z, S\otimes \fl}$ is a compact operator.
  \end{enumerate}
\end{lemma}

\begin{proof}%[Proof of \autoref{Lem_BorderlineHardyInequality_Consequences}]
  For $\epsilon  > 0$ set $\chi_\epsilon \coloneq \chi(r/\epsilon)$.
  Let $\phi \in H^1\Gamma\paren{X\setminus Z, S\otimes \fl}$.
  Since $\abs{r\rd\chi_\epsilon} \lesssim 1$,
  \begin{equation*}
    \Abs{\nabla(\chi_\epsilon\phi)}_{L^2}
    \leq
    \Abs{(r\rd\chi_\epsilon) r^{-1}\phi}_{L^2}
    +
    \Abs{\chi_\epsilon \nabla \phi}_{L^2}
    \lesssim
    \paren*{\int_{\supp(\rd\chi_\epsilon)} \abs{r^{-1}\phi}^2 + \abs{\nabla\phi}^2}^{1/2}.
  \end{equation*}
  Therefore, by \autoref{Hyp_BorderlineHardyInequality} and monotone convergence,
  \begin{equation*}
    \lim_{\epsilon \downarrow 0} \Abs{\nabla(\chi_\epsilon\phi)}_{L^2} = 0.
  \end{equation*}
  This implies \autoref{Lem_BorderlineHardyInequality_Consequences_H1=H10}.

  Let $(\phi_n) \in H^1\Gamma\paren{X\setminus Z, S\otimes \fl}^\N$ with $\Abs{\phi_n}_{H^1} = 1$ for every $n \in \N$.
  For every $\epsilon > 0$, a subsequence of $((1-\chi_\epsilon)\phi_n)$ converges in $L^2\Gamma\paren{X\setminus Z, S\otimes \fl}$.
  By the borderline Hardy inequality,
  \begin{equation*}
    \Abs{\chi_\epsilon\phi_n}_{L^2} \lesssim \epsilon.
  \end{equation*}
  Therefore, \autoref{Lem_BorderlineHardyInequality_Consequences_H1L2Compact} follows from a diagonal sequence argument.
\end{proof}

The proof of \autoref{Prop_D_LeftSemiFredholm} also uses the following observation.

\begin{prop}[{cf.~\citet[Theorem 1.10]{Bismut1989:LocalIndexNonKahler}}]
  \label{Prop_DiracOperatorWithSkewTorsion}
  ~
  \begin{enumerate}
  \item
    \label{Prop_DiracOperatorWithSkewTorsion_SelfAdjoint}
    $D$ is formally self-adjoint;
    in fact:
    for every $\phi,\psi \in H_\loc^1\Gamma(X,S \otimes \fl)$
    \begin{equation*}
      \Inner{D\phi,\psi} - \Inner{\phi,D \psi}
      =
      \div(v)
      \qwithq
      v
      \coloneq
      \sum_{i=1}^n \Inner{\gamma(e_i)\phi,\psi} e_i.
    \end{equation*}
    Here, and throughout this article, $\paren{e_1,\ldots,e_n}$ denotes a local orthonormal frame; and
    $\div(v)$ denotes the divergence of $v$.
  \item
    \label{Prop_DiracOperatorWithSkewTorsion_SchrödingerLichnerowicz}
    $D$ satisfies
    \begin{equation*}
      D^2 = \nabla^*\nabla + \tau \nabla +  \gamma(F_\nabla)
    \end{equation*}
    with $\tau \in \Gamma\paren{X,\Hom(T^*X\otimes S, S)}$ depending linearly on $\Tor$.
  \end{enumerate}
\end{prop}

\begin{proof}%[Proof of \autoref{Prop_DiracOperatorWithSkewTorsion_SchrödingerLichnerowicz}]
  The following argument can be found in \cite[Proof of Theorem 1.10]{Bismut1989:LocalIndexNonKahler} and is repeated here only for the reader's convenience.

  By direct computation,
  \begin{equation*}
    \div v
    =
    \sum_{i=1}^n \sL_{e_i} \Inner{v,e_i}
    =
    \Inner{D\phi,\psi} - \Inner{\phi,D\psi}
    +
    \sum_{i=1}^n \Inner{\gamma\paren{\nabla_{e_i}^Te_i}\phi,\psi}
  \end{equation*}
  and
  \begin{equation*}
    \Inner{\nabla_{e_i}^Te_i,-}
    =
    \tfrac12\Tor(e_i,e_i,-) = 0.
  \end{equation*}  
  This proves \autoref{Prop_DiracOperatorWithSkewTorsion_SelfAdjoint}.

  By direct computation,
  \begin{align*}
     D^2 = \sum_{i,j=1}^n \gamma(e_i){\nabla}_{e_i} \gamma(e_j){\nabla}_{e_j}
    &=
      \sum_{i,j=1}^n \gamma(e_j)\gamma(e_j) {\nabla}_{e_i} {\nabla}_{e_j}
      +
      \gamma(e_i)\gamma({\nabla}_{e_i}^Te_j) {\nabla}_{e_j} \\
    &=
      \nabla^*\nabla +  \gamma(F_{ \nabla})
      +
      \sum_{i,j=1}^n \gamma(e_i)\gamma({\nabla}_{e_i}^Te_j) {\nabla}_{e_j}
  \end{align*}
  and
  \begin{equation*}
    \Inner{\nabla_{e_i}^Te_j,e_k}
    =    
    \tfrac12\Tor(e_i,e_j,e_k).
  \end{equation*}
  This proves \autoref{Prop_DiracOperatorWithSkewTorsion_SchrödingerLichnerowicz}.
\end{proof}

\begin{proof}[Proof of \autoref{Prop_D_LeftSemiFredholm}]
  By \autoref{Lem_BorderlineHardyInequality_Consequences}~\autoref{Lem_BorderlineHardyInequality_Consequences_H1=H10} and \autoref{Prop_DiracOperatorWithSkewTorsion},
  \begin{equation}
    \label{Eq_D_EllipticEstimate}
    \Abs{\phi}_{H^1}
    \lesssim
    \Abs{D\phi}_{L^2} + \Abs{\phi}_{L^2}
  \end{equation}
  for every $\phi \in H^1\Gamma\paren{X\setminus Z, S\otimes \fl}$.
  Therefore and by \autoref{Lem_BorderlineHardyInequality_Consequences}~\autoref{Lem_BorderlineHardyInequality_Consequences_H1L2Compact},
  $D$ is left semi-Fredholm.
\end{proof}

With the exception of a few edge cases---e.g.: if $Z = \emptyset$ or $Z \subset X$ is a finite subset of a surface \cites[§3.4.2]{Doan2024}[§4.1--4.5]{HaydysTakahashiMazzeo2023:Index}---the operator $D \co H^1\Gamma\paren{X\setminus Z, S\otimes\fl} \to L^2\Gamma\paren{X\setminus Z, S\otimes\fl}$ is not Fredholm:
its cokernel is $\infty$--dimensional.
Therefore, it is useful to consider $D$ as an unbounded operator and systematically study its closed extensions;
cf.~\cite[Chapter 6]{BuhlerSalamon2018:FunctionalAnalysis}.

\begin{definition}
  \label{Def_Dmin}
  The \defined{minimal extension}
  \begin{align*}
    \Dmin \co
    \dom(\Dmin) \to L^2\Gamma\paren{X\setminus Z, S\otimes\fl};
  \end{align*}
  is the operator  $D \co H^1\Gamma\paren{X\setminus Z, S\otimes\fl} \to L^2\Gamma\paren{X\setminus Z, S\otimes\fl}$
  considered as unbounded operator on $L^2\Gamma\paren{X\setminus Z, S\otimes\fl}$.
\end{definition}

The following is crucial for the rest of this section.

\begin{prop}  
  \label{Prop_DMin_ClosedDenselyDefinedSymmetric}
  $\Dmin$ is closed, densely defined, and symmetric.
\end{prop}

\begin{proof}%[Proof of \autoref{Prop_DMin}]
  % densely defined
  Evidently, $\Dmin$ is densely defined.
  % 
  % symmetric
  By \autoref{Prop_DiracOperatorWithSkewTorsion}~\autoref{Prop_DiracOperatorWithSkewTorsion_SelfAdjoint},
  \begin{equation*}
    \Inner{D \phi,\psi} = \Inner{\phi,D \psi}
  \end{equation*}
  for every $\phi,\psi \in \Gamma_c\paren{X\setminus Z, S\otimes\fl}$.
  Therefore, by \autoref{Lem_BorderlineHardyInequality_Consequences}~\autoref{Lem_BorderlineHardyInequality_Consequences_H1=H10} and since $\Gamma_c\paren{X\setminus Z, S\otimes\fl} \subset H_0^1\Gamma\paren{X\setminus Z, S\otimes\fl}$ is dense, $\Dmin$ is symmetric.
  %
  % closed
  By \autoref{Eq_D_EllipticEstimate}
  the Sobolev norm
  $\Abs{-}_{H^1}$
  and the graph norm
  $\Abs{-}_D \coloneq \paren{\Abs{-}_{L^2}^2 + \Abs{D-}_{L^2}^2}^{1/2}$
  are equivalent.
  Therefore and since $H^1\Gamma\paren{X\setminus Z, S\otimes\fl}$ is complete,
  $\Dmin$ is closed.
\end{proof}

\begin{definition}
  \label{Def_DMax}
  The \defined{maximal extension}
  \begin{equation*}
    \Dmax \co \dom(\Dmax) \to  L^2\Gamma\paren{X\setminus Z, S\otimes\fl}
  \end{equation*}
  is the adjoint of $\Dmin$ in the sense of unbounded operators;
  that is:
  \begin{equation*}
    \dom(D_{\max})
    \coloneq
    \set*{
      \phi \in L^2\Gamma\paren{X\setminus Z, S\otimes\fl}
      :
      \Inner{\phi,\Dmin-}_{L^2} \co \dom(\Dmin) \to \R ~\text{is $\Abs{-}_{L^2}$--bounded}
    }    
  \end{equation*}
  and
  for every $\phi \in \dom(\Dmax)$ and $\psi \in \dom(\Dmin)$
  \begin{equation*}
    \Inner{\Dmax\phi,\psi} = \Inner{\phi,\Dmin\psi}.
  \end{equation*}
  $\Dmax \phi$ exists by the Hahn--Banach Theorem and the Riesz Representation Theorem, and is unique because $\dom(\Dmin)$ is dense.
\end{definition}

\begin{remark}
  \label{Rmk_Dmax}
  It is convenient to consider
  $D \co H_\loc^1\Gamma\paren{X\setminus Z, S\otimes\fl} \to L_\loc^2\Gamma\paren{X\setminus Z, S\otimes\fl}$.
  From this perspective,
  \begin{equation*}
    \dom(\Dmax) = \set{ \phi \in H_\loc^1\Gamma\paren{X\setminus Z, S\otimes\fl} : \phi,D\phi \in L^2\Gamma\paren{X\setminus Z, S\otimes\fl} };
  \end{equation*}
  and it is excusable to drop the subscripts from $\Dmin\phi$, $\Dmax \phi$, etc.
\end{remark}

%%% Local Variables:
%%% mode: latex
%%% TeX-master: "DiracOperatorsTwistedByRamifiedLineBundles"
%%% ispell-local-dictionary: "british"
%%% End:

%% file: ClosedExtensionsAndBranchingConditions.tex
\subsection{Closed extensions and residue conditions}
\label{Sec_ClosedExtensionsResidueCondition}

The closed extensions of $\Dmin$ can be systematically understood as follows.

\begin{definition}
  \label{Def_GelfandRobbinQuotient}
  The \defined{Gelfand--Robbin quotient} is the Hilbert space
  \begin{equation*}
    \GelfandRobbinQuotient
    \coloneq
    \frac{\dom(\Dmax)}{\dom(\Dmin)}.
  \end{equation*}
  Since $\Dmin$ is closed,
  $\dom(\Dmin) \subset \dom(\Dmax)$ is a $\Abs{-}_D$--closed subspace.
  Denote the canonical projection map by
  \begin{equation*}
    [-] \co \dom(\Dmax) \to \GelfandRobbinQuotient.
    \qedhere
  \end{equation*}
\end{definition}

\begin{remark}
  \label{Rmk_GelfandRobbinQuotient}
  $\GelfandRobbinQuotient$ is localised on $Z$ in the following sense:
  $[\phi] = [\chi(r/\epsilon) \phi]$ for every $\epsilon > 0$ and $\phi \in \dom(\Dmax)$.
\end{remark}

\begin{definition}
  \label{Def_ResidueCondition}
  A \defined{residue condition} is a closed subspace
  \begin{equation*}
    \ResidueCondition \subset \GelfandRobbinQuotient.
    \qedhere
  \end{equation*}
\end{definition}

\begin{prop}[{closed extension=residue condition; cf.~\cite[Lemma 3.3~(a)]{BoossBavnbekFurutani1998:Maslov}}]
  \label{Prop_ClosedExtension}
  If $\ResidueCondition \subset \GelfandRobbinQuotient$ is a residue condition,
  then
  \begin{equation*}
    D_\ResidueCondition
    \coloneq
    \Dmax|_{\dom(D_\ResidueCondition)}
    \qwithq
    \dom(D_\ResidueCondition)
    \coloneq
    [-]^{-1}(\ResidueCondition)
  \end{equation*}
  is a closed extension of $\Dmin$.
  Moreover,
  every closed extension of $\Dmin$ is of this form.
\end{prop}

\begin{proof}%[Proof of \autoref{Prop_ClosedExtension}]
  Let $\ResidueCondition \subset \GelfandRobbinQuotient$ be a residue condition.
  The canonical projection $[-] \co \dom(\Dmax) \to \GelfandRobbinQuotient$ is bounded.
  Therefore, $\dom(D_\ResidueCondition) \coloneq [-]^{-1}(\ResidueCondition) \subset \dom(\Dmax)$ is a $\Abs{-}_D$--closed subspace;
  hence: $D_\ResidueCondition$ is closed.

  Let $\bar D$ be a closed extension of $\Dmin$.
  Since $\dom(\bar D) \subset \dom(\Dmax)$ is a $\Abs{-}_D$--closed subspace,
  $\ResidueCondition \coloneq [\dom(\bar D)] = \frac{\dom(\bar D)}{\dom(\Dmin)} \subset \GelfandRobbinQuotient$ is a closed subspace.
  Since $\dom(\Dmin) \subset \dom(\bar D)$,
  $\dom(\bar D) = [-]^{-1}(\ResidueCondition)$;
  hence: $\bar D = D_\ResidueCondition$.
\end{proof}

%%% Local Variables:
%%% mode: latex
%%% TeX-master: "DiracOperatorsTwistedByRamifiedLineBundles"
%%% ispell-local-dictionary: "british"
%%% End:

%% file: GreensFormAndAdjointExtensions.tex
\subsection{The Green's form and adjoint extensions}
\label{Sec_GreensFormAndAdjointExtensions}

The Gelfand--Robbin quotient carries a symplectic structure related to the construction of adjoint extensions.

\begin{definition}
  \label{Def_GreensForm}
  The \defined{Green's form} $G \in \Hom\paren{\Wedge^2 \GelfandRobbinQuotient,\R}$ is defined by
  \begin{equation*}
    G([\phi] \wedge [\psi])
    \coloneq
    \Inner{D\phi,\psi}_{L^2} - \Inner{\phi,D\psi}_{L^2}.
    \qedhere
  \end{equation*}
\end{definition}

\begin{prop}[{cf.~\cites[Lemma 3.1, Proposition 3.2]{BoossBavnbekFurutani1998:Maslov}[Remark B.1~(ii)]{Salamon2008}}]
  \label{Prop_GreensForm_Symplectic}
  $G$ is symplectic;
  that is:
  it induces a Hilbert space isomorphism
  \begin{align*}
    J \co \GelfandRobbinQuotient &\to \DualGelfandRobbinQuotient \coloneq \sL(\GelfandRobbinQuotient,\R) \\
    [\phi] &\mapsto G([\phi] \wedge -).
  \end{align*}
  Moreover:
  if $\sharp \co \DualGelfandRobbinQuotient \to \GelfandRobbinQuotient$ denotes the isomorphism induced by the inner product,
  then $\sharp \circ J$ is an isometric complex structure on $\GelfandRobbinQuotient$.
\end{prop}

\begin{proof}%[Proof of \autoref{Prop_GreensForm_Symplectic}]
  The canonical projection induces an isometry
  $[-] \co \dom(\Dmin)^{\perp_D} \iso \GelfandRobbinQuotient$.
  Here $\perp_D$ indicates the orthogonal complement with respect to the graph inner product
  \begin{equation*}
    \Inner{-,-}_D \coloneq \Inner{-,-}_{L^2} + \Inner{D-,D-}_{L^2}.
  \end{equation*}
  By direct inspection,
  \begin{align*}
    \dom(\Dmin)^{\perp_D}
    &=
      \set*{
      \phi \in \dom(\Dmax)
      :
      \Inner{\phi,\psi}_{L^2} + \Inner{D\phi,D\psi}_{L^2}
      =
      0
      ~\textnormal{for every}~
      \psi \in \dom(\Dmin)
      } \\
    &=
      \set*{
      \phi \in \dom(\Dmax)
      :
      D\phi \in \dom(\Dmax)
      ~\text{and}~
      D^2\phi = -\phi      
      }.
  \end{align*}
  Therefore,
  $D$ induces an isometric complex structure $D \co \dom(\Dmin)^{\perp_D} \to \dom(\Dmin)^{\perp_D}$.

  The diagram
  \begin{equation*}
    \begin{tikzcd}
      \dom(\Dmin)^{\perp_D} \ar{r}{D} \ar{d}[swap]{[-]} & \dom(\Dmin)^{\perp_D} \ar{d}{[-]} \\
      \GelfandRobbinQuotient \ar{r}[swap]{\sharp \circ J} & \GelfandRobbinQuotient
    \end{tikzcd}
  \end{equation*}
  commutes because
  for every $\phi,\psi \in \dom(\Dmin)^{\perp_D}$
  \begin{equation*}
    G([\phi] \wedge [\psi])
    = \Inner{D\phi,\psi}_{L^2} - \Inner{\phi,D\psi}_{L^2}
    = \Inner{D\phi,\psi}_{L^2} + \Inner{D^2\phi,D\psi}_{L^2}
    = \Inner{D\phi,\psi}_{D}.
  \end{equation*}  
  This proves the assertion. 
  \qedhere
\end{proof}

\begin{prop}[{cf.~\cite[Lemma 3.3~(b)]{BoossBavnbekFurutani1998:Maslov}}]
  \label{Prop_AdjointExtension}
  Let $\ResidueCondition \subset \GelfandRobbinQuotient$ be a residue condition.  
  The adjoint $D_\ResidueCondition^*$ of $D_\ResidueCondition$ is the closed extension $D_{\ResidueCondition^G}$ associated with
  the symplectic complement
  \begin{equation*}
    \ResidueCondition^G
    \coloneq
    \set*{
      [\phi] \in \GelfandRobbinQuotient :
      G([\phi]\wedge[\psi]) = 0 \textnormal{ for every } [\psi] \in \ResidueCondition
    }.
  \end{equation*}  
  In particular, $D_\ResidueCondition$ is self-adjoint if and only if $\ResidueCondition$ is Lagrangian.
\end{prop}

\begin{proof}
  A moment's thought shows that
  \begin{align*}
    \dom(D_\ResidueCondition^*)
    &
      =
      \set{
      \phi \in \dom(\Dmax)
      :
      \Inner{D\phi,\psi}_{L^2} = \Inner{\phi,D\psi}_{L^2}
      \text{ for every } \psi \in \dom(D_\ResidueCondition)
      } \\
    &
      =
      \set{
      \phi \in \dom(\Dmax)
      :
      G([\phi] \wedge [\psi]) = 0
      \textnormal{ for every } [\psi] \in \ResidueCondition
      }.
  \end{align*}
  This proves the assertion.
\end{proof}

\begin{example}
  \label{Ex_CalderonSubspace_Lagrangian}
  The \defined{Calderón subspace}
  \begin{equation*}
    \Lambda
    \coloneq
    [\ker \Dmax]
    \subset
    \GelfandRobbinQuotient
  \end{equation*}
  is a Lagrangian residue condition.  
  Indeed, $\Lambda \subset \Lambda^G$ because 
  $G([\phi] \wedge [\psi])
  = \Inner{D\phi,\psi}_{L^2} - \Inner{\phi,D\psi}_{L^2}
  = 0$
  for every $\phi,\psi \in \ker \Dmax$.
  Moreover,
  if $[\phi] \in \Lambda^G$,
  then $D\phi \perp_{L^2} \ker \Dmax = \paren{\im \Dmin}^{\perp_{L^2}}$;
  therefore, there is a $\psi \in \dom(\Dmin)$ with $D\psi = D\phi$;
  hence: $[\phi] = [\phi-\psi] \in \Lambda$.
  With respect to the isomorphism $[-] \co \dom(\Dmin)^{\perp_D} \iso \GelfandRobbinQuotient$
  \begin{equation*}
    \Lambda \iso \set{ \phi \in \dom(\Dmin)^{\perp_D} : D\phi \in \im \Dmin}.
    \qedhere
  \end{equation*}
\end{example}

\begin{example}
  \label{Ex_CalderonSubspacePerp_Lagrangian}
  As a consequence of \autoref{Prop_GreensForm_Symplectic},
  the orthogonal complement of the Calderón subspace
  \begin{equation*}
    \Lambda^\perp = \sharp \circ J(\Lambda) \subset \GelfandRobbinQuotient
  \end{equation*}
  is a Lagrangian residue condition;
  moreover:
  \begin{equation*}
    \GelfandRobbinQuotient = \Lambda \oplus \Lambda^\perp.
  \end{equation*}
  With respect to the isomorphism $[-] \co \dom(\Dmin)^{\perp_D} \iso \GelfandRobbinQuotient$
  \begin{equation*}
    \Lambda^\perp \iso \set*{ \phi \in \dom(\Dmin)^{\perp_D} : \phi \in \im \Dmin}.
    \qedhere
  \end{equation*}
\end{example}

\begin{example}
  \label{Ex_AbstractMITBag_MutuallyAdjoint}
  Suppose that $S$ carries a parallel orthogonal complex structure $i$ which commutes with $\gamma$;
  that is: $(S,\gamma,\nabla,\Tor)$ is a \defined{complex Dirac bundle with skew torsion}.
  Evidently, $D$ is complex linear and $\GelfandRobbinQuotient$ inherits $i$ as an isometric complex structure $i$.
  This induces an orthogonal decomposition
  \begin{equation*}
    \GelfandRobbinQuotient = \AbstractBag_+ \oplus \AbstractBag_-
    \qwithq
    \AbstractBag_\pm
    \coloneq
    \set*{
      [\phi] \in \GelfandRobbinQuotient
      :
      \sharp \circ J[\phi] = \pm i[\phi]
    }.
  \end{equation*}  
  Since $i$ and $\sharp\circ J$ commute,
  $\AbstractBag_\pm \subset \GelfandRobbinQuotient$ are complex subspaces and,
  therefore, $\AbstractBag_\pm$ are mutually adjoint:
  \begin{equation*}
    \AbstractBag_\pm^G = \AbstractBag_\mp.
    \qedhere
  \end{equation*}
\end{example}

\begin{remark}[defect indices]
  \label{Rmk_DefectIndices}
  If $H$ is a \emph{complex} Hilbert space and $A \co \dom(A) \to H$ is a closed and symmetric unbounded \emph{complex} linear operator,
  then its closed \emph{complex linear} extensions traditionally are studied via \citeauthor{vonNeumann1930}'s theory of \defined{defect subspaces} and \defined{defect indices} \cites[Kapitel VII]{vonNeumann1930}[§X.1]{ReedSimon1975:MMP2}.
  The defect subspaces of $A$ are $\ker(A^* \mp i)$ and its defect indices are $n_\pm \coloneq \dim \ker(A^* \mp i)$.
  The maximal domain orthogonally decomposes as
  \begin{equation*}
    \dom(A^*) = \dom(A) \oplus \ker(A^* - i) \oplus \ker(A^* + i)
  \end{equation*}
  with respect to the graph inner product.
  Therefore,
  \begin{equation*}
    \GelfandRobbinQuotient \coloneq \frac{\dom(A^*)}{\dom(A)} \iso \ker(A^* - i) \oplus \ker(A^* + i).
  \end{equation*}
  In particular,
  closed self-adjoint complex linear extensions of $A$ correspond to closed \emph{complex} Lagrangian subspaces $\ResidueCondition \subset \ker(A^* - i) \oplus \ker(A^* + i)$.
  The latter exist if and only if $n_+ = n_-$.
  Of course, by Zorn's Lemma,
  $\GelfandRobbinQuotient$ always has a (real) Lagrangian subspace.
\end{remark}

\begin{prop}[Spectral theory]
  \label{Prop_SpectralTheory}
  Let $\ResidueCondition \subset \GelfandRobbinQuotient$ be a residue condition.
  If $\ResidueCondition \subset \GelfandRobbinQuotient$ is a Lagrangian and $\dom(D_\ResidueCondition) \incl L^2\Gamma\paren{X\setminus Z, S\otimes\fl}$ is compact,
  then $\spec(D_\ResidueCondition)$ consists only of point spectrum, is contained in $\R$ and discrete, and for every $\lambda \in \spec(D_\ResidueCondition)$ the eigenspace $\ker \paren{D_\ResidueCondition - \lambda\one}$ is finite-dimensional;
  moreover:
  $L^2\Gamma\paren{X\setminus Z, S\otimes\fl}$ decomposes as a (Hilbert space) direct sum
  \begin{equation*}
    L^2\Gamma\paren{X\setminus Z, S\otimes\fl}
    =
    \bigoplus_{\lambda \in \spec(D_\ResidueCondition)} \ker\paren{D_\ResidueCondition - \lambda \one}.
  \end{equation*}  
\end{prop}

\begin{proof}%[Proof of \autoref{Prop_SpectralTheory}]  
  By assumption,
  $D_\ResidueCondition$ is self-adjoint and has compact resolvent.
  The assertion, therefore, follows from the spectral theory of such operators;
  see, e.g., \cite[Theorem 6.3.13]{BuhlerSalamon2018:FunctionalAnalysis}.
\end{proof}

%%% Local Variables:
%%% mode: latex
%%% TeX-master: "DiracOperatorsTwistedByRamifiedLineBundles"
%%% ispell-local-dictionary: "british"
%%% End:

%% file: FredholmExtensions.tex
\subsection{Fredholm extensions}
\label{Sec_FredholmExtension}

The following characterises residue conditions $\ResidueCondition \subset \GelfandRobbinQuotient$ which correspond to Fredholm extensions $D_\ResidueCondition$ in terms of the relation between $\ResidueCondition$ and the Calderón subspace $\Lambda$.

\begin{definition}
  \label{Def_Delta}
  Let $\ResidueCondition \subset \GelfandRobbinQuotient$ be a residue condition.
  Denote by
  \begin{equation*}
    \delta_\ResidueCondition \co \Lambda \to \GelfandRobbinQuotient/\ResidueCondition
    \qandq
    \delta^\ResidueCondition \co \ResidueCondition \to \GelfandRobbinQuotient/\Lambda
  \end{equation*}
  the compositions of the canonical inclusions and projections.  
\end{definition}

\begin{prop}[{cf.~\cite[Lemma B.3]{Salamon2008}}]
  \label{Prop_FredholmResidueCondition}
  Let $\ResidueCondition \subset \GelfandRobbinQuotient$ be a residue condition.
  The closed extension $D_\ResidueCondition$ is Fredholm if and only if $\delta_\ResidueCondition$ is Fredholm if and only if $\delta^\ResidueCondition$ is Fredholm;
  moreover:
  \begin{equation*}
    \ind D_\ResidueCondition = \ind \delta_\ResidueCondition = \ind \delta^\ResidueCondition.
  \end{equation*}
\end{prop}

The proof relies on the following observation.

\begin{lemma}
  \label{Lem_FredholmResidueCondition}
  For every residue condition $\ResidueCondition \subset \GelfandRobbinQuotient$,
  there are short exact sequences
  \begin{equation*}
    \ker D_{\min}
    \incl \ker D_\ResidueCondition
    \onto \ker \delta_\ResidueCondition
    \qandq
    \coker \delta_\ResidueCondition
    \incl \coker D_\ResidueCondition
    \onto \coker D_{\max}.
  \end{equation*}
\end{lemma}

\begin{proof}%[Proof of \autoref{Lem_FredholmResidueCondition}]
  The Snake Lemma applied to
  \begin{equation*}
    \begin{tikzcd}
      \ker \Dmin \ar[equal]{r} \ar[hook]{d} & \ker \Dmin \ar[hook]{d} \\
       \ker D_\ResidueCondition \ar[hook]{r} & \ker \Dmax % \ar{d}
      \ar{r}{\delta_\ResidueCondition \circ [-]} & \GelfandRobbinQuotient/\ResidueCondition
    \end{tikzcd}
  \end{equation*}
  yields an exact sequence
  \begin{equation*}
    \frac{\ker D_\ResidueCondition}{\ker \Dmin}
    \incl
    \frac{\ker \Dmax}{\ker \Dmin} \iso \Lambda
    \xrightarrow{\delta_\ResidueCondition} \GelfandRobbinQuotient/\ResidueCondition.
  \end{equation*}
  This induces the first short exact sequence.
  
  The Snake Lemma applied to
  \begin{equation*}
    \begin{tikzcd}
      \dom(D_\ResidueCondition) \ar[hook]{r} \ar{d}{D_\ResidueCondition} & \dom(\Dmax) \ar[two heads]{r} \ar{d}{\Dmax} & \GelfandRobbinQuotient/\ResidueCondition \\
      L^2\Gamma\paren{X\setminus Z,S\otimes\fl} \ar[equals] {r}& L^2\Gamma\paren{X\setminus Z,S\otimes\fl}
    \end{tikzcd}
  \end{equation*}
  yields an exact sequence
  \begin{equation*}
    \ker D_\ResidueCondition
    \incl \ker \Dmax
    \xrightarrow{\delta_\ResidueCondition \circ [-]} \GelfandRobbinQuotient/\ResidueCondition
    \to \coker D_\ResidueCondition
    \onto \coker \Dmax.
  \end{equation*}
  Since $\coker \delta_\ResidueCondition \circ [-] = \coker \delta_\ResidueCondition$,
  this induces the second short exact sequence.
\end{proof}

\begin{proof}[Proof of \autoref{Prop_FredholmResidueCondition}]  
  A moment's thought shows that
  \begin{equation*}
    \ker \delta_\ResidueCondition = \Lambda \cap \ResidueCondition = \ker \delta^\ResidueCondition
    \qandq
    \coker \delta_\ResidueCondition \iso \frac{\GelfandRobbinQuotient}{\Lambda+\ResidueCondition} \iso \coker \delta^\ResidueCondition.
  \end{equation*}
  By \autoref{Prop_D_LeftSemiFredholm},
  $\ker \Dmin \iso \coker \Dmax$ is finite-dimensional.  
  Therefore,
  the assertion is an immediate consequence of \autoref{Lem_FredholmResidueCondition}.
\end{proof}

\begin{remark}
  \label{Rmk_FredholmPair}
  The proof of \autoref{Prop_FredholmResidueCondition} shows that $\delta_\ResidueCondition$ and $\delta^\ResidueCondition$ are Fredholm if and only if $(\Lambda,\ResidueCondition)$ forms a \defined{Fredholm pair} in the sense of \citet[Chapter IV §4.1]{Kato1995}.
\end{remark}

\begin{example}
  \label{Ex_CalderonSubspacePerp_Fredholm}
  Every complement $\ResidueCondition$ of the Calderón subspace $\Lambda$, in particular: $\ResidueCondition = \Lambda^\perp$, produces a Fredholm extension of index $0$ because $\delta^\ResidueCondition \co \ResidueCondition \to \GelfandRobbinQuotient/\Lambda \iso \ResidueCondition$ is an isomorphism.
\end{example}

\begin{example}
  \label{Ex_AbstractMITBag_Fredholm}
  The residue conditions $\AbstractBag_\pm$ defined in \autoref{Ex_AbstractMITBag_MutuallyAdjoint} satisfy
  \begin{equation*}
    \ker D_{\AbstractBag_\pm} = \ker \Dmin
  \end{equation*}
  and, therefore, produce Fredholm extensions of index $0$;
  indeed:
  if $\phi \in \ker D_{\AbstractBag_\pm}$,
  then
  \begin{equation*}
    0 = 2\Inner{D\phi,i\phi}_{L^2}
    =
    \Inner{D\phi,i\phi}_{L^2} - \Inner{\phi,D i\phi}_{L^2} 
    =
    G([\phi],i[\phi])
    =
    \Inner{\sharp \circ J [\phi],i[\phi]}_{\GelfandRobbinQuotient}
    =
    \pm\Abs{[\phi]}_{\GelfandRobbinQuotient}^2.
    \qedhere
  \end{equation*}
\end{example}

The following are occasionally useful to compute or relate indices.

\begin{prop}[Nested Fredholm residue conditions]
  \label{Prop_FredholmResidueConditionComparison}
  Let $\ResidueCondition_1 \subset \ResidueCondition_2 \subset \GelfandRobbinQuotient$ be residue conditions.
  If $\delta_{\ResidueCondition_1},\delta_{\ResidueCondition_2}$ are Fredholm,
  then
  \begin{equation*}
    \ind D_{\ResidueCondition_2} = \ind D_{\ResidueCondition_1} + \dim \ResidueCondition_2/\ResidueCondition_1.
  \end{equation*}
\end{prop}

\begin{cor}
  Let $\ResidueCondition \subset \GelfandRobbinQuotient$ be a residue condition.
  If $\ResidueCondition \subset \ResidueCondition^G$ and $\delta_{\ResidueCondition}$ is Fredholm
  then
  \begin{equation*}
    \ind D_\ResidueCondition = -\frac12 \dim{\ResidueCondition^G}/\ResidueCondition;
  \end{equation*}
  in particular: if $\ResidueCondition$ is Lagrangian, then $\ind D_\ResidueCondition=0$.
\end{cor}

\autoref{Prop_FredholmResidueConditionComparison} is an immediate consequence of \autoref{Prop_FredholmResidueCondition} and the following.

\begin{lemma}[Nested residue conditions]
  \label{Lem_FredholmResidueConditionComparison}
  Let $\ResidueCondition_1 \subset \ResidueCondition_2 \subset \GelfandRobbinQuotient$ be residue conditions.
  There is an exact sequence
  \begin{equation*}
    \ker \delta_{\ResidueCondition_1} \incl \ker \delta_{\ResidueCondition_2} \to \ResidueCondition_2/\ResidueCondition_1
    \to \coker \delta_{\ResidueCondition_1} \onto \coker \delta_{\ResidueCondition_2}.
  \end{equation*}
\end{lemma}

\begin{proof}%[Proof of \autoref{Lem_FredholmResidueConditionComparison}]
   The Snake Lemma applied to
  \begin{equation*}
    \begin{tikzcd}
      &\Lambda \ar[equal]{r} \ar{d}{\delta_{\ResidueCondition_1}} & \Lambda  \ar{d}{\delta_{\ResidueCondition_2}} \\
      \ResidueCondition_2/\ResidueCondition_1 \ar[hook]{r} & \GelfandRobbinQuotient/\ResidueCondition_1 \ar[two heads]{r} & \GelfandRobbinQuotient/\ResidueCondition_2. 
    \end{tikzcd}
  \end{equation*}
  yields the exact sequence.
\end{proof}

\begin{prop}[Deformation of residue conditions]
  \label{Prop_VariationOfBoundaryConditions}
  Let $\ResidueCondition$ be a Hilbert space.
  Let $\iota_{-} \co [0,1] \to \sL(\ResidueCondition,\GelfandRobbinQuotient)$ be a continuous path of embeddings.
  If $D_{\ResidueCondition_t}$ with $\ResidueCondition_t \coloneq \iota_t(\ResidueCondition)$ is Fredholm for every $t \in [0,1]$,
  then  
  \begin{equation*}
    \ind D_{\ResidueCondition_0} = \ind D_{\ResidueCondition_1}.
  \end{equation*}
\end{prop}

\begin{proof}%[Proof of \autoref{Prop_VariationOfBoundaryConditions}]
  By assumption,
  $t \mapsto \delta^{\ResidueCondition_t} \circ \iota_t$ is a continuous path of Fredholm operators and $\iota_t \co \ResidueCondition \to \ResidueCondition_t$ is an isomorphism for every $t \in [0,1]$;
  therefore:
  $t \mapsto \ind D_{\ResidueCondition_t} = \ind \delta^{\ResidueCondition_t} = \ind \delta^{\ResidueCondition_t} \circ \iota_t$ is constant.
\end{proof}

%%% Local Variables:
%%% mode: latex
%%% TeX-master: "DiracOperatorsTwistedByRamifiedLineBundles"
%%% ispell-local-dictionary: "british"
%%% End:

%% file: Chirality.tex
\subsection{Chirality operators}
\label{Sec_Chirality}

In the presence of a chirality operator the theory discussed above refines as follows.

\begin{definition}
  \label{Def_ChiralityOperator}
  A \defined{chirality operator} on $(S,\gamma,\nabla,\Tor)$ is a self-adjoint parallel isometry $\epsilon \in \Gamma\paren{X,\O(S)}$ such that
  \begin{equation*}
    \gamma \epsilon + \epsilon\gamma = 0.
    \qedhere
  \end{equation*}
\end{definition}

\begin{example}
  \label{Ex_CanonicalChiralityOperators}
  Assume that $X$ is oriented.
  If $\dim X = 0 \pmod 4$,
  then $\epsilon \coloneq \gamma(\vol_g)$ is a chirality operator.
  If $\dim X = 0 \pmod 2$ and $(S,\gamma,\nabla,\Tor)$ is a complex Dirac bundle with skew torsion as in \autoref{Ex_AbstractMITBag_MutuallyAdjoint},
  then $\epsilon \coloneq i^{\floor{(n+1)/2}}\gamma(\vol_g)$ is a chirality operator.
  % \TODO{\cite{Lawson1989}}
\end{example}

\begin{prop}[Chirality operators induce a $\set{\pm 1}$--grading]
  \label{Prop_ChiralityOperator}
  If $\epsilon$ is a chirality operator for $(S,\gamma,\nabla,\Tor)$,
  then the following hold:
  \begin{enumerate}
  \item
    \label{Prop_ChiralityOperator_Bundle}
    The subbundles $S^\pm \coloneq \ker \paren{\one \mp \epsilon} \subset S$ are parallel,
    $S$ orthogonally decomposes as
    \begin{equation*}
      S = S^+ \oplus S^-,
    \end{equation*}
    and $\gamma \in \Gamma\paren{X,\Hom(TX,\Hom(S^+,S^-)\oplus\Hom(S^-,S^+))}$.
  \item
    \label{Prop_ChiralityOperator_Dirac}
    The minimal and maximal extensions decompose as
    \begin{equation*}
      \Dmin =
      \begin{pmatrix}
        0 & \Dmin^- \\
        \Dmin^+ & 0
      \end{pmatrix}
      \qandq
      \Dmax =
      \begin{pmatrix}
        0 & \Dmax^- \\
        \Dmax^+ & 0
      \end{pmatrix}
    \end{equation*}
    with
    \begin{gather*}
      \Dmin^\pm \co
      \dom(\Dmin^\pm) \coloneq H^1\Gamma\paren{X\setminus Z,S^\pm\otimes\fl}
      \to
      L^2\Gamma\paren{X\setminus Z,S^\mp\otimes\fl} \qand \\
      \Dmax^\pm \co
      \dom(\Dmax^\pm) \coloneq\dom(\Dmax) \cap L^2\Gamma\paren{X\setminus Z,S^\pm\otimes\fl}
      \to
      L^2\Gamma\paren{X\setminus Z,S^\mp\otimes\fl}.
    \end{gather*}
  \item
    \label{Prop_ChiralityOperator_GelfandRobbinQuotient}
    $\GelfandRobbinQuotient$ orthogonally decomposes as
    \begin{equation*}
      \GelfandRobbinQuotient = \GelfandRobbinQuotient^+ \oplus \GelfandRobbinQuotient^-
      \qwithq
      \GelfandRobbinQuotient^\pm \coloneq \frac{\dom(\Dmax^\pm)}{\dom(\Dmin^\pm)};
    \end{equation*}
    moreover, $\GelfandRobbinQuotient^\pm \subset \GelfandRobbinQuotient$ are Lagrangian.
  \item
    \label{Prop_ChiralityOperator_+->-}    
    If $\ResidueCondition^+ \subset \GelfandRobbinQuotient^+$ is a closed subspace, a \defined{positive chirality residue condition},
    then there is a unique closed subspace $\ResidueCondition^- \subset \GelfandRobbinQuotient^-$ such that $\ResidueCondition \coloneq \ResidueCondition^+ \oplus \ResidueCondition^- \subset \GelfandRobbinQuotient$ is a Lagrangian residue condition.
  \item
    \label{Prop_ChiralityOperator_Fredholm}
    Let $\ResidueCondition^+ \subset \GelfandRobbinQuotient^+$ be a \defined{positive chirality residue condition} and $\ResidueCondition^-$ as above.
    The operator $\delta^\ResidueCondition$ is Fredholm
    if and only if $\delta^{\ResidueCondition_+} \co \ResidueCondition^+ \to \GelfandRobbinQuotient^+/\Lambda^+$ is Fredholm.
  \end{enumerate}
\end{prop}

\begin{proof}%[Proof of \autoref{Prop_ChiralityOperator}]
  \autoref{Prop_ChiralityOperator_Bundle}, \autoref{Prop_ChiralityOperator_Dirac}, and \autoref{Prop_ChiralityOperator_GelfandRobbinQuotient} are immediate consequences of \autoref{Def_ChiralityOperator}.

  A moment's thought shows that \autoref{Prop_ChiralityOperator_+->-} holds with
  $\ResidueCondition^- \coloneq (\ResidueCondition^+)^G \cap \GelfandRobbinQuotient^-$.

  Evidently, $\delta^\ResidueCondition = \delta^{\ResidueCondition_+} \oplus \delta^{\ResidueCondition_-}$ is Fredholm if and only if $\delta^{\ResidueCondition_+}$ and $\delta^{\ResidueCondition_-}$ are Fredholm.
  The Green's form $G$ induces isomorphisms
  \begin{equation*}
    \Lambda^- \iso (\GelfandRobbinQuotient^+/\Lambda^+)^*
    \qandq
    \ResidueCondition^+ \iso (\GelfandRobbinQuotient^-/\ResidueCondition^-)^*.
  \end{equation*}
  This identifies the dual of $\delta^{\ResidueCondition^+}$ with
  $\delta_{\ResidueCondition_-} \co \Lambda^- \to \GelfandRobbinQuotient^-/\ResidueCondition^-$.
  By the closed image theorem,
  if $\delta^{\ResidueCondition^+}$ is Fredholm,
  then $\delta_{\ResidueCondition_-}$ is Fredholm.
  As in the proof of \autoref{Prop_FredholmResidueCondition},
  $\delta_{\ResidueCondition_-}$ is Fredholm if and only if $\delta^{\ResidueCondition_-}$ is Fredholm.
  This proves \autoref{Prop_ChiralityOperator_Fredholm}.
\end{proof}

%%% Local Variables:
%%% mode: latex
%%% TeX-master: "DiracOperatorsTwistedByRamifiedLineBundles"
%%% ispell-local-dictionary: "british"
%%% End:

%% file: GelfandRobbinQuotientGeometric.tex
\section{The Gelfand--Robbin quotient, II: geometric realisation}
\label{Sec_GeometricRealisation}

The usefulness of the theory laid out in \autoref{Sec_GelfandRobbinQuotient} hinges upon being able to understand $\GelfandRobbinQuotient$,
e.g., to specify interesting residue conditions.
Since $\GelfandRobbinQuotient$ localises on $Z$ as explained in \autoref{Rmk_GelfandRobbinQuotient},
it is plausible that it admits a more geometric description.
The purpose of this section is to develop such a description,
assuming the following geometric condition on the branching locus $Z$ throughout.

\begin{hypothesis}
  \label{Hyp_CodimensionTwoCooriented}
  $Z \subset X$ is a closed cooriented submanifold of codimension two.
\end{hypothesis}

\begin{remark}
  \label{Rmk_NonOrientableBranchingLoci}
  The assumption that $Z$ is cooriented simplifies the upcoming discussion,
  but is not essential.
  \autoref{Sec_NonCoorientableBranchingLoci} explains how to remove it,
  by introducing more notation.
\end{remark}

\begin{lemma}[{borderline Hardy inequality; \citet[Lemma 2.6]{Takahashi2015}}]
  \label{Lem_BorderlineHardyInequality}
  \autoref{Hyp_CodimensionTwoCooriented} implies \autoref{Hyp_BorderlineHardyInequality}.
\end{lemma}

\begin{proof}%[Proof of \autoref{Lem_BorderlineHardyInequality}}]
  Let $r > 0$.
  Denote by $\fl$ the \emph{non-trivial} Euclidean line bundle over $rS^1 \coloneq \set{ z \in \C : \abs{z} = r }$.
  A moment's thought and a scaling consideration show that
  \begin{equation*}
    \int_{rS^1} \abs{r^{-1}s}^2 \lesssim \int_{rS^1} \abs{\nabla s}^2
  \end{equation*}
  for every $s \in \Gamma\paren{rS^1,\fl}$.
  This immediately implies the assertion.
\end{proof}

\input{BlowUp}
\input{ModelOperator}
\input{ModelGelfandRobbinQuotient}
\input{SpectralDecomposition}
\input{LeadingOrderTerms}
\input{AssemblyOfResidueMap}
\input{SpectralAndLocalResidueConditions}

%%% Local Variables:
%%% mode: latex
%%% TeX-master: "DiracOperatorsTwistedByRamifiedLineBundles"
%%% ispell-local-dictionary: "british"
%%% End:

%% file: BlowUp.tex
\subsection{The blow-up of \texorpdfstring{$X$}{X} along \texorpdfstring{$Z$}{Z}}
\label{Sec_BlowUp}

It is convenient to blow-up $X$ along $Z$;
that is: to replace $Z \subset X$ with the following $\U(1)$--principal bundle.

\begin{definition}
  \label{Def_FrameBundle}
  Since $Z$ is cooriented, its normal bundle $NZ$ is a Hermitian line bundle over $Z$.
  Its \defined{frame bundle} is
  \begin{equation*}
    \pi \co F \coloneq \set{ v \in NZ : \abs{v} = 1 } \to Z
  \end{equation*}
  together with $F \circlearrowleft \U(1)$ defined by $v \cdot e^{i\alpha} \coloneq e^{i\alpha}v$.
  Denote the Levi-Civita connection on $F$ by $i\theta \in \Omega^1(F,i\R)$.
\end{definition}

\begin{remark}
  \label{Rmk_FrameBundle}
  The \defined{tautological section} $\del_r \in \Gamma\paren{F,\pi^*NZ}$ and $\del_\alpha \coloneq i\del_r$ trivialise $\pi^*NZ$.     
\end{remark}

In order to replace $Z \subset X$ with $F$ a choice is required.

\begin{definition}
  \label{Def_TubularNeighbourhood}
  Set $U \coloneq [0,1) \cdot F \subset NZ$.
  A \defined{tubular neighbourhood} $\jmath \co U \emb X$ of $Z \subset X$ is an embedding such that $\jmath \circ 0 = \id_Z$ and
  the composition
  \begin{equation*}
    NZ \incl 0^*TNZ \xrightarrow{T\jmath} TX|_Z \onto NZ
  \end{equation*}
  is the identity.
  Here $0 \co Z \to NZ$ denotes the zero section.
\end{definition}

\emph{Choose} a tubular neighbourhood $\jmath \co U \emb X$.

\begin{definition}
  \label{Def_Blowup}
  Set $\hat U \coloneq [0,1) \times F$.
  The \textbf{blow-up} of $X$ along $Z$ is the manifold with boundary 
  \begin{equation*}
    \hat X \coloneq \hat U \cup_{\jmath} \paren*{X\setminus Z}
  \end{equation*}
  obtained by gluing $\hat U$ and $X\setminus Z$ along $\jmath$.  
  The \textbf{blow-down map} $\beta \co \hat X \to X$ is defined by
  $\beta(r,v) \coloneq \jmath(rv)$ for $(r,v) \in \hat U$ and
  $\beta(x) \coloneq x$ for $x \in X\setminus Z$.
\end{definition}

\emph{Henceforth},
identify $U \subset NZ$ and $\jmath(U) \subset X$;
moreover, identify $\del \hat X = F$.

\begin{definition}
  \label{Def_US}
  Set
  \begin{equation*}
    \hat S \coloneq \beta^*S
    \qandq
    \underline S \coloneq \hat S|_F = \pi^*(S|_Z).
  \end{equation*}
  Endow $S|_Z$ with the complex structure $I \coloneq \gamma(\vol_{NZ})$ and
  $\underline S$ with the quaternionic structure
  \begin{equation*}
    I \coloneq \gamma(\vol_{NZ}), \quad
    J \coloneq \gamma(\del_r), \qandq
    K = IJ \coloneq \gamma(\del_\alpha) \in \Gamma\paren{F,\End(\underline S)}.
  \end{equation*}
  Since $X\setminus Z \emb \hat X$ is a homotopy-equivalence,
  $\fl$ extends uniquely to a Euclidean line bundle
  \begin{equation*}
    \hat \fl \to \hat X.
  \end{equation*}
  Set
  \begin{equation*}
    \underline \fl \coloneq \hat \fl|_F.
    \qedhere
  \end{equation*}
\end{definition}

%%% Local Variables:
%%% mode: latex
%%% TeX-master: "DiracOperatorsTwistedByRamifiedLineBundles"
%%% ispell-local-dictionary: "british"
%%% End:

%% file: ModelOperator.tex
\subsection{The model operator}
\label{Sec_ModelOperator}

The purpose of this subsection is to construct a model $\mathring D$ for $D$ near $Z$.
This construction relies on the following.

\begin{definition}[Restriction of Dirac bundles]
  \label{Def_RestrictionOfDiracBundles}
  Denote the second fundamental form of $Z$ with respect to $\nabla^T$ by $\rII \in \Gamma\paren{Z,\Hom(TZ,\Hom(TZ,NZ))}$.
  The \defined{restriction} of $(S,\gamma,\nabla,\Tor)$ to $Z$ is the quadruple
  $\paren{S|_Z,\gamma|_{TZ},\nabla|_Z + \tfrac12\gamma(\rII),\Tor|_Z}$
  with
  \begin{equation*}
    \gamma(\rII)(v) \coloneq \sum_{i=1}^{n-2} \gamma(\rII(v)e_i)\gamma(e_i).
  \end{equation*}
  Here $\paren{e_1,\ldots,e_{n-2}}$ denotes a local orthonormal frame of $TZ$.  
\end{definition}

\begin{prop}
  \label{Prop_RestrictionOfDiracBundles}
  $\paren{S|_Z,\gamma|_{TZ},\nabla|_Z + \tfrac12\gamma(\rII),{\Tor}|_Z}$ is a Dirac bundle with skew torsion over $(Z,g|_Z)$.
\end{prop}

\begin{proof}%[Proof of \autoref{Prop_RestrictionOfDiracBundles}]
  Evidently, $(S|_Z,\gamma|_{TZ})$ forms a Clifford module bundle over $(Z,g|_Z)$.
  Denote by $\nabla^{T,\parallel}$ the orthogonal affine connection on $(Z,g|_Z)$ induced by $\nabla^T$.

  Since $(S,\gamma,\nabla,\Tor)$ is a Dirac bundle with skew torsion over $(X,g)$,
  for every $v,w \in \Vect(Z)$
  \begin{equation*}
    [\nabla_v,\gamma(w)]
    =\gamma(\nabla_v^{T,\parallel} w) + \gamma(\rII(v)w);
  \end{equation*}
  moreover,
  by direct computation,
  using that $\rII(v)e_i$ is normal,
  \begin{align*}
    [\gamma(\rII)(v),\gamma(w)]
    =
      \sum_{i=1}^{n-2} [\gamma(\rII(v)e_i)\gamma(e_i),\gamma(w)] 
    &=
      \sum_{i=1}^{n-2} \gamma(\rII(v)e_i)\paren{\gamma(e_i)\gamma(w) + \gamma(w)\gamma(e_i)} \\
    &=
      -2\sum_{i=1}^{n-2} \gamma(\rII(v)e_i) \Inner{e_i,w} 
      =
      -2\gamma(\rII(v)w).
  \end{align*}
  A moment's thought shows that if $\nabla^\parallel$ denotes the Levi-Civita connection of $(Z,g|_Z)$,
  then
  \begin{equation*}
    \Inner{\nabla_u^{T,\parallel}v,w}
    =
    \Inner{\nabla_u^\parallel v,w} + (\tfrac12{\Tor}|_Z)(u,v,w).
  \end{equation*}
  This proves the assertion.
\end{proof}

\begin{prop}
  \label{Prop_ModelDiracBundleWithSkewTorsion}
  Denote by $g^\parallel \coloneq g|_Z$ and $g^\perp$ the Euclidean metrics on $TZ$ and $NZ$ induced by $g$.
  Denote by $\Pi \co U \to Z$ the projection map and identify $TU = \Pi^*\paren{TZ \oplus NZ}$ using the Levi-Civita connection.
  Consider $U \subset X$ equipped with the Riemannian metric
  \begin{equation*}
    \mathring g \coloneq \Pi^*(g^\parallel \oplus g^\perp).
  \end{equation*}
  The quadruple $(\mathring{S}, \mathring{\gamma}, \mathring{\nabla},\mathring{\Tor})$ consisting of
  \begin{gather*}
    \mathring{S} \coloneq \Pi^*(S|_Z), \quad
    \mathring{\gamma} \coloneq \Pi^*(\gamma|_Z), \quad
    \mathring{\nabla} \coloneq \Pi^*\paren{\nabla|_Z + \tfrac12\gamma(\rII)}, \\ \andq
    \mathring{\Tor} \coloneq \Pi^*(\Tor_Z^{3,0} + \Tor_Z^{1,2}) + r \pr_F^*(\rd\theta\wedge\theta)
  \end{gather*}
  is a Dirac bundle with skew torsion over $(U,\mathring g)$.
  Here $\Tor_Z^{p,q}$ denotes the $(p,q)$ component with respect to $\Wedge^\bullet (T^*Z\oplus N^*Z) = \Wedge^\bullet T^*Z \otimes \Wedge^\bullet N^*Z$ of the restriction of $\Tor$ to $Z$.
\end{prop}

\begin{proof}%[Proof of \autoref{Prop_ModelDiracBundleWithSkewTorsion}]
  Denote by $\nabla^{T,\parallel}$ and $\nabla^{T,\perp}$ the orthogonal covariant derivatives on $TZ$ and $NZ$ induced by $\nabla^T$ respectively.
  If $v \in \Vect(Z)$ and $w \in \Gamma\paren{Z,NZ}$,
  then
  \begin{equation*}
    [\nabla_v,\gamma(w)]
    =\gamma(\nabla_v^{T,\perp} w) - \gamma(\rII(v)^*w)
  \end{equation*}
  and, moreover, as in the proof of \autoref{Prop_RestrictionOfDiracBundles},
  \begin{align*}
    [\gamma(\rII)(v),\gamma(w)]
    =
    \sum_{i=1}^{n-2} [\gamma(\rII(v)e_i)\gamma(e_i),\gamma(w)]
    &=
      - \sum_{i=1}^{n-2} \paren*{\gamma(\rII(v)e_i)\gamma(w) + \gamma(w)\gamma(\rII(v)e_i)} \gamma(e_i) \\
    &=
      2\sum_{i=1}^{n-2} \Inner{\rII(v)e_i,w} \gamma(e_i)
      =
      2\gamma(\rII^*(v)w).
  \end{align*}
  This together with the analogous computation in the proof of \autoref{Prop_RestrictionOfDiracBundles} proves that $\mathring \gamma$ is parallel with respect to $\mathring \nabla$ and $\Pi^*\paren{\nabla^{T,\parallel}\oplus \nabla^{T,\perp}}$.
  Therefore, it remains to identify the torsion of $\mathring\nabla^T \coloneq \Pi^*\paren{\nabla^{T,\parallel}\oplus \nabla^{T,\perp}}$.
  
  Denote by $\tilde- \co \Vect(Z) \to \Vect(U\setminus Z)$ and  $\tilde- \co \Gamma(Z,NZ) \to \Vect(U\setminus Z)$ the lifting maps.
  For $u,v \in \Vect(Z)$ and $n,m \in \Gamma(Z,NZ)$,
  by direct computation,  
  \begin{align*}
    \mathring\nabla_{\tilde u}^T\tilde v - \mathring\nabla_{\tilde v}^T\tilde u - [\tilde u,\tilde v]
    &=
      \Tor_Z^{3,0}(u,v,-)^\sharp
      +
      \widetilde{[u,v]} - [\tilde u,\tilde v] \\
    &=     
      \Tor_Z^{3,0}(u,v,-)^\sharp
      +
      (\pr_F^*\rd\theta)(\tilde u,\tilde v) \otimes \del_\alpha;
  \end{align*}
  moreover,
  \begin{equation*}
    \mathring\nabla_{\tilde n}^T\tilde m - \mathring\nabla_{\tilde m}^T\tilde n - [\tilde n,\tilde m]
    =
    0 \qandq
    \mathring\nabla_{\tilde v}^T\tilde n - \mathring\nabla_{\tilde n}^T\tilde v - [\tilde v,\tilde n]
    =
    \Tor_Z^{1,2}(v,n,-)^\sharp.
  \end{equation*}
  This proves the assertion.
\end{proof}

\begin{definition}%[Model Dirac operator]
  \label{Def_ModelDiracOperator}
  Denote by $\mathring \fl$ the pullback of $\underline\fl$ along the projection $U\setminus Z \iso F \times (0,1) \to F$.
  The \defined{model Dirac operator}
  \begin{equation*}
    \mathring D
    \co
    H_\loc^1\Gamma\paren{U\setminus Z,\mathring S \otimes \mathring \fl}
    \to
    L_\loc^2\Gamma\paren{U \setminus Z,\mathring S \otimes \mathring \fl}
  \end{equation*}
  is the Dirac operator associated with $(\mathring S,\mathring \nabla,\mathring \gamma,\mathring \Tor)$ twisted by $\mathring \fl$.
  \qedhere
\end{definition}

\begin{remark}
  \label{Rmk_ModelDiracOperator}
  More explicitly,
  the model Dirac operator $\mathring D$ is of the form
  \begin{equation*}
    \mathring D = J\paren{\del_r - r^{-1}I\mathring\nabla_{\del_\alpha}} + D_Z
    \qwithq
    D_Z \coloneq \sum_{i=1}^{n-2} \mathring\gamma(\tilde e_i) \mathring\nabla_{\tilde e_i}
  \end{equation*}
  with $(\tilde e_1,\ldots,\tilde e_{n-2})$ denoting the horizontal lift of a local $g^\parallel$--orthonormal frame.
\end{remark}

\emph{Choose} an isometry $\mathring S \iso S|_U$ which agrees with $\id_{S|_Z}$ over $Z$ and such that $\mathring\nabla - \nabla - \tfrac12\Pi^*(\gamma(\rII))$ vanishes along $Z$; and an isometry $\mathring\fl \iso \fl|_{U\setminus Z}$;
moreover, \emph{henceforth}, regard these as identifications.

\begin{prop}
  \label{Prop_ModelOperatorEstimate}
  The \defined{error term}
  \begin{equation*}
    \ErrorTerm \coloneq D - \mathring D
    \co
    H_\loc^1\Gamma(U \setminus Z,\mathring S \otimes \mathring \fl)
    \to
    L_\loc^2\Gamma(U \setminus Z,\mathring S \otimes \mathring \fl)
  \end{equation*}
  is of the form
  \begin{equation*}
    \ErrorTerm = a\mathring\nabla + b - \tfrac12 \Pi^*\paren{\gamma(H_Z)} + \tfrac 12 \Pi^*\paren{\gamma(\Tor_Z^{2,1})}
  \end{equation*}
  with
  $a \in \Gamma\paren{U,\Hom(T^*U\otimes \mathring S,\mathring S)}$,
  $b \in \Gamma\paren{U,\End(\mathring S)}$, and
  $H_Z$ denoting the mean curvature of $Z$.
  Moreover, $a$ and $b$ vanish along $Z$.
\end{prop}

\begin{proof}
  If $(e_1,\ldots,e_{n-2})$ is a local orthonormal frame of $TZ$,
  then
  \begin{align*}
    -\frac12\sum_{i=1}^{n-2} \gamma(e_i) \gamma(\rII)(e_i)
    &=
      \frac12 \sum_{i,j=1}^{n-2}  \gamma(e_i)\gamma(e_j)\gamma(\rII(e_i)e_j) \\
    &= 
     -\frac12 \gamma(H_Z)
     +\frac14 \sum_{i,j=1}^{n-2} \gamma(e_i)\gamma(e_j)\paren[\big]{\gamma(\rII(e_i)e_j)-\gamma(\rII(e_j)e_i)} \\
    &=
      -\frac12 \gamma(H_Z)
     + \frac12\gamma(\Tor_Z^{2,1}).
  \end{align*}
  Therefore, the assertion follows from the fact that
  \begin{equation*}
    \mathring g - g, \quad
    \mathring\nabla - \nabla - \tfrac12\Pi^*(\gamma(\rII)), \qandq
    \mathring\gamma - \gamma
  \end{equation*}
  vanish along $Z$.
\end{proof}

\begin{remark}
  The error term $\ErrorTerm$ in \autoref{Prop_ModelOperatorEstimate} is not small but only bounded.
  It is possible,
  but not necessary for the purposes of the present article,
  to improve the error term by a more judicious choice of isomorphism as in \cite[paragraph before Proposition 3.2]{Donaldson2021}.
\end{remark}

%%% Local Variables:
%%% mode: latex
%%% TeX-master: "DiracOperatorsTwistedByRamifiedLineBundles"
%%% ispell-local-dictionary: "british"
%%% End:

%% file: ModelGelfandRobbinQuotient.tex
\subsection{The model Gelfand--Robbin quotient}
\label{Sec_ModelGelfandRobbinQuotient}

By (the proof of) \autoref{Prop_DMin_ClosedDenselyDefinedSymmetric},
the model minimal extension
\begin{equation*}
  \mathringDmin \coloneq \mathring D \co \dom(\mathringDmin) \coloneq H_0^1\Gamma\paren{U\setminus Z,\mathring S \otimes\mathring\fl} \to L^2\Gamma\paren{U\setminus Z,\mathring S \otimes\mathring\fl}
\end{equation*}
is closed, densely defined, and symmetric.
A moment's thought shows that the domain of the model maximal extension
\begin{equation*}
  \mathringDmax \coloneq \mathringDmin^*
\end{equation*}
is
\begin{equation*}
  \dom(\mathringDmax)
  \coloneq
  \set{
    \phi \in H_\loc^1\Gamma\paren{U\setminus Z,\mathring S \otimes\mathring\fl}
    :
    \phi,
    \mathring D \phi \in L^2\Gamma\paren{U\setminus Z,\mathring S \otimes\mathring\fl}
  }.
\end{equation*}
The construction from \autoref{Sec_ClosedExtensionsResidueCondition} and \autoref{Sec_GreensFormAndAdjointExtensions} yields the following.

\begin{definition}
  \label{Def_ModelGelfandRobbinQuotient}
  The \defined{model Gelfand--Robbin quotient} is the Hilbert space
  \begin{equation*}
    \mathring{\GelfandRobbinQuotient}
    \coloneq
    \frac{\dom(\mathringDmax)}{\dom(\mathringDmin)}
  \end{equation*}
  equipped with the \defined{model Green's form} $\mathring{G} \in \Hom\paren{\Wedge^2 \mathring{\GelfandRobbinQuotient},\R}$ defined by
  \begin{equation*}
    \mathring{G}([\phi] \wedge [\psi])
    \coloneq
    \Inner{\mathring D\phi,\psi}_{L^2} - \Inner{\phi,\mathring D\psi}_{L^2}.
    \qedhere
  \end{equation*}
\end{definition}

By (the proof of) \autoref{Prop_GreensForm_Symplectic},
$\paren{\mathring{\GelfandRobbinQuotient},\mathring{G}}$ is a symplectic Hilbert space.
In the sense of \autoref{Rmk_GelfandRobbinQuotient},
$\mathring\GelfandRobbinQuotient$ has contributions from $\set{0}\times F \subset \hat U$ and $\set{1}\times F$.
Only the former is relevant for the purposes of this section.
The following construction extracts this part.

\begin{prop}
  \label{Prop_ModelGelfandRobbinQuotient_Decomposition}
  The subspace
  \begin{equation*}
    \mathring{\GelfandRobbinQuotient}_0 \coloneq
    \frac{\chi(r) \dom(\mathringDmax) + \dom(\mathringDmin)}{\dom(\mathringDmin)}
    \subset
    \mathring{\GelfandRobbinQuotient}
  \end{equation*}
  is closed and symplectic.
\end{prop}

\begin{proof}
  Define the bounded linear operator $\pi \in \sL(\mathring{\GelfandRobbinQuotient})$ by
  $\pi([\phi]) \coloneq [\chi(r)  \phi]$.
  Since 
  $\chi(r)(1-\chi(r))  \dom(\mathringDmax) \subset \dom(\mathringDmin)$,
  $\pi^2 = \pi$;
  that is: $\pi$ is a projection.
  Hence,
  $\mathring{\GelfandRobbinQuotient}_0 = \im \pi = \ker (\one-\pi)$ is closed.

  % $\supp\paren{1-\chi(r) - \chi \circ (1-r)} \subset [1/4,3/4]$
  Since $\paren{1-\chi(r) - \chi(1-r)} \dom(\mathringDmax) \subset \dom(\mathringDmin)$,  
  $(\one-\pi)[\phi] = [\chi(1-r)  \phi]$.
  Therefore,
  \begin{equation*}
    \mathring{G}([\phi]\wedge[\psi])
    =
    \mathring{G}(\pi[\phi]\wedge\pi[\psi])
    +
    \mathring{G}((\one-\pi)[\phi]\wedge (\one-\pi)[\psi]).
  \end{equation*}
  Hence, $\mathring{\GelfandRobbinQuotient}_0$ is symplectic.
\end{proof}

\newcommand{\Symplectomorphism}{\textnormal{cut-off}}
\begin{prop}
  \label{Prop_SymplectomorphismOfGelfandRobbinQuotients}
  There is a unique isomorphism of symplectic Hilbert spaces
  \begin{equation*}
    \Symplectomorphism \co (\GelfandRobbinQuotient,G) \iso (\mathring{\GelfandRobbinQuotient}_0,\mathring{G})
  \end{equation*}
  satisfying
  $\Symplectomorphism([\phi])
  =
  [\chi(r)  \phi]$
  for every $\phi \in \dom(\Dmax)$.
\end{prop}

The proof requires the following preparation.

\begin{lemma}[$\dom(\Dmax)$ vs.~$\dom(\mathringDmax)$]
  \label{Lem_DomDMaxVsDomMathringDMax}
   The following hold:
  \begin{enumerate}
  \item
    \label{Lem_DomDMaxVsDomMathringDMax_1}
    If $\phi \in \dom(\mathringDmax)$,
    then $\chi(r) \phi \in \dom(\Dmax)$ and
    $\Abs{\chi(r)  \phi}_{D} \lesssim \Abs{\phi}_{\mathring{D}}$.
  \item
    \label{Lem_DomDMaxVsDomMathringDMax_2}
    If $\phi \in \dom(\Dmax)$,
    then $\chi(r)  \phi \in \dom(\mathringDmax)$ and
    $\Abs{\chi(r)  \phi}_{\mathring D} \lesssim \Abs{\phi}_{D}$.
  \end{enumerate}
\end{lemma}

\begin{proof}%[Proof of \autoref{Lem_RDomDMaxInDomDMin}]
  Let $\phi \in \dom(\mathringDmax)$.
  Let $\eta \in C_c^\infty(U\setminus Z,[0,1])$.
  By \autoref{Prop_DiracOperatorWithSkewTorsion}~\autoref{Prop_DiracOperatorWithSkewTorsion_SchrödingerLichnerowicz},
  \begin{align*}
    \int_{U\setminus Z} \eta^2 \abs{\mathring{\nabla} (r \chi(r) \phi)}^2
    &=
      \int_{U\setminus Z}
      \eta^2
      \abs{\mathring{D} (r\chi(r)\phi)}^2 \\
    &\quad-
      \int_{U\setminus Z}
      \eta^2 \paren[\big]{
      \Inner{\tau \mathring\nabla (r\chi(r)\phi),r\chi(r)\phi}
      + \Inner{\mathring{\gamma}(F_{\mathring{\nabla}})r\chi(r)\phi,r\chi(r)\phi}
      } \\
    &\quad
      + 2\int_{U\setminus Z}
      \eta \Inner{\mathring{D}(r\chi(r)\phi),\mathring{\gamma}(\rd\eta)r\chi(r)\phi} \\
    &\quad
      - 2\int_{U\setminus Z}
      \eta \Inner{\mathring{\nabla}(r\chi(r)\phi),\rd\eta \otimes r\chi(r)\phi}.
  \end{align*}
  Therefore, using Cauchy--Schwarz and rearranging terms,
  \begin{equation*}
    \int_{U\setminus Z} \eta^2 \abs{\mathring{\nabla} (r\chi(r)\phi)}^2
    \lesssim
    \int_{U\setminus Z}  \abs{\mathring{D} \phi}^2
    + r^2\paren{1 + \abs{\rd \eta}^2} \abs{\phi}^2.
  \end{equation*}
  Since $\eta_\epsilon \coloneq 1-\chi(r/\epsilon)$ satisfies
  $r\abs{\rd\eta_\epsilon} \lesssim 1$,
  \begin{equation*}
    \int_{U\setminus Z}\abs{\mathring{\nabla} (r\chi(r)\phi)}^2
    =
    \lim_{\epsilon \downarrow 0} \int_{U\setminus Z} \eta_\epsilon^2\abs{\mathring{\nabla} (r\phi)}^2
    \lesssim
    \int_{U\setminus Z}  \abs{\mathring{D} \phi}^2 + \abs{\phi}^2.
  \end{equation*}
  Therefore,  $r\chi(r)  \phi \in \dom(\mathringDmin)$ and
  \begin{equation*}
    \Abs{r\chi(r) \phi}_{H^1}
    \lesssim
    \Abs{\phi}_{\mathring{D}}.
  \end{equation*}
  
  By \autoref{Prop_ModelOperatorEstimate} and the above,
  \begin{equation}
    \label{Eq_ErrorEstimate}
    \Abs{\ErrorTerm \chi(r) \phi}_{L^2}
    \lesssim
    \Abs{\mathring\nabla (r\phi)}_{L^2} + \Abs{\phi}_{L^2}
    \lesssim
    \Abs{\mathring D \phi}_{L^2} + \Abs{\phi}_{L^2}.
  \end{equation}
  This implies \autoref{Lem_DomDMaxVsDomMathringDMax_1}.
  The proof of \autoref{Lem_DomDMaxVsDomMathringDMax_2} is similar.
\end{proof}

\begin{proof}[Proof of \autoref{Prop_SymplectomorphismOfGelfandRobbinQuotients}]
  By \autoref{Lem_DomDMaxVsDomMathringDMax},
  $\Symplectomorphism$ is an isomorphism of Hilbert spaces.
  To prove that $\Symplectomorphism$ is a symplectomorphism,
  let $\phi,\psi \in \dom(\Dmax)$ and set
  \begin{equation*}
    v \coloneq \sum_{i=1}^n \Inner{\gamma(e_i)\phi,\psi} e_i
    \qandq
    \mathring{v} \coloneq \sum_{i=1}^n \Inner{\mathring{\gamma}(\mathring{e_i})\chi(r)\phi,\chi(r)\psi} \mathring{e_i}
  \end{equation*}
  with $(e_1,\ldots,e_n)$ and $(\mathring e_1,\ldots,\mathring e_n)$ denoting local $g$-- and $\mathring{g}$--orthonormal frames respectively.

  Assume, without loss of generality,
  that $\supp(\phi) \cup \supp(\psi) \subset \paren{\chi \circ r}^{-1}(1) \subset U\setminus Z$.
  With $\eta_\epsilon$ as in the proof of \autoref{Lem_DomDMaxVsDomMathringDMax},  
  \begin{align*}
    \paren{G-\Symplectomorphism^*\mathring G}([\phi] \wedge [\psi])
    &=
      \int_{U\setminus Z}
      \div_g(v)\vol_g
      -\div_{\mathring{g}}(\mathring v)  \vol_{\mathring{g}} \\
    &=
      \lim_{\epsilon \downarrow 0}
      \int_{U\setminus Z}
      \eta_\epsilon  \paren[\big]{\div_g(v)\vol_g
      -\div_{\mathring{g}}(\mathring v)  \vol_{\mathring{g}}} \\
    &=
      - \lim_{\epsilon \downarrow 0}\int_{U\setminus Z} \rd \eta_\epsilon \wedge \paren[\big] {i_v \vol_g -i_{\mathring{v}} \vol_{\mathring{g}}}.
  \end{align*}
  Since $r\abs{\rd \eta_\epsilon} \lesssim 1$,
  \begin{equation*}
    \abs*{\rd \eta_\epsilon \wedge \paren[\big] {i_v \vol_g - i_{\mathring{v}} \vol_{\mathring{g}}}}
    \lesssim
    \abs{\phi}\abs{\psi}.
  \end{equation*}
  Therefore,
  \begin{equation*}
    \abs{\paren{G-\Symplectomorphism^*\mathring G}([\phi] \wedge [\psi])}
    \lesssim
    \lim_{\epsilon \downarrow 0} \int_{\supp \rd\eta_\epsilon} \abs{\phi}\abs{\psi} \,\vol_{\mathring{g}}
    = 0.
    \qedhere
  \end{equation*}
\end{proof}

%%% Local Variables:
%%% mode: latex
%%% TeX-master: "DiracOperatorsTwistedByRamifiedLineBundles"
%%% ispell-local-dictionary: "british"
%%% End:

%% file: SpectralDecomposition.tex
\subsection{Spectral decomposition}
\label{Sec_SpectralDecomposition}

This subsection decomposes $(\mathring{\GelfandRobbinQuotient}_0,\mathring G)$ into concretely understandable summands.

\begin{definition}[$\fl$ determines $NZ^\lambda$]
  \label{Def_NZLambda}
  The ramified Euclidean line bundle $\fl$ determines the following:
  \begin{enumerate}
  \item
    The $2\pi$--periodic vector field $\del_\alpha$ generating $F \circlearrowleft \U(1)$ uniquely lifts along
    \begin{equation*}
      \rho \co \tilde F \coloneq \set{ \ell \in \underline\fl : \abs{\ell} = 1} \to F
    \end{equation*}
    to a $4\pi$--periodic vector field $\tfrac12\del_\beta$.
    The $2\pi$--periodic vector field $\del_\beta$ generates $\tilde F \circlearrowleft \U(1)$ with respect to which $\tilde \pi \co \tilde F \to Z$ is a $\U(1)$--principal bundle.
  \item  
    Let $\lambda \in \tfrac12\Z$.
    The Hermitian line bundle
    \begin{equation*}
      NZ^\lambda \coloneq \tilde F \times_{\U(1)} \C
    \end{equation*}
    arises from $\tilde F$ via the representation $\U(1) \circlearrowright \C$ of weight $2\lambda$.
    The Levi-Civita connection on $F$ induces a connection on $\tilde F$ and,
    therefore, a unitary covariant derivative $\nabla^\lambda$ on $NZ^\lambda$.
    \qedhere
  \end{enumerate}
\end{definition}

\begin{remark}
  \label{Rmk_NZLambda}
  By construction
  $\paren{NZ^1,\nabla^1} \iso \paren{NZ,\nabla^\LC}$
  and for every $\lambda,\mu \in \tfrac12\Z$
  \begin{equation*}
    \paren{NZ^\lambda,\nabla^\lambda} \otimes_\C \paren{NZ^\mu,\nabla^\mu} \iso \paren{NZ^{\lambda+\mu},\nabla^{\lambda+\mu}}.
    \qedhere
  \end{equation*}
\end{remark}

\begin{prop}
  \label{Prop_NZLambda}
  For every $\lambda \in \Z - 1/2$
  there is an isomorphism
  \begin{equation*}
    P_\lambda \co \pi^*(NZ^\lambda,\nabla^\lambda) \iso \paren{\underline\fl\otimes \C, \nabla_{\underline \fl \otimes \C} + i\lambda\theta}
  \end{equation*}
  of Hermitian line bundles with unitary connections.
\end{prop}

\begin{proof}%[Proof of \autoref{Prop_NZLambda}]
  Consider the $\U(1)$--principal bundle $\tilde F\times_{\set{\pm 1}} \U(1) \to F$ obtained by extending the $\set{\pm 1}$--principal bundle $\tilde F \to F$ along the inclusion $\iota \co \set{\pm 1} \incl \U(1)$.
  The $\U(1)$--principal bundles $\tilde F\times_{\set{\pm 1}} \U(1) \to F$ and $\pi^*\tilde F \to F$ are isomorphic via $[f,z] \mapsto (\rho(f), f\cdot z)$.

  Let $\lambda \in \Z-1/2$.
  The representation $\U(1) \circlearrowright \C$ of weight $2\lambda$ restricts to the usual representation $\set{\pm 1} \circlearrowright \C$ along $\iota$.
  Therefore, $\pi^*NZ^\lambda$ and $\underline\fl$ both arise from the representation of weight $2\lambda$.
  Hence, they are isomorphic as Hermitian line bundles.

  The Levi-Civita connection $i\theta$ on $F \to \Z$ induces the connection $\frac{i}{2}\rho^*\theta$ on $\tilde F \to Z$.
  Therefore, the induced connection on $\tilde F \times_{\set{\pm 1}} \U(1) \iso \pi^*\tilde F$ is (the descent of) $\frac{i}{2}\rho^*\theta + \mu_{\U(1)}$.
  Here $\mu_{\U(1)} \in \U(1)$ denotes the Maurer--Cartan form on $\U(1)$.
  The connection on $\tilde F \times_{\set{\pm 1}} \U(1)$ induced by the flat connection on $\tilde F \to F$ is (the descent) of $\mu_{\U(1)}$.
  This implies the assertion about the covariant derivatives.
\end{proof}

\begin{prop}[Spectral decomposition of $L^2\Gamma\paren{F,\underline S \otimes \underline\fl}$]
  \label{Prop_SpectralDecomposition}
  For every $\lambda \in \Z-1/2$ and $\mu \in \R$
  set
  \begin{align*}
    E_{\lambda,\mu}
    &\coloneq
      \set[\big]{
      \phi \in L^2\Gamma\paren{F,\underline S \otimes \underline\fl}
      :
      I\mathring\nabla_{\del_\alpha} \phi = \lambda \phi,
      D_Z\phi = \mu \phi
      } \qand \\
    \check E_{\lambda,\mu}
    &\coloneq
      \set[\big]{
      \check\phi \in L^2\Gamma\paren{Z,S|_Z \otimes_\C NZ^\lambda}
      :
      D_{S|_Z \otimes_\C NZ^\lambda} \check\phi = \mu \check\phi
      }.        
  \end{align*}
  Here $D_Z$ is as in \autoref{Rmk_ModelDiracOperator} and $D_{S|_Z \otimes_\C NZ^\lambda}$ arises from \autoref{Def_RestrictionOfDiracBundles} and twisting by $(NZ^\lambda,\nabla^\lambda)$;
  moreover: the tensor product is with respect to the complex structure $I$ on $S|_Z$.
  The following hold:
  \begin{enumerate}
  \item
    \label{Prop_SpectralDecomposition_=}
    For every $\lambda \in \Z-1/2$ and $\mu \in \R$,
    $P_\lambda$ induces an isomorphism
    \begin{equation*}
      \pi^* \check E_{\lambda,\mu}
      \iso
      E_{\lambda,\mu}.
    \end{equation*}
  \item
    \label{Prop_SpectralDecomposition_Spectrum}
    The subset
    \begin{equation*}
      \Spectrum
      \coloneq
      \set[\big]{
        (\lambda,\mu) \in \paren{\Z-1/2} \times \R
        :
        E_{\lambda,\mu} \neq 0
      }
    \end{equation*}
    is discrete.
    Moreover, for every $(\lambda,\mu) \in \Spectrum$, $E_{\lambda,\mu}$ is finite-dimensional.
  \item
    \label{Prop_SpectralDecomposition_J}
    For every $(\lambda,\mu) \in \Spectrum$
    \begin{equation*}
      JE_{\lambda,\mu} = E_{-(\lambda+1),-\mu}.
    \end{equation*}
  \item
    \label{Prop_SpectralDecomposition_Decomposition}
    The Hilbert space $L^2\Gamma\paren{F,\underline S \otimes \underline\fl}$ decomposes as a (Hilbert space) direct sum
    \begin{equation*}
      L^2\Gamma\paren{F,\underline S \otimes \underline\fl}
      =
      \bigoplus_{(\lambda,\mu) \in \Spectrum} E_{\lambda,\mu}.
    \end{equation*}

  \end{enumerate}
\end{prop}

\begin{proof}%[Proof of \autoref{Prop_SpectralDecomposition}]
  By Fourier analysis,
  the Hilbert space $L^2\Gamma\paren{F,\underline S \otimes \underline\fl}$ decomposes as a direct sum
  \begin{equation*}
    L^2\Gamma\paren{F,\underline S \otimes \underline\fl}
    =
    \bigoplus_{\lambda \in \Z-1/2} E_\lambda
    \qwithq
    E_\lambda
    \coloneq
    \set{
      \phi \in L^2\Gamma\paren{F,\underline S \otimes \underline\fl}
      :
      I\mathring\nabla_{\del_\alpha} \phi = \lambda \phi
    }.
  \end{equation*}
  By \autoref{Prop_NZLambda}, $P_\lambda$ induces an isomorphism
  $\pi^*L^2\Gamma\paren{Z,S|_Z \otimes_\C NZ^\lambda} \iso E_\lambda$.
  By the spectral theory of Dirac operators,
  for every $\lambda \in \Z-1/2$,
  $\spec\paren{D_{S|_Z \otimes_\C NZ^\lambda}} \subset \R$ is discrete and
  the Hilbert space $L^2\Gamma\paren{Z,S|_Z \otimes_\C NZ^\lambda}$ decomposes as a direct sum of finite-dimensional eigenspaces $\check E_{\lambda,\mu}$ of $D_{S|_Z \otimes_\C NZ^\lambda}$.
  This proves \autoref{Prop_SpectralDecomposition_=}, \autoref{Prop_SpectralDecomposition_Spectrum}, and \autoref{Prop_SpectralDecomposition_Decomposition}.

  \autoref{Prop_SpectralDecomposition_J} holds because
  $J$ and $D_Z$ anti-commute and
  $I\mathring\nabla_{\del_\alpha} J = -J\paren{I\mathring\nabla_{\del_\alpha} + \one}$
  since $\mathring\nabla_{\del_\alpha} \del_r = \del_\alpha$.
\end{proof}

Since $\vol_{\mathring g} = \rd r \wedge r \theta \wedge \vol_{g|_Z}$,
by Fubini's theorem and \autoref{Prop_SpectralDecomposition},
\begin{equation*}
  L^2\Gamma\paren{U\setminus Z, \mathring S \otimes \mathring\fl}  
  = L^2\paren{(0,1),r\rd r;L^2\Gamma\paren{F,\underline S \otimes \underline\fl}}  
  =
  \bigoplus_{\lambda \in \Z-1/2} \bigoplus_{\mu \in \Spectrum_\lambda}
  L^2\paren{(0,1),r\rd r;E_{\lambda,\mu}};
\end{equation*}
moreover, $\Dmax$ decomposes as follows.

\begin{definition}
  \emph{Choose} a fundamental domain $\check\Spectrum \subset \Spectrum$ for the involution $(\lambda,\mu) \mapsto (-(\lambda+1),-\mu)$.
  \emph{Choose} a real subspace $E_{-1/2,0}^\R \subset E_{-1/2,0}$ with respect to $J$.
  For every $(\lambda,\mu) \in \check\Spectrum$
  set
  \begin{equation*}
    V_{\lambda,\mu}
    \coloneq
    \begin{cases}
      E_{-1/2,0}^\R \oplus JE_{-1/2,0}^\R & \textnormal{if } (\lambda,\mu) = (-1/2,0) \\
      E_{\lambda,\mu} \oplus E_{-(\lambda+1),-\mu} & \textnormal{otherwise};
    \end{cases}
  \end{equation*}
  moreover,
  define $\mathring D^{\lambda,\mu} \co H_\loc^1((0,1);V_{\lambda,\mu}) \to L_\loc^2((0,1);V_{\lambda,\mu})$ by
  \begin{equation*}
    \mathring{D}^{\lambda,\mu}
    \coloneq      
    \begin{pmatrix}
      \mu & J\paren{\del_r  + \frac{\lambda+1}{r}} \\
      J\paren{\del_r - \frac{\lambda}{r}} & -\mu
    \end{pmatrix}
  \end{equation*}
  and set
  \begin{equation*}
    \dom(\mathringDmax^{\lambda,\mu})
    \coloneq
    \set[\big]{    
      \phi \in H_\loc^1((0,1);V_{\lambda,\mu})
      :
      \phi,\mathring D\phi \in
      L^2\paren{(0,1),r\rd r;V_{\lambda,\mu}}
    }.
    \qedhere
  \end{equation*}
\end{definition}

\begin{remark}
  \label{Rmk_E-1/2,0}
  The purpose of the artificial decomposition of $E_{-1/2,0}$ is to avoid a case distinction in the definition of $\mathring D^{\lambda,\mu}$.
\end{remark}

\begin{prop}[Spectral decomposition of $\dom(\mathringDmax)$]
  \label{Prop_SpectralDecomposition_DomDMax}  
  The following hold:
  \begin{enumerate}
  \item
    \label{Prop_SpectralDecomposition_DomDMax_Dom}  
    The Hilbert space $\dom(\mathringDmax)$ decomposes as a (Hilbert space) direct sum
    \begin{equation*}
      \dom(\mathringDmax)
      =
      \bigoplus_{(\lambda,\mu) \in \check\Spectrum} \dom(\mathringDmax^{\lambda,\mu}).
    \end{equation*}
  \item
    \label{Prop_SpectralDecomposition_DomDMax_D}
    The model operator $\mathring D$ decomposes as
    \begin{equation*}
      \mathring D
      =
      \bigoplus_{(\lambda,\mu) \in \check\Spectrum} 
      \mathring D^{\lambda,\mu}.
    \end{equation*}
  \end{enumerate}
\end{prop}

\begin{proof}%[Proof of \autoref{Prop_SpectralDecomposition_ModelOperator}]
  This is an immediate consequence of \autoref{Rmk_ModelDiracOperator} and \autoref{Prop_SpectralDecomposition}.
\end{proof}

\begin{remark}
  \label{Rmk_ODE}
  The ordinary differential equation $\mathring D^{\lambda,\mu} \phi = \psi$ can be solved explicitly
  in terms of modified Bessel functions of the second kind or using the Hankel transform.
  However, none of this is necessary for the purpose of this article.
\end{remark}

Finally,
here is the desired decomposition of $(\mathring{\GelfandRobbinQuotient}_0,\mathring G)$.

\begin{cor}[Spectral decomposition of $(\mathring\GelfandRobbinQuotient_0,\mathring G)$]
  \label{Cor_SpectralDecomposition_GelfandRobbinQuotient}
  The symplectic Hilbert space $(\mathring\GelfandRobbinQuotient_0,\mathring G)$ decomposes as a (Hilbert space) direct sum
  \begin{equation*}
    \paren{\mathring{\GelfandRobbinQuotient}_0,\mathring G}
    =
    \bigoplus_{(\lambda,\mu) \in \check\Spectrum}
    \paren{\mathring{\GelfandRobbinQuotient}_0^{\lambda,\mu},\mathring G_{\lambda,\mu}}
  \end{equation*}
  with
  \begin{align*}
    \mathring{\GelfandRobbinQuotient}_0^{\lambda,\mu}
    &\coloneq
      \frac{\chi(r) \dom(\mathringDmax^{\lambda,\mu})+ \dom(\mathringDmin)}{\dom(\mathringDmin)}
      \qand \\
    \mathring G_{\lambda,\mu}([\phi] \wedge [\psi])
    &\coloneq
      \Inner{\mathring D^{\lambda,\mu}\phi,\psi}_{L^2} -   \Inner{\phi,\mathring D^{\lambda,\mu}\psi}_{L^2}.
      \qedhere
  \end{align*}
\end{cor}

%%% Local Variables:
%%% mode: latex
%%% TeX-master: "DiracOperatorsTwistedByRamifiedLineBundles"
%%% ispell-local-dictionary: "british"
%%% End:

%% file: LeadingOrderTerms.tex
\subsection{Leading order terms}
\label{Sec_LeadingOrderTerms}

This subsection determines $\paren{\mathring{\GelfandRobbinQuotient}_0^{\lambda,\mu},\mathring G_{\lambda,\mu}}$ based on the following observation.

\begin{lemma}[{Leading order terms; cf.~\cites[Lemma 2.1]{BrueningSeeley1988}[Lemma 3.51]{Doan2024}}]
  \label{Lem_LeadingTerm}
  Let $\lambda \in \R$.
  Let $\phi \in H_\loc^1((0,1))$ with $\phi,\paren{\partial_r - \lambda/r}\phi \in L^2((0,1),r \rd r)$.
  The following hold:
  \begin{enumerate}
  \item
    \label{Lem_LeadingTerm_(-1,0)}
    If $\lambda \in (-1,0)$,
    then there is a unique $a \in \R$ such that
    $\lim_{r\downarrow 0} \phi(r) - ar^{\lambda} = 0$.
  \item
    \label{Lem_LeadingTerm_=0}
    If $\lambda = 0$,
    then $\abs{\phi(r)} \lesssim_{\phi} \abs{\log(r)}^{1/2}$.
  \item
    \label{Lem_LeadingTerm_Else}
    If $\lambda \notin (-1,0]$,
    then $\lim_{r\downarrow 0} \phi(r) = 0$.
  \item
    \label{Lem_LeadingTerm_Estimate}
    If $\lim_{r \downarrow 0} \phi(r) = \lim_{r \uparrow 1} \phi(r) = 0$,
    then
    \begin{equation*}
      \int_0^1 \paren*{\abs{\del_r \phi}^2 + \tfrac{\lambda^2}{r^2}\abs{\phi}^2} \,r\rd r
      =
      \int_0^1 \abs*{\paren[\big]{\del_r-\tfrac{\lambda}{r}} \phi}^2 r\rd r.
    \end{equation*}
  \end{enumerate}
\end{lemma}

\begin{proof}%[Proof of \autoref{Lem_LeadingTerm}]
  The proof is almost identical to that of \cite[Lemma 3.50]{Doan2024} and is repeated here only for the reader's convenience.
  
  Evidently,
  $(\del_r-\lambda/r)r^\lambda = 0$ and $r^\lambda \in L^2\paren{(0,1),r\rd r}$ if and only if $\lambda > -1$.
  Set $\psi \coloneq \paren{\del_r - \lambda/r}\phi$.
  By variation of parameters,
  there is a unique $a \in \R$ such that
  \begin{equation*}
    \tilde\phi(r)
    \coloneq
    \phi(r) - a r^\lambda
    =
    \begin{cases}
      r^\lambda\int_0^r s^{-(\lambda+1)}\psi(s) \,s \rd s & \text{if}~\lambda < 0 \\
      -r^\lambda\int_r^1 s^{-(\lambda+1)}\psi(s) \,s \rd s & \text{if}~\lambda \geq 0.
    \end{cases}
  \end{equation*}
  Of course, if $\lambda \leq -1$, then $a = 0$.

  If $\lambda < 0$,
  then, by Cauchy--Schwarz and monotone convergence,
  \begin{equation*}
    \abs{\tilde\phi(r)}^2 \leq \frac{1}{2\abs{\lambda}} \int_0^r \abs{\psi(s)}^2 \,s\rd s = o(1)
    \quad\textnormal{as}\quad r \downarrow 0.
  \end{equation*}
  If $\lambda = 0$,
  then
  \begin{equation*}
    \abs{\tilde\phi(r)}^2 \leq \abs{\log(r)} \int_r^1 \abs{\psi(s)}^2 \,s\rd s = O(\abs{\log(r)})
    \quad\textnormal{as}\quad r \downarrow 0.
  \end{equation*}  
  If $\lambda > 0$, then,
  by Cauchy–Schwarz,
  for $r \leq \epsilon \leq 1$
  \begin{equation*}
    \abs{\tilde\phi(r)}^2
    \leq
    \frac{1}{\lambda}
    \int_0^\epsilon \abs{\psi(s)}^2 \,s\rd s
    +
    \frac{(r/\epsilon)^{2\lambda}}{\lambda} \int_\epsilon^1 \abs{\psi(s)}^2 \,s\rd s
    \eqcolon \rI(\epsilon) + \rII(r,\epsilon).
  \end{equation*}
  By monotone convergence,
  $\lim_{\epsilon \downarrow 0} \rI(\epsilon) = 0$.
  Evidently,
  $\lim_{r \downarrow 0} \rII(r,\epsilon) = 0$.
  Therefore, $\tilde\phi(r) = o(1)$ as $r \downarrow 0$.
  These observations imply \autoref{Lem_LeadingTerm_(-1,0)}, \autoref{Lem_LeadingTerm_=0}, and \autoref{Lem_LeadingTerm_Else}.  
  
  \autoref{Lem_LeadingTerm_Estimate} is a consequence of
  \begin{equation*}
    \int_0^1 \abs*{\paren[big]{\del_r-\tfrac{\lambda}{r}} \phi}^2 r\rd r
    =
    \int_0^1 \paren*{\abs{\del_r\phi}^2 + \tfrac{\lambda^2}{r^2} \abs{\phi}^2} \, r\rd r
    -\lambda
    \int_0^1 \del_r \abs{\phi}^2 \rd r.
    \qedhere
  \end{equation*}
\end{proof}

\begin{cor}[Identification of $\mathring\GelfandRobbinQuotient_0^{\lambda,\mu}$ for $\lambda \neq -1/2$]
  For every $(\lambda,\mu) \in \check\Spectrum$ with $\lambda \neq -1/2$
  \begin{equation*}
    \mathring\GelfandRobbinQuotient_0^{\lambda,\mu} = 0.
    \qedhere
  \end{equation*}
\end{cor}

\begin{definition}
  For $(-1/2,\mu) \in \check\Spectrum$
  define the \defined{residue map} $\res_\mu \co \chi(r) \dom(\mathringDmax^{-1/2,\mu}) \to V_{-1/2,\mu}$ by
  \begin{equation*}
    \phi - r^{-1/2} \res_\mu(\phi) \in \dom(\mathringDmin),
  \end{equation*}
  and the symplectic form $\check\Omega_\mu \in \Hom(\Wedge^2V_{-1/2,\mu},\R)$ by
  \begin{equation*}
    \check\Omega_\mu(\phi \wedge \psi) \coloneq -\Inner{J\phi,\psi}.
    \qedhere
  \end{equation*}
\end{definition}

\begin{prop}[Identification of $\mathring{\GelfandRobbinQuotient}_0^{-1/2,\mu}$: symplectic structure]
  \label{Prop_Identification_Symplectic}
  For every $(-1/2,\mu) \in \check\Spectrum$
  the residue map induces an isomorphism
  \begin{equation*}
    \res_\mu
    \co
    \paren{\mathring{\GelfandRobbinQuotient}_0^{-1/2,\mu},\mathring G^{-1/2,\mu}}
    \iso
    \paren{V_{-1/2,\mu},\check\Omega_\mu}.
  \end{equation*}
\end{prop}

\begin{proof}%[Proof of \autoref{Prop_Identification_Symplectic}]
  For $\phi,\psi\in \chi(r) \dom\paren{\mathringDmax^{-1/2,\mu}}$,
  by direct computation using $\del_r + \tfrac{1}{2r} = r^{-1/2} \del_r r^{1/2}$,
  \begin{align*}
    \mathring G^{-1/2,\mu}([\phi]\wedge[\psi])
    &=
      \int_0^1 \paren[\big]{\Inner{J\paren{\del_r + \tfrac{1}{2r}}\phi,\psi} - \Inner{\phi,J\paren{\del_r + \tfrac{1}{2r}}\psi}} \, r\rd r \\
    &=
      \int_0^1 \del_r \Inner{J r^{1/2}\phi,r^{1/2}\psi} \, \rd r \\
    &=
      -\Inner{J \res_\mu([\phi]),\res_\mu([\psi])}
      =
      \res_\mu^*\check\Omega_\mu([\phi]\wedge[\psi]).
  \end{align*}
  This together with \autoref{Lem_LeadingTerm} immediately implies the assertion.
\end{proof}

Although $\res_\mu$ is an isomorphism, the norms on $\mathring{\GelfandRobbinQuotient}_0^{-1/2,\mu}$ and $V_{-1/2,\mu}$ are \emph{not uniformly} equivalent.
The following discussion rectifies this.

\begin{definition}
  \label{Def_BranchingLocusOperator_Mu}
  Let $(-1/2,\mu) \in \check\Spectrum$.
  \begin{enumerate}
  \item
    Define the \defined{branching locus operator} $A_\mu \co V_{-1/2,\mu} \to V_{-1/2,\mu}$ by
    \begin{equation*}
      A_\mu
      \coloneq
      -J\mathring D^{-1/2,\mu} - \del_r - \tfrac{1}{2r}
      =
      \begin{pmatrix}
        0 & J \mu \\
        -J \mu & 0
      \end{pmatrix}.
    \end{equation*}
  \item
    Define the norm $\Abs{-}_{\check H} \co V_{-1/2,\mu} \to [0,\infty)$ by    
    \begin{equation*}
      \Abs{\rho}_{\check H}^2
      \coloneq
      \bracket{\mu} \Abs{\one_{(-\infty,0)}(A_\mu)\rho}^2
      +
      \bracket{\mu}^{-1} \Abs{\one_{[0,\infty)}(A_\mu)\rho}^2.
    \end{equation*}
    Here $\one_{(-\infty,0)}(A_\mu)$ and $\one_{[0,\infty)}(A_\mu)$ denote the orthogonal projection to the negative and non-negative eigenspaces of $A_\mu$ respectively;
    and $\bracket{x} \coloneq \paren{1+x^2}^{1/2}$ as explained in the conventions at the end of \autoref{Sec_Introduction}.
  \item
    Define the norm $\Abs{-}_{H^{-1/2}} \co V_{-1/2,\mu} \to [0,\infty)$ by
    \begin{equation*}
      \Abs{\rho}_{H^{-1/2}}^2
      \coloneq
      \bracket{\mu}^{-1} \Abs{\rho}^2.
      \qedhere
    \end{equation*}
  \end{enumerate}
\end{definition}

\begin{prop}[Identification of $\mathring{\GelfandRobbinQuotient}_0^{-1/2,\mu}$: uniform norms]
  \label{Prop_Identification_UniformNorms}
  For every $(-1/2,\mu) \in \check\Spectrum$
  \begin{equation*}
    \Abs{\res_\mu([\phi])}_{\check H}
    \asymp
    \Abs{[\phi]}_{\mathring{\GelfandRobbinQuotient}}.
  \end{equation*}
  Here $x \asymp y$ is a shorthand for $x \lesssim y \lesssim x$.
\end{prop}

The proof uses the following right inverse of $\res_\mu [-]$.

\begin{definition}
  For $(-1/2,\mu) \in \check\Spectrum$ define the \defined{extension map} $\ext_\mu \co V_{-1/2,\mu} \to \dom(\mathringDmax^{-1/2,\mu})$ by
  \begin{equation*}
    \ext_\mu(\rho) \coloneq r^{-1/2}e^{-\abs{\mu} r} \rho.
    \qedhere
  \end{equation*}
\end{definition}

Evidently, $\ext_\mu$ lifts the inverse of $\res_\mu \co \mathring{\GelfandRobbinQuotient}_0^{-1/2,\mu} \iso V_{-1/2,\mu}$.
Therefore,
\autoref{Prop_Identification_UniformNorms} is an immediate consequence of the following.

\begin{lemma}[Uniform estimates for $\res_\mu$ and $\ext_\mu$]
  \label{Lem_UniformEstimatesForRestrictionAndExtension}
  Let $(-1/2,\mu) \in \check\Spectrum$.
  The following hold:
  \begin{enumerate}
  \item
    \label{Lem_UniformEstimatesForRestrictionAndExtension_1}
    For every $\phi \in \chi(r) \dom(\mathringDmax^{-1/2,\mu})$
    \begin{equation*}
      \Abs{\res_\mu([\phi])}_{\check H}
      \lesssim
      \Abs{\phi}_{\mathring D}.
    \end{equation*}
  \item
    \label{Lem_UniformEstimatesForRestrictionAndExtension_2}
    For every $\rho \in V_{-1/2,\mu}$
    \begin{equation*}
      \Abs{\ext_\mu(\rho)}_{L^2}
      \lesssim
      \Abs{\rho}_{H^{-1/2}}
      \qandq
      \Abs{\ext_\mu(\rho)}_{\mathring D}
      \lesssim
      \Abs{\rho}_{\check H}.
    \end{equation*}    
  \end{enumerate}
\end{lemma}

\begin{proof}%[Proof of \autoref{Lem_UniformEstimatesForRestrictionAndExtension}]
  There is no loss in assuming that $\lim_{r \uparrow 1} \phi(r) = 0$.
  Evidently,
  $\rho \coloneq \res_\mu([\phi])$ satisfies
  \begin{equation*}
    \rho
    =
    -\int_0^1 \del_r \paren{ r^{1/2} e^{-\abs{\mu} r} \phi} \,\rd r \\
    =
    -\int_0^1 r^{1/2}e^{-\abs{\mu} r} \paren{-J\mathring D^{-1/2,\mu} - A_\mu - \abs{\mu}} \phi \,\rd r.
  \end{equation*}
  Therefore, by Cauchy--Schwarz,
  \begin{align*}
    \Abs{\one_{(-\infty,0)}(A_\mu)\rho}^2
    &\lesssim
      \int_0^1 e^{-2\abs{\mu}r} \,\rd r
      \int_0^1 \Abs{\mathring D^{-1/2,\mu}\one_{(-\infty,0)}(A_\mu)\phi}^2 \,r \rd r \\
    &\lesssim
      \int_0^1 e^{-2\abs{\mu}r} \,\rd r
      \,
      \Abs{\one_{(-\infty,0)}(A_\mu)\phi}_{\mathring D}^2
  \end{align*}
  and
  \begin{align*}
    \Abs{\one_{[0,\infty)}(A_\mu)\rho}^2
    &\lesssim
      \int_0^1 e^{-2\abs{\mu}r} \,\rd r
      \int_0^1 \Abs{\mathring D^{-1/2,\mu}\one_{[0,\infty)}(A_\mu)\phi}^2 + \abs{\mu}^2\Abs{\one_{[0,\infty)}(A_\mu)\phi}^2 \,r \rd r \\
    &\lesssim
      \bracket{\mu}^2
      \int_0^1 e^{-2\abs{\mu}r} \,\rd r
      \,
      \Abs{\one_{[0,\infty)}(A_\mu)\phi}_{\mathring D}^2.
  \end{align*}
  The estimate in \autoref{Lem_UniformEstimatesForRestrictionAndExtension_1} follows because
  \begin{equation*}
    \int_0^1 e^{-2\abs{\mu} r} \,\rd r
    \asymp
    \bracket{\mu}^{-1}
    \qandq
    \Abs{\one_{(-\infty,0)}(A_\mu)\phi}_{\mathring D}^2 + \Abs{\one_{[0,\infty)}(A_\mu)\phi}_{\mathring D}^2
    =
    \Abs{\phi}_{\mathring D}^2.
  \end{equation*}

  To prove \autoref{Lem_UniformEstimatesForRestrictionAndExtension_2},
  observe that
  \begin{align*}
    \Abs{\ext_\mu(\rho)}_{L^2}^2
    \lesssim
    \int_0^1 r^{-1} e^{-2 \abs{\mu} r} \Abs{\rho}^2 \, r\rd r
    \lesssim
    \bracket{\mu}^{-1} \Abs{\rho}^2
  \end{align*}
  and
  \begin{align*}
    \Abs{\mathring D^{-1/2,\mu}\ext_\mu(\rho)}_{L^2}^2
    &=
      \int_0^1 \paren[\big]{r^{-1/2}e^{-\abs{\mu}r}}^2 \Abs{(-\abs{\mu}+A_\mu)\rho}^2 \,r\rd r \\
    &\lesssim
      \int_0^1 \paren[\big]{r^{-1/2}\abs{\mu} e^{-\abs{\mu} r}}^2 \Abs{\one_{(-\infty,0)}(A_\mu)\rho}^2 \, r\rd r \\
    &\lesssim
      \bracket{\mu} \Abs{\one_{(-\infty,0)}(A_\mu)\rho}^2.
      \qedhere
  \end{align*}
\end{proof}

%%% Local Variables:
%%% mode: latex
%%% TeX-master: "DiracOperatorsTwistedByRamifiedLineBundles"
%%% ispell-local-dictionary: "british"
%%% End:

%% file: AssemblyOfResidueMap.tex
\subsection{Assembly of the residue map}
\label{Sec_AssemblyOfResidueMap}

This subsection (re)assembles the summands of the decomposition \autoref{Cor_SpectralDecomposition_GelfandRobbinQuotient} identified in \autoref{Sec_LeadingOrderTerms} in a more geometric fashion.

\begin{definition}
  \label{Def_AssemblyOfResidueMap}
  ~
  \begin{enumerate}
  \item 
    The \defined{residue bundle} is
    \begin{equation*}
      \check S \coloneq S|_Z \otimes_\C NZ^{-1/2}.
    \end{equation*}
    As a consequence of \autoref{Prop_SpectralDecomposition}~\autoref{Prop_SpectralDecomposition_J} (or by direct inspection),
    $\check S$ inherits the quaternionic structure $I,J,K=IJ$ from $\underline S$.  
    Define the symplectic form $\check\Omega \in \Gamma\paren{Z,\Hom(\Wedge^2 \check S,\R)}$ by
    \begin{equation*}
      \check\Omega \coloneq -2\pi\Inner{J-,-}.
    \end{equation*}
  \item
    The \defined{branching locus operator} $A \co \Gamma\paren{Z,\check S} \to \Gamma\paren{Z,\check S}$ is defined by
    \begin{equation*}
      A \coloneq - JD_{\check S}
    \end{equation*}
    with $D_{\check S} \coloneq D_{S|_Z \otimes_\C NZ^{-1/2}}$ as in \autoref{Prop_SpectralDecomposition}.
    Since $J$ and $D_{\check S}$ anti-commute, $A$ is (formally) self-adjoint.
  \item
    Denote by $\one_{(-\infty,0)}(A)$ and $\one_{[0,\infty)}(A)$ the orthogonal projection to the negative and non-negative eigenspaces of $A$ respectively.
    Define the norm $\Abs{-}_{\check H} \co \Gamma\paren{Z,\check S} \to [0,\infty)$ by
    \begin{equation*}
      \Abs{\phi}_{\check H}
      \coloneq
      \Abs{\one_{(-\infty,0)}(A)\phi}_{H^{1/2}}
      +
      \Abs{\one_{[0,\infty)}(A)\phi}_{H^{-1/2}}
    \end{equation*}
    and denote by $\check H\Gamma\paren{Z,\check S}$ the completion of $\Gamma\paren{Z,\check S}$ with respect to $\Abs{-}_{\check H}$.
    \qedhere
  \end{enumerate}
\end{definition}

\begin{prop}
  \label{Prop_Reassembly}
  $\check\Omega$ extends to a symplectic structure $\check\Omega \in \sL\paren{\Wedge^2 \check H\Gamma\paren{Z,\check S},\R}$;
  moreover:
  the inclusions $V_{-1/2,\mu} \incl \Gamma\paren{Z,\check S}$ assemble into an isomorphism of symplectic Hilbert spaces 
  \begin{equation*}    
    \bigoplus_{(-1/2,\mu) \in \check\Spectrum} \paren{V_{-1/2,\mu},\Abs{-}_{\check H};\check\Omega_\mu}
    \iso
    \paren{\check H\Gamma\paren{Z,\check S};\check\Omega}.
  \end{equation*} 
\end{prop}

\begin{proof}%[Proof of \autoref{Prop_Reassembly}]
  This is an immediate consequence of \autoref{Prop_SpectralDecomposition}, \autoref{Prop_Identification_Symplectic}, and \autoref{Prop_Identification_UniformNorms}.
  The (possibly mysterious) factor $2\pi$ arises because $\check\Omega_\mu$ is defined using the $L^2$ inner product on $F$ instead of $Z$ and $\vol(F) = 2\pi \, \vol(Z)$.
\end{proof}

\begin{definition}%[Residue and extension map]
  ~
  \begin{enumerate}
  \item 
    The \defined{residue map} $\res \co \GelfandRobbinQuotient \to 
     \check H\Gamma\paren{Z,\check S}$ is obtained as the composition of the following maps
    \begin{equation*}
      \GelfandRobbinQuotient
      \xrightarrow{\Symplectomorphism}
      \mathring{\GelfandRobbinQuotient}_0
      =
      \bigoplus_{(-1/2,\mu) \in \check\Spectrum} \mathring{\GelfandRobbinQuotient}_0^{-1/2,\mu}
      \xrightarrow{(\res_\mu)}
      \bigoplus_{(-1/2,\mu) \in \check\Spectrum} \paren{V_{-1/2,\mu},\Abs{-}_{\check H}}
      \iso
      \check H\Gamma\paren{Z,\check S}.
    \end{equation*}
  \item
    The \defined{extension map}
    $\ext \co \check H\Gamma\paren{Z,\check S} \to \dom(\Dmax)$ is obtained as the composition of the following maps
    \begin{align*}
      \check H\Gamma\paren{Z,\check S}
      \iso
      \bigoplus_{(-1/2,\mu) \in \check\Spectrum} \paren{V_{-1/2,\mu},\Abs{-}_{\check H}}
      \xrightarrow{(\ext_\mu)}
      &\bigoplus_{(-1/2,\mu) \in \check\Spectrum}
      \dom(\mathringDmax^{-1/2,\mu}) \\
      &\incl
      \dom(\mathringDmax)
      \xrightarrow{\chi(r)}
      \dom(\Dmax).
      \qedhere
    \end{align*}
  \end{enumerate}
\end{definition}

\begin{theorem}
  \label{Thm_ResidueAndExtensionMap}
  The following hold:
  \begin{enumerate}
  \item
    \begin{enumerate}
    \item
      \label{Thm_ResidueAndExtensionMap_Symplectomorphism}
      The residue map is an isomorphism of symplectic Hilbert spaces
      \begin{equation*}
        \res \co
        \paren{\GelfandRobbinQuotient,G}
        \iso  
        \paren{\check H\Gamma\paren{Z,\check S},\check\Omega}.
      \end{equation*}
    \item
      \label{Thm_ResidueAndExtensionMap_Characterisation}
      The subspace $r^{-1/2}\Gamma\paren{\hat{X},\hat{S}\otimes\hat{\fl}} \cap \dom(\Dmax)$ is dense in $\dom(\Dmax)$;
      hence: the residue map is uniquely determined by
      \begin{equation*}
        \pi^*\res[r^{-1/2}\phi] = \phi|_{\del\hat{X}}
      \end{equation*}   
      for every $r^{-1/2}\phi \in r^{-1/2}\Gamma\paren{\hat{X},\hat{S}\otimes\hat{\fl}} \cap \dom(\Dmax)$.
    \end{enumerate}
  \item
    \begin{enumerate}
    \item
      \label{Thm_ResidueAndExtensionMap_RightInverse}
      The extension map $\ext \co \check H\Gamma\paren{Z,\check S} \to \dom(\Dmax)$ is a right-inverse of $\res [-]$.
    \item
      \label{Thm_ResidueAndExtensionMap_L2}
      The extension map extends to a bounded linear map
      \begin{equation*}
        \ext \co H^{-1/2}\Gamma\paren{Z,\check S} \to L^2\Gamma\paren{X\setminus Z, S\otimes\fl}.
      \end{equation*}
    \end{enumerate}
  \end{enumerate}    
\end{theorem}

\begin{proof}%[Proof of \autoref{Thm_ResidueAndExtensionMap}]
  \autoref{Thm_ResidueAndExtensionMap_Symplectomorphism} is an immediate consequence of
  \autoref{Prop_SymplectomorphismOfGelfandRobbinQuotients},  
  \autoref{Cor_SpectralDecomposition_GelfandRobbinQuotient},
  \autoref{Prop_Identification_Symplectic},
  \autoref{Prop_Identification_UniformNorms}, and
  \autoref{Prop_Reassembly}.
  
  \autoref{Thm_ResidueAndExtensionMap_Characterisation} is a consequence of \autoref{Lem_DomDMaxVsDomMathringDMax} and \autoref{Prop_SpectralDecomposition_DomDMax}.
  
  \autoref{Thm_ResidueAndExtensionMap_RightInverse} holds by construction and   
  \autoref{Thm_ResidueAndExtensionMap_L2} follows from \autoref{Lem_UniformEstimatesForRestrictionAndExtension}~\autoref{Lem_UniformEstimatesForRestrictionAndExtension_2}.
\end{proof}

\begin{remark}
  \label{Rmk_GeometricRealisation_Chirality}
  If $\epsilon$ is a chirality operator,
  then
  $\check S$ orthogonally decomposes as $\check S = \check S^+ \oplus \check S^-$,
  $\check S^\pm \subset \check S$ are Lagrangian subbundles,
  $A$ preserves this splitting,
  $\check H\Gamma\paren{Z,\check S}$ orthogonally decomposes as $\check H\Gamma\paren{Z,\check S} = \check H\Gamma\paren{Z,\check S^+} \oplus \check H\Gamma\paren{Z,\check S^-}$,
  and
  the residue map restricts to isomorphisms
  \begin{equation*}
    \res \co \GelfandRobbinQuotient^\pm \iso \check H\Gamma\paren{Z,\check S^\pm}.
    \qedhere
  \end{equation*}  
\end{remark}

%%% Local Variables:
%%% mode: latex
%%% TeX-master: "DiracOperatorsTwistedByRamifiedLineBundles"
%%% ispell-local-dictionary: "british"
%%% End:

%% file: SpectralAndLocalResidueConditions.tex
\subsection{Spectral and local residue conditions}
\label{Sec_SpectralAndLocalResidueConditions}

\autoref{Thm_ResidueAndExtensionMap} makes it possible to define a wider variety of residue conditions than those considered in \autoref{Sec_GelfandRobbinQuotient}.
Here are some examples.

\begin{example}
  \label{Ex_APSResidueCondition}
  The \defined{APS residue condition} is defined by
  \begin{equation*}
    \ResidueCondition_\APS
    \coloneq
    \one_{(-\infty,0)}(A)H^{1/2}\Gamma\paren{Z,\check S}
    \subset
    \check H\Gamma\paren{Z,\check S};
  \end{equation*}
  cf.~\citet[(2.3)]{Atiyah1975a}.
\end{example}

\begin{prop}[Criterion for left semi-Fredholmness]
  \label{Prop_BH-1/2CompactLeftSemiFredholm}
  Let $\ResidueCondition \subset \check H\Gamma(Z,\check S)$ be a residue condition.
  If $\ResidueCondition \incl H^{-1/2}\Gamma\paren{Z, \check S}$ is compact,
  then
  $\dom(D_\ResidueCondition) \incl L^2\Gamma\paren{X\setminus Z, S\otimes\fl}$ is compact and $D_\ResidueCondition$ is left semi-Fredholm.
\end{prop}

\begin{proof}%[Proof of \autoref{Prop_BH-1/2CompactLeftSemiFredholm}]  
  By \autoref{Lem_BorderlineHardyInequality_Consequences}~\autoref{Lem_BorderlineHardyInequality_Consequences_H1L2Compact}
  and the assumption,
  the composition
  \begin{align*}
    \dom(D_\ResidueCondition) &\to H^1\Gamma\paren{X\setminus Z,S \otimes \fl} \oplus \ResidueCondition \incl L^2\Gamma\paren{X\setminus Z,S \otimes \fl} \oplus H^{-1/2}\Gamma\paren{Z, \check S} \\
    \phi &\mapsto (\phi-\ext\,\res[\phi],\res[\phi])
  \end{align*}
  is compact.
  Therefore,
  by \autoref{Thm_ResidueAndExtensionMap}~\autoref{Thm_ResidueAndExtensionMap_L2},
  $\dom(D_\ResidueCondition) \incl L^2\Gamma\paren{X\setminus Z, S\otimes\fl}$ is compact.
  Since for every $\phi \in \dom(D_\ResidueCondition)$
  \begin{align*}
    \Abs{\phi}_D
    \lesssim
      \Abs{D\phi}_{L^2}
      + \Abs{\phi}_{L^2},
  \end{align*}
  $D_\ResidueCondition$ is left semi-Fredholm.
\end{proof}

\begin{example}
  \label{Ex_APSResidueCondition_Fredholm}
  Since $\ResidueCondition_\APS^G = \ResidueCondition_\APS \oplus \ker A$,  
  by \autoref{Prop_FredholmResidueConditionComparison} and \autoref{Prop_BH-1/2CompactLeftSemiFredholm},
  $D_{\ResidueCondition_\APS}$ is Fredholm of index $-\frac12 \dim\ker A$.  
  In particular, $\dim\ker A$ is even and inherits a symplectic structure from $G$.
  If $L \subset \ker A$ is Lagrangian,
  then $\ResidueCondition_\APS  \oplus L \incl H^{-1/2}\Gamma\paren{Z,\check S}$ is compact and Lagrangian.
  In particular,
  \autoref{Prop_SpectralTheory} applies.
\end{example}

\begin{definition}
  \label{Def_LocalBoundaryCondition}  
  Let $V \subset \check S$ be a subbundle.
  The \defined{local residue condition} associated with $V$ is
  \begin{equation*}
    \ResidueCondition_V
    \coloneq
    \check H\Gamma(Z,V) \subset \check H\Gamma\paren{Z,\check S}.
    \qedhere   
  \end{equation*}
\end{definition}

\begin{prop}%[adjoint local boundary condition]
  \label{Prop_AdjointLocalBoundaryCondition}
  Let $V \subset \check S$ be a subbundle.
  If $V^{\check\Omega}$ denotes the symplectic complement of $V \subset \check S$,
  then
  \begin{equation*}
    \ResidueCondition_V^G = \ResidueCondition_{V^{\check\Omega}}.
  \end{equation*}
\end{prop}

\begin{proof}%[Proof of \autoref{Prop_AdjointLocalBoundaryCondition}]
  This is an immediate consequence of \autoref{Prop_AdjointExtension} and \autoref{Thm_ResidueAndExtensionMap}.
\end{proof}

\begin{example}%[Neumann and Dirichlet boundary conditions for the Hodge--de Rham operator]
  \label{Ex_NeumannDirichlet}
  Consider the Dirac bundle $(S,\gamma,\nabla)$ corresponding to the Hodge--de Rham operator $\rd + \rd^*$;
  that is: $S \coloneq \Wedge T^*X$ with $\gamma(v)\phi \coloneq v^\flat \wedge \phi - i_v\phi$.
  Decompose
  \begin{equation*}
    \check S = \check S_N \oplus \check S_D
  \end{equation*}
  with
  $\check S_N \coloneq S_N \otimes_\C NZ^{-1/2}$ and
  $\check S_D \coloneq S_D \otimes_\C NZ^{-1/2}$, and
  \begin{equation*}
    S_N
    \coloneq \paren{\R \oplus \Wedge^2 N^*Z} \otimes \Wedge T^*Z
    \qandq
    S_D
    \coloneq N^*Z \otimes \Wedge T^*Z .
  \end{equation*}
  The corresponding residue conditions are Lagrangian.
\end{example}

\begin{example}
  \label{Ex_MITBag}
  Assume that $(S,\gamma,\nabla,\Tor)$ is a complex Dirac bundle with skew torsion.
  The \defined{MIT bag residue conditions} are the local residue conditions $\ResidueCondition_\bag^\pm$ arising from the decomposition
  \begin{equation*}
    \check S = \check S_\bag^+ \oplus \check S_\bag^-
    \qwithq
    \check S_\bag^\pm \coloneq \ker \paren{\one\mp iJ};
  \end{equation*}
  cf.~\cite{Johnson1975:MITBag}.  
  Since $J$ is an isometry, and $i$ and $J$ commute,
  $ \check S_\bag^+$ and $ \check S_\bag^-$ are perpendicular complex subbundles.
  In particular, $(\ResidueCondition_\bag^\pm)^G = \ResidueCondition_\bag^\mp$.
\end{example}

By the usual argument,
the above leads to the following variation of the bordism theorem;
see, e.g., \cite[§8.5]{BaerBallmann2012:BVP}.
The significance of this result in the present context, however, is somewhat mysterious.

\begin{prop}
  \label{Prop_BordismTheorem}
  Assume the situation of \autoref{Ex_MITBag}.
  The components $A^\pm$ of $A$ in the decomposition
  \begin{equation*}
    A\eqcolon
    \begin{pmatrix}
      0 &  A^-\\
      A^+ & 0 
    \end{pmatrix}
  \end{equation*}
  satisfy
  \begin{equation*}
    \ind A^\pm=0.
  \end{equation*}
\end{prop}

The proof relies on the following observation.

\begin{lemma}
  \label{Lem_MITBag}
  ~
  \begin{enumerate}
  \item
    \label{Lem_MITBag_APS}
    $\ResidueCondition_\bag^\pm = \ker A^\pm \oplus (\one\mp  i J)\ResidueCondition_\APS$.
  \item
    \label{Lem_MITBag_Ker}
    $\ker D_{\ResidueCondition_\bag^\pm} = \ker \Dmin$.
  \end{enumerate}
\end{lemma}

\begin{proof}%[Proof of \autoref{Lem_MITBag}]
  Since $A$ and $J$ anti-commute,
  if $\phi \in \ResidueCondition_\bag^\pm$,
  then 
  \begin{equation*}
    \one_{(0,\infty)}(A)\phi
    =
    \pm \one_{(0,\infty)}(A) iJ\phi
    =
    \mp iJ \one_{(-\infty,0)}(A)\phi.
  \end{equation*}
  This implies \autoref{Lem_MITBag_APS};
  in particular: $\ResidueCondition_\bag^\pm = H^{1/2}\Gamma\paren{Z,\check S_\bag^\pm}$.

  The argument for \autoref{Lem_MITBag_Ker} is as in \autoref{Ex_AbstractMITBag_Fredholm};
  indeed:
  for every $\phi \in \ker D_{\ResidueCondition_\bag^\pm}$
  \begin{equation*}
    0
    = 2\Inner{D\phi,i\phi}_{L^2}
    = -2\pi\Inner{J \res \phi,i \res \phi}
    = \mp 2\pi \Abs{\res \phi}_{L^2}^2.
    \qedhere
  \end{equation*}
\end{proof}

\begin{proof}[Proof of \autoref{Prop_BordismTheorem}]
  The following proof is essentially identical to the one presented in \cite[§8.5]{BaerBallmann2012:BVP}.
  Since $(A^+)^*=A^-$,
  \begin{equation*}
    -\ind A^-=\ind A^+=\dim \ker A^+-\dim \ker A^-.
  \end{equation*}
  For every $t\in [0,1]$,
  set
  \begin{equation*} 
    \ResidueCondition_t^\pm
    \coloneq \ker A^\pm \oplus (\one\mp t i J)\ResidueCondition_\APS.
  \end{equation*}
  Since $(\ResidueCondition_t^+)^G=\ResidueCondition_t^-$, by \autoref{Prop_BH-1/2CompactLeftSemiFredholm}, $D_{\ResidueCondition_t^\pm}$ is Fredholm.
  Moreover, by \autoref{Prop_FredholmResidueConditionComparison},
  \begin{equation*}
    -\ind D_{\ResidueCondition_t^\mp}=\ind D_{\ResidueCondition_t^\pm}=\ind D_{ (\one\mp t i J)\ResidueCondition_\APS}+ \dim \ker A^\pm.
  \end{equation*}
  In particular,
  \begin{equation*}
    \ind A^+ = 2\ind D_{\ResidueCondition_0^+}.
  \end{equation*}
  Therefore, it remains to prove that $\ind D_{\ResidueCondition_0^+} = 0$.
  By \autoref{Prop_VariationOfBoundaryConditions},
  it suffices to prove that $\ind D_{\ResidueCondition_1^+} =  0$.
  By \autoref{Lem_MITBag},
  $\ResidueCondition_1^+ = \ResidueCondition_\bag^+$ and $\ind D_{\ResidueCondition_\bag^+}= 0$.
\end{proof}

%%% Local Variables:
%%% mode: latex
%%% TeX-master: "DiracOperatorsTwistedByRamifiedLineBundles"
%%% ispell-local-dictionary: "british"
%%% End:

%% file: RegularityTheory.tex
\section{Regularity theory}
\label{Sec_RegularityTheory}

This section continues to assume \autoref{Hyp_CodimensionTwoCooriented} throughout.
The geometric realisation $\check H\Gamma\paren{Z,\check S}$ of $\GelfandRobbinQuotient$ developed in \autoref{Sec_GeometricRealisation} leads to the $L^2$ regularity theory laid out in the following.

\input{SobolevSpaces_1}
\input{EllipticRegularityAndEstimates}
\input{FredholmExtensionsInHigherRegularity}
\input{SobolevSpaces_2}
\input{SobolevSpaces_3}
\input{SymbolicRegularityCriterion}

%%% Local Variables:
%%% mode: latex
%%% TeX-master: "DiracOperatorsTwistedByRamifiedLineBundles"
%%% ispell-local-dictionary: "british"
%%% End:

%% file: SobolevSpaces_1.tex
\subsection{Adapted Sobolev spaces, I: definition}
\label{Sec_SobolevSpaces_1}

Here is the scale of Sobolev spaces for which the regularity theory is developed.

\begin{definition}[conormal differential operators]
  \label{Def_DifferentialOperators}
  Denote by
  $\DiffOp^\bullet(S\otimes\fl)$ the $\N_0$--filtered ring of differential operators acting on $S\otimes\fl$.
  \begin{enumerate}
  \item
    A vector field $v \in \Vect(\hat X)$ is \defined{conormal} if $v|_{\del \hat X} \in \Vect(\del \hat X)$.
    Denote the subspace of conormal vector fields by $\Vect_b(\hat X)$.
  \item
    The space $\DiffOp_b^\bullet(S \otimes \fl) \subset \DiffOp^\bullet(S \otimes \fl)$ of \defined{conormal} differential operators is the filtered subring generated by $\Gamma\paren{\hat X,\End(\hat S \otimes \hat\fl)}$ and
    differential operators of the form $\nabla_v$ with $v \in \Vect_b(\hat X)$.    
    \qedhere
  \end{enumerate}
\end{definition}

\begin{definition}[conormal and adapted Sobolev spaces]
  \label{Def_ConormalAdaptedSobolevSpaces}
  Let $k \in \N_0$.
  \begin{enumerate}
  \item[(b)]
    The \defined{conormal Sobolev space} $H_b^k\Gamma\paren{X\setminus Z, S\otimes\fl}$ is defined by
    \begin{equation*}
      H_b^k\Gamma\paren{X\setminus Z, S\otimes\fl}
      \coloneq
      \set*{
        \phi \in H_\loc^k\Gamma\paren{X\setminus Z, S\otimes\fl}
        :
        \begin{aligned}
          & P \phi \in L^2\Gamma\paren{X\setminus Z, S\otimes\fl} \text{ for} \\
          & \text{every }
            P \in \DiffOp_b^k(S\otimes\fl)
        \end{aligned}
      }.
    \end{equation*}
    \emph{Choose} a finite subset $\sP_b^k \subset \DiffOp_b^k(S\otimes\fl)$ which spans $\DiffOp_b^k(S\otimes\fl)$ over $\Gamma\paren{\hat X,\End(\hat S\otimes\hat\fl)}$.
    Define the norm $\Abs{-}_{H_b^k} \co  H_b^k\Gamma\paren{X\setminus Z, S\otimes\fl} \to [0,\infty)$ by
    \begin{equation*}
      \Abs{\phi}_{H_b^k}^2
      \coloneq
      \sum_{P \in \sP_b^k} \Abs{P\phi}_{L^2}^2.
    \end{equation*}
  \item[(a)]
    The \defined{adapted Sobolev space} $H_\adm^{k+1}\Gamma\paren{X\setminus Z, S\otimes\fl}$ is defined by
    \begin{equation*}
      H_\adm^{k+1}\Gamma\paren{X\setminus Z, S\otimes\fl}
      \coloneq
      \set*{
        \phi \in H_\loc^{k+1}\Gamma\paren{X\setminus Z, S\otimes\fl}
        :
        \begin{aligned}
          &\phi \in H_b^{k+1}\Gamma\paren{X\setminus Z, S\otimes\fl} ~\text{and} \\
          &D\phi \in H_b^k\Gamma\paren{X\setminus Z, S\otimes\fl} 
        \end{aligned}
      }.
    \end{equation*}
    Define the norm $\Abs{-}_{H_\adm^{k+1}} \co  H_\adm^{k+1}\Gamma\paren{X\setminus Z, S\otimes\fl} \to [0,\infty)$ by
    \begin{equation*}
      \Abs{\phi}_{H_\adm^{k+1}}
      \coloneq
      \Abs{\phi}_{H_b^{k+1}} + \Abs{D\phi}_{H_b^k}.
      \qedhere
    \end{equation*}
  \end{enumerate}
\end{definition}

$\paren{H_b^k\Gamma\paren{X\setminus Z, S\otimes\fl},\Abs{-}_{H_b^k}}$ and
$\paren{H_\adm^{k+1}\Gamma\paren{X\setminus Z, S\otimes\fl},\Abs{-}_{H_\adm^{k+1}}}$ are Hilbert spaces.
Evidently,
different choices of $\sP_b^k$ lead to equivalent norms.
The following discussion leads to a particularly convenient choice.

\begin{definition}[Convenient vector fields]
  \label{Def_ConvenientVectorFields}
  ~
  \begin{enumerate}
  \item
    Denote by $\Vect_c(X\setminus Z) \subset \Vect_b(\hat X)$ the subspace of vector fields supported in $X \setminus Z \subset \hat X$.
  \item
    Denote by $\Vect_{b;c}(\hat U) \subset \Vect_b(\hat X)$ the subspace of vector fields supported in $\hat U \subset \hat X$.
    For $v \in \Vect_{b;c}(\hat U)$,
    $\mathring\nabla_v \in \DiffOp_b^1(S\otimes\fl)$.
  \item
    Denote by $\Vect_{b;c,0}(\hat U)$ the subspace of those $v \in \Vect_{b;c}(\hat U)$ which are $\U(1)$--invariant on $\del \hat X$; that is: $[\del_\alpha,v|_{\del\hat X}] = 0$.
    \qedhere
  \end{enumerate}
\end{definition}

\begin{remark}
  \label{Rmk_ConvenientVectorFields}
  $\Vect_{b;c}(\hat U)$ is generated by
  $\chi(r) \del_\alpha$, $\chi(r) r\del_r$, vector fields of the form $\chi(r) v$ where $v$ is lifted from $Z$, and vector fields vanishing near $Z$.
  In particular, these are elements of $\Vect_{b;c,0}(\hat U)$.
\end{remark}

\begin{lemma}[Commutation relations]
  \label{Lem_CommutationRelations}
  The following commutation relations hold:
  \begin{enumerate}
  \item
    \label{Lem_CommutationRelations_1}
    For every $v \in \Vect_c(X\setminus Z)$, $k \in \N$, and $P \in \DiffOp_b^k(S\otimes\fl)$
    \begin{equation*}
      [\nabla_v,P] \in \DiffOp_b^k(S\otimes\fl).
    \end{equation*}    
  \item    
    \label{Lem_CommutationRelations_2}
    For every $v \in \Vect_{b;c}(\hat U)$, $k \in \N$, and $P \in \DiffOp_b^k(S\otimes\fl)$
    \begin{equation*}
      [\mathring\nabla_v,P] \in \DiffOp_b^k(S\otimes\fl).
    \end{equation*}
  \item
    \label{Lem_CommutationRelations_3}
    For every $v \in \Vect_{b;c,0}(\hat U)$
    \begin{equation*}
      [\mathring\nabla_v,\mathring D], [\mathring\nabla_v,D]
      \in \DiffOp_b^1(S\otimes\fl) + \DiffOp_b^0(S\otimes\fl) D.
    \end{equation*}
  \end{enumerate}
\end{lemma}

\begin{proof}%[Proof of \autoref{Lem_CommutationRelations}]
  If $v,w \in \Vect(X\setminus Z)$ and $T \in \Gamma\paren{X\setminus Z,S\otimes\fl}$,
  then
  \begin{equation*}
    [\nabla_v,T] = \nabla_vT 
    \qandq
    [\nabla_v,\nabla_w] = \nabla_{[v,w]} + F_\nabla(v,w).
  \end{equation*}
  Therefore, for every $k \in \N$ and $P \in \DiffOp^k(S\otimes \fl)$, $[\nabla_v,P] \in \DiffOp^k(S\otimes\fl)$.

  If $v \in \Vect_c(X\setminus Z)$, $k \in \N$, and $P \in \DiffOp_b^k(S\otimes\fl)$,
  then $\supp([\nabla_v,P]) \subset X \setminus Z$;
  therefore and by the above observation,
  $[\nabla_v,P] \in \DiffOp_b^k(S\otimes\fl)$.
  This proves \autoref{Lem_CommutationRelations_1}.

  \autoref{Lem_CommutationRelations_2} is immediate from the above observation. 

  Let $v \in \Vect_{b;c,0}(\hat U)$.
  By \autoref{Prop_ModelOperatorEstimate},
  \begin{equation*}
    D - \chi(r) \mathring D \in \DiffOp_b^1(S\otimes\fl)
  \end{equation*}
  Therefore, it suffices to prove that
  \begin{equation*}
    [\mathring\nabla_v,\chi(r) \mathring D] \in \DiffOp_b^1(S\otimes\fl) + \DiffOp_b^0(S\otimes\fl) \chi(r) \mathring D.
  \end{equation*}
  By direct computation,
  \begin{equation*}
    [\mathring\nabla_{r\del_r}, \mathring\nabla_{\del_r} - r^{-1}I\mathring\nabla_{\del_\alpha}]
    =
    -\paren{\mathring\nabla_{\del_r} - r^{-1}I\mathring\nabla_{\del_\alpha}}
    \qandq
    [\mathring\nabla_{\del_\alpha}, \mathring\nabla_{\del_r} - r^{-1}I\mathring\nabla_{\del_\alpha}]
    =
    0;
  \end{equation*}
  moreover, if $v$ is the lift of a vector field along $Z$,
  then
  \begin{equation*}
    [\mathring\nabla_v, \mathring\nabla_{\del_r} - r^{-1}I\mathring\nabla_{\del_\alpha}]
    =
    0.
  \end{equation*}
  By \autoref{Rmk_ModelDiracOperator}, \autoref{Rmk_ConvenientVectorFields} and since $\chi(r) D_Z \in \DiffOp_b^1(S\otimes\fl)$ and $\chi(r) J \in \DiffOp_b^0(S\otimes\fl)$,
  this implies \autoref{Lem_CommutationRelations_3}.
\end{proof}

\begin{remark}
  The reader should be warned that a number of variations of \autoref{Lem_CommutationRelations} that one might naively guess are false.
  In particular,
  if $T \in \Gamma\paren{\hat X,\End(\hat S \otimes \hat\fl)}$,
  then $[T,D]$ need not be admissible:
  e.g., $[\gamma(v),D] = 2\gamma(v) D - \sum_{i=1}^n \gamma(e_i)\gamma(\nabla^T_{e_i}v) + 2\nabla_v$.
\end{remark}

\begin{cor}[Convenient choice of $\sP_b^k$]
  \label{Cor_ChoiceOfGenerators}
  Set $\sP_b^0 \coloneq \set{\id_{S\otimes\fl}}$.
  For every $k \in \N$
  there is a finite subset $\sP_b^k \subset \DiffOp_b^k(S\otimes\fl)$ which spans $\DiffOp_b^k(S\otimes\fl)$ over $\Gamma\paren{\hat X,\End(\hat S\otimes\hat\fl)}$
  such that
  $\sP_b^{k-1} \subset \sP_b^k$ and
  every $P \in \sP_b^k\setminus\sP_b^{k-1}$ is of the form
  \begin{equation*}
    P = \nabla_{u_1} \cdots \nabla_{u_\ell} \mathring\nabla_{v_1} \cdots \mathring\nabla_{v_{k-\ell}} 
  \end{equation*}
  with
  \begin{equation*}
    u_1,\ldots,u_\ell \in \Vect_c(X\setminus Z)
    \qandq
    v_1,\ldots,v_{k-\ell} \in \Vect_{b;c,0}(\hat U).
    \qedhere
  \end{equation*}
\end{cor}

\emph{Henceforth},
for every $k \in \N$,
$\sP_b^k$ is assumed to be chosen as in \autoref{Cor_ChoiceOfGenerators}.
This leads to the following convenient description of the above spaces based on the spectral decomposition from \autoref{Sec_SpectralDecomposition}.

\begin{prop}
  \label{Prop_HbkHakSpectralDescription}
  Let $k \in \N_0$ and $\phi \in L^2\Gamma\paren{X \setminus Z,S\otimes \fl}$.
  \label{Prop_HbkHakSpectralDescription_SuppU}
  If $\supp(\phi) \subset r^{-1}((0,3/4]) \subset U$ and $\phi$ is decomposed into
  \begin{equation*}
    L^2\Gamma\paren{U\setminus Z, \mathring S \otimes \mathring\fl} \ni
    \phi = \sum_{(\lambda,\mu) \in \Spectrum} \phi_{\lambda,\mu} \in  
    \bigoplus_{(\lambda,\mu) \in \Spectrum}
    L^2\paren{(0,1),r\rd r;E_{\lambda,\mu}}
  \end{equation*}
  as in \autoref{Sec_SpectralDecomposition},
  then:
  \begin{enumerate}
  \item
    \label{Prop_HbkHakSpectralDescription_SuppU_Hbk}
    $\phi \in H_b^k\Gamma\paren{X\setminus Z, S\otimes\fl}$ if and only if
    \begin{equation*}
      \sum_{\ell=0}^k \sum_{(\lambda,\mu) \in \Spectrum}
      \int_0^1 \paren{\bracket{\lambda}+\bracket{\mu}}^{2(k-\ell)} \abs{ (r\del_r)^\ell \phi_{\lambda,\mu}(r)}^2 \,r\rd r < \infty;
    \end{equation*}
    moreover, the above is uniformly equivalent to $\Abs{\phi}_{H_b^k}^2$.
  \item
    \label{Prop_HbkHakSpectralDescription_SuppU_Hak}
    $\phi \in H_a^{k+1}\Gamma\paren{X\setminus Z, S\otimes\fl}$ if and only if
    \begin{equation*}
      \sum_{m=0}^1 \sum_{\ell=0}^{k+1-m} \sum_{(\lambda,\mu) \in \Spectrum}
      \int_0^1 \paren{\bracket{\lambda}+\bracket{\mu}}^{2(k+1-m-\ell)} \abs{ (r\del_r)^\ell  (\del_r-\tfrac {\lambda} {r})^m\phi_{\lambda,\mu}(r)}^2 \,r\rd r < \infty;
    \end{equation*}
    moreover, the above is uniformly equivalent to $\Abs{\phi}_{H_a^{k+1}}^2$.      
  \end{enumerate}
\end{prop}

The proof uses the following amplification of \autoref{Lem_DomDMaxVsDomMathringDMax}.

\begin{lemma}
  \label{Lem_DHbkVsMathringDHbk}
  Let $k \in \N_0$ and $\phi \in H_\loc^{k+1}\Gamma\paren{X\setminus Z,S\otimes\fl}\cap H_b^k\Gamma\paren{X\setminus Z,S\otimes\fl}$.
  If $\supp(\phi) \subset r^{-1}((0,3/4]) \subset U$, then $D \phi\in H_b^k\Gamma\paren{X\setminus Z,S\otimes\fl}$ if and only if $\mathring D \phi\in H_b^k\Gamma\paren{X\setminus Z,S\otimes\fl}$; moreover,  in either case
  \begin{equation*}
   \Abs{D\phi}_{H_b^k}+  \Abs{\phi}_{H_b^k}  \asymp_k \Abs{\mathring D\phi}_{H_b^k} + \Abs{\phi}_{H_b^k}.
  \end{equation*}
\end{lemma}

\begin{proof}
  By \autoref{Lem_DomDMaxVsDomMathringDMax}, the assertion holds for $k=0$.
  
  Assume that $k \in \N$ and that the assertion holds for $k-1$ instead of $k$.
  Let $\phi \in H_\loc^{k+1}\Gamma\paren{X\setminus Z,S\otimes\fl}\cap H_b^k\Gamma\paren{X\setminus Z,S\otimes\fl}$.
  For every $v \in \Vect_{b;c,0}(\hat U)$,
  by assumption and \autoref{Lem_CommutationRelations},
  \begin{align*}
    \Abs{\mathring\nabla_v D\phi}_{H_b^{k-1}}
    &\leq
      \Abs{D\mathring\nabla_v\phi}_{H_b^{k-1}}
      + \Abs{[\mathring\nabla_v,D\phi]}_{H_b^{k-1}}\\
    &\lesssim
      \Abs{D\mathring\nabla_v\phi}_{H_b^{k-1}}
      + \Abs{D\phi}_{H_b^{k-1}}
      + \Abs{\phi}_{H_b^k} \\
    &\lesssim_k
      \Abs{\mathring D \mathring\nabla_v\phi}_{H_b^{k-1}}
      + \Abs{\mathring D\phi}_{H_b^{k-1}}
      + \Abs{\phi}_{H_b^k} \\
    &\leq
      \Abs{\mathring D \phi}_{H_b^k}
      + \Abs{[\mathring D, \mathring\nabla_v]\phi}_{H_b^{k-1}}
      + \Abs{\phi}_{H_b^k}
      \lesssim
      \Abs{\mathring D \phi}_{H_b^k}
      + \Abs{\phi}_{H_b^k}. 
  \end{align*}
  Therefore,
  \begin{equation*}
    \Abs{D\phi}_{H_b^k} \lesssim_k \Abs{\mathring D\phi}_{H_b^k} + \Abs{\phi}_{H_b^k}.
  \end{equation*}
  The analogous inequality with the roles of $D$ and $\mathring D$ exchanged is proved similarly.
\end{proof}

\begin{proof}[Proof of \autoref{Prop_HbkHakSpectralDescription}]
  If $v$ is a vector field lifted from $Z$ to $F$,
  then for every $\phi \in E_{\lambda,\mu}$
  \begin{equation}
    \label{Eq_NablaVPsi}
    \Abs{\mathring\nabla_v\phi} \lesssim_v \bracket{\mu}\Abs{\phi}.
  \end{equation}
  This implies \autoref{Prop_HbkHakSpectralDescription_SuppU_Hbk}.  
  This combined with \autoref{Lem_DHbkVsMathringDHbk} implies \autoref{Prop_HbkHakSpectralDescription_SuppU_Hak}.
\end{proof}

%%% Local Variables:
%%% mode: latex
%%% TeX-master: "DiracOperatorsTwistedByRamifiedLineBundles"
%%% ispell-local-dictionary: "british"
%%% End:

%% file: EllipticRegularityAndEstimates.tex
\subsection{Elliptic regularity and estimates}
\label{Sec_EllipticRegularityAndEstimates}

Here is the fundamental regularity result.

\begin{theorem}[elliptic regularity and estimates, I]
  \label{Thm_EllipticRegularityAndEstimates_1}
  For every $k \in \N_0$
  \begin{equation*}
    H_\adm^{k+1}\Gamma\paren{X\setminus Z, S\otimes\fl}
    =
    \set*{
      \phi \in H_\loc^{k+1}\Gamma\paren{X\setminus Z, S\otimes\fl}
      :
      \begin{aligned}
        & \phi,D\phi \in H_b^k\Gamma\paren{X\setminus Z,S\otimes\fl} \\
        & \textnormal{and}~
          \res[\phi] \in H^{k+\frac12}\Gamma\paren{Z,\check S}
      \end{aligned}
    };
  \end{equation*}
  moreover:
  for every $\phi \in H_\adm^{k+1}\Gamma\paren{X\setminus Z, S\otimes\fl}$
  \begin{equation*}
    \Abs{\phi}_{H_\adm^{k+1}}
    \asymp_k
    \Abs{D\phi}_{H_b^k}
    + \Abs{\phi}_{H_b^k}
    + \Abs{\res [\phi]}_{H^{k+1/2}}.
  \end{equation*}
\end{theorem}

The proof relies on the following observations.

\begin{lemma}
  \label{Lem_ExtensionMap_Higher_Regularity}
  For every $k \in \N_0$
  the extension map $\ext \co \check H\Gamma\paren{Z,\check S} \to \dom(\Dmax)$ restricts to a bounded injective linear map with closed image
  \begin{equation*}
    \ext \co H^{k+1/2}\Gamma(Z,\check S) \to H_a^{k+1}\Gamma\paren{X\setminus Z,S\otimes\fl}.
  \end{equation*}
\end{lemma}

\begin{proof}%[Proof of \autoref{Lem_ExtensionMap_Higher_Regularity}]
  By  \autoref{Prop_HbkHakSpectralDescription},
  it suffices to prove that for every
  $(-1/2,\mu) \in \check\Spectrum$ and $\phi \in V_{-1/2,\mu} \subset \check H\Gamma\paren{Z,\check S}$
  \begin{align*}
    \sum_{m=0}^1 \sum_{\ell=0}^{k+1-m}
    \int_0^1 \bracket{\mu}^{2(k+1-m-\ell)} \abs{ (r\del_r)^\ell  (\del_r+\tfrac{1}{2r})^m \paren{\chi(r)r^{-1/2}e^{-\abs{\mu}r}\phi}}^2 \,r\rd r
    &\asymp_k
      \bracket{\mu}^{2k+1}\Abs{\phi}^2 \\
    &\asymp_k
      \Abs{\phi}_{H^{k+1/2}}^2.
  \end{align*}
  By direct computation,
  \begin{equation*}
    (r\del_r)^\ell  (\del_r+\tfrac{1}{2r})^m \paren{\chi(r)r^{-1/2}e^{-\abs{\mu}r}\phi}
    = f_{\ell,m}(r) r^{-1/2}e^{-\abs{\mu}r} \phi
  \end{equation*}
  with $f_{\ell,m}(r)$ recursively defined by
  \begin{equation*}
    f_{\ell,m}(r)
    \coloneq
    \begin{cases}
      \chi(r) & \text{if}~\ell=m=0 \\
      \chi'(r) - \abs{\mu}\chi(r) & \text{if}~ \ell = 0 ~\text{and}~ m = 1 \\
      rf_{\ell,m-1}'(r) - \paren*{\tfrac12+\abs{\mu}} f_{\ell,m-1}(r) & \text{if}~ \ell \geq 1.
    \end{cases}
  \end{equation*}
  A brief computation using scaling considerations shows that
  \begin{equation*}
    \int_0^1 f_{\ell,m}^2(r) e^{-2\abs{\mu}r} \,\rd r \lesssim_{\ell,m} \bracket{\mu}^{2(\ell+m)-1};
  \end{equation*}
  moreover,
  \begin{equation*}
    \int_0^1 f_{0,0}^2(r) e^{-2\abs{\mu}r} \,\rd r \gtrsim \bracket{\mu}^{-1}.
  \end{equation*}
  This proves the assertion.
\end{proof}

\begin{lemma}
  \label{Lem_ResidueMap_Higher_Regularity}
  For every $k \in \N_0$
  the residue map $\res \co \dom(\Dmax) \to \check H\Gamma\paren{Z,\check S}$ restricts to a bounded surjective linear map
  \begin{equation*}
    \res \co H_a^{k+1}\Gamma\paren{X\setminus Z,S\otimes\fl} \to H^{k+1/2}\Gamma\paren{Z,\check S}.
  \end{equation*}
\end{lemma}

\begin{proof}%[Proof of \autoref{Lem_ResidueMap_Higher_Regularity}]
  Let $\phi \in \chi(r)\dom\paren{\mathringDmax^{-1/2,\mu}}$ with $\supp \phi \subset r^{-1}((0,1/2])$.
  As in the proof of \autoref{Lem_UniformEstimatesForRestrictionAndExtension},
  \begin{equation*}
    \res_\mu([\phi])
    =
    - \int_0^1 \del_r (r^{1/2} e^{-\abs{\mu}r} \phi) \,\rd r 
    =
    - \int_0^1 r^{-1/2} e^{-\abs{\mu}r} \paren{\del_r + \tfrac{1}{2r} - \abs{\mu}} \phi \,r\rd r.
  \end{equation*}
  Since
  \begin{equation*}
    \int_0^1 e^{-2\abs{\mu}r} \,\rd r \asymp \bracket{\mu}^{-1},
  \end{equation*}
  by Cauchy--Schwarz 
  \begin{equation*}    
    \bracket{\mu}\Abs{\res_\mu([\phi])}^2
    \lesssim \int_0^1 \abs*{\paren{\del_r + \tfrac{1}{2r}}\phi}^2 + \bracket{\mu}^2\abs{\phi}^2 \,r\rd r.
  \end{equation*}
  In light of \autoref{Prop_HbkHakSpectralDescription} this implies the assertion.  
\end{proof}

\begin{lemma}
  \label{Lem_ResDelB}
  Let $k \in \N$,
  $\phi \in H_\loc^{k+1}\Gamma\paren{X\setminus Z,S\otimes\fl}$ with
  $\phi,D\phi \in H_b^k\Gamma\paren{X\setminus Z,S\otimes\fl}$,
  and $P \in \DiffOp_b^k(S\otimes\fl)$.
  If $\res(\phi) = 0$,
  then $P\phi \in \dom(\Dmax)$ and $\res(P\phi) = 0$.
\end{lemma}

\begin{proof}
  By induction, it suffices to prove this for $k = 1$. 
  In fact,
  it suffices to assume $\supp(\phi) \subset r^{-1}((0,1/2]) \subset U$ and consider $P$ of the form $\chi(r) \mathring \nabla_{r\del_r}$, $\chi(r) \mathring\nabla_{\del_\alpha}$ or $\chi(r) \mathring\nabla_v$ where $v$ is a vector field lifted from $Z$.
  In either case it suffices to prove that
  \begin{equation*}
    DP\phi, r^{-1}P\phi \in L^2\Gamma\paren{X\setminus Z,S\otimes\fl}.
  \end{equation*}
  By \autoref{Lem_CommutationRelations},
  $DP\phi = [D,P]\phi + PD\phi \in L^2\Gamma\paren{X\setminus Z,S\otimes\fl}$.   

  Since $H^1\Gamma\paren{X\setminus Z, S\otimes\fl} = \ker \res \subset H_a^1\Gamma\paren{X\setminus Z, S\otimes\fl}$ and by \autoref{Lem_BorderlineHardyInequality},  
  \begin{equation*}
    \chi(r) r^{-1}\phi
    \in
    L^2\Gamma\paren{X\setminus Z,S\otimes\fl}
    \qandq
    \nabla \phi
    \in
    L^2\Omega^1\paren{X\setminus Z,S\otimes\fl}.
  \end{equation*}  
  This immediately implies the assertion for $P = \chi(r)\mathring\nabla_{r\del_r}$ and $P = \chi(r)\mathring\nabla_{\del_\alpha}$ because $\mathring g(\del_r,\del_r) = r^{-2} \mathring g(\del_\alpha,\del_\alpha) = 1$.
  
  It remains to deal with $P = \chi(r)\mathring\nabla_v$ where $v$ is a vector field lifted from $Z$.
  Since $[\mathring\nabla_{\del_\alpha},\mathring\nabla_v] = 0$,
  with respect to the decomposition
  \begin{equation*}
    L^2\Gamma\paren{U\setminus Z, \mathring S \otimes \mathring\fl}
    =
    \bigoplus_{(\lambda,\mu) \in \Spectrum}
    L^2\paren{(0,1),r\rd r;E_{\lambda,\mu}}
  \end{equation*}
  from \autoref{Sec_SpectralDecomposition}
  only those components of $\phi$ with $\lambda = -1/2$ contribute to $\res(\mathring\nabla_v \phi)$.
  Therefore, it can be assumed that
  \begin{equation*}
    \phi =
    \sum_{(-1/2,\mu) \in \check\Spectrum} \phi_{-1/2,\mu}
    \in  
    \bigoplus_{(-1/2,\mu) \in \check\Spectrum}
    L^2\paren{(0,1),r\rd r;V_{\lambda,\mu}}.
  \end{equation*}
  Since $\res(\phi) = 0$,
  for every $(-1/2,\mu) \in \check\Spectrum$
  \begin{equation*}
    \int_0^1 r^{-2} \Abs{\phi_{-1/2,\mu}}^2 \,r \rd r
    \lesssim
    \int_0^1 \Abs{\mathring D\phi_{-1/2,\mu}}^2 + \Abs{\phi_{-1/2,\mu}}^2 \,r \rd r.
  \end{equation*}
  By \autoref{Lem_DHbkVsMathringDHbk},
  $\phi,\mathring D\phi \in H_b^1\Gamma\paren{X\setminus Z, S\otimes \fl}$ and therefore by \autoref{Prop_HbkHakSpectralDescription},
  \begin{equation*}
    \sum_{(-1/2,\mu) \in \check\Spectrum}
    \int_0^1 r^{-2} \bracket{\mu}^2 \Abs{\phi_{-1/2,\mu}}^2 \,r \rd r
    \lesssim
    \sum_{(-1/2,\mu) \in \check\Spectrum}
    \int_0^1 \bracket{\mu}^2 \Abs{\mathring D\phi_{-1/2,\mu}}^2 + \bracket{\mu}^2 \Abs{\phi_{-1/2,\mu}}^2 \,r \rd r < \infty.
  \end{equation*}
  Finally,
  by \autoref{Eq_NablaVPsi},
  \begin{equation*}
    \Abs{r^{-1} \mathring\nabla_v \phi}_{L^2}^2
    \lesssim
    \sum_{(-1/2,\mu) \in \check\Spectrum}
    \int_0^1 r^{-2} \bracket{\mu}^2 \Abs{\phi_{-1/2,\mu}}^2 \,r \rd r
    <
    \infty.
    \qedhere
  \end{equation*}
\end{proof}

\begin{proof}[Proof of \autoref{Thm_EllipticRegularityAndEstimates_1}]
  Let $k \in \N_0$.
  By \autoref{Lem_ResidueMap_Higher_Regularity},
  it suffices to prove that for every
  $\phi \in H_\loc^{k+1}\Gamma\paren{X\setminus Z,S\otimes\fl}$
  with
  $\phi,D\phi \in H_b^k\Gamma\paren{X\setminus Z,S\otimes\fl}$
  and $\res[\phi] \in H^{k+1/2}\Gamma\paren{Z,\check S}$
  \begin{equation*}
    \Abs{\phi}_{H_b^{k+1}}
    \lesssim_k
    \Abs{D\phi}_{H_b^k}
    + \Abs{\phi}_{H_b^k}
    + \Abs{\res [\phi]}_{H^{k+1/2}}.
  \end{equation*}
  Since
  \begin{equation*}
    \phi = (\phi - \ext\res[\phi]) + \ext\res[\phi]
  \end{equation*}
  and by \autoref{Lem_ExtensionMap_Higher_Regularity},
  it suffices to prove the above assuming $\res[\phi] = 0$.

  Since $H^1\Gamma\paren{X\setminus Z, S\otimes\fl} \incl H_a^1\Gamma\paren{X\setminus Z, S\otimes\fl}$ and by \autoref{Eq_D_EllipticEstimate},
  the assertion holds for $k = 0$.  
  Suppose that $k \in \N$.
  By \autoref{Lem_ResDelB},
  for every $P \in \sP_b^k$,
  $\res(P\phi) = 0$.
  Therefore, by the preceding and \autoref{Lem_CommutationRelations},
  \begin{equation*}    
    \Abs{P\phi}_{H_b^1}
    \lesssim
    \Abs{DP\phi}_{L^2} + \Abs{P\phi}_{L^2}
    \lesssim_P
    \Abs{D\phi}_{H_b^k} + \Abs{\phi}_{H_b^k}.
  \end{equation*}
  This implies the assertion.
\end{proof}

For suitable residue conditions $\ResidueCondition \subset \check H\Gamma\paren{Z,\check S}$,
the term $\res[\phi]$ in \autoref{Thm_EllipticRegularityAndEstimates_1} can be absorbed provided $\res[\phi] \in \ResidueCondition$.

\begin{definition}%[$(k+\tfrac12)$--regular residue condition]
  \label{Def_K+1/2Regular}
  Let $\ResidueCondition \subset \check H\Gamma\paren{Z,\check S}$ be a residue condition.
  \begin{enumerate}
  \item
    Let $k \in \N_0$.
    $\ResidueCondition$ is \defined{$(k+\tfrac12)$--regular} if for every $\phi \in \ResidueCondition$
    \begin{equation*}
      \Abs{\phi}_{H^{k+1/2}}
      \lesssim_{\ResidueCondition,k}
      \Abs{\one_{(-\infty,0)}(A)\phi}_{H^{k+1/2}}
      +
      \Abs{\phi}_{\check H}.
    \end{equation*}
  \item
    $\ResidueCondition$ is \defined{$\infty$--regular} if it is $(k+1/2)$--regular for every $k \in \N_0$.
    \qedhere
  \end{enumerate}
\end{definition}

\begin{example}
  \label{Ex_APSResidueCondition_InftyRegular}
  The APS residue condition $\ResidueCondition_\APS$ defined in \autoref{Ex_APSResidueCondition} is $\infty$--regular.  
\end{example}

\begin{theorem}[elliptic regularity and estimates, II]
  \label{Thm_EllipticRegularityAndEstimates_2}
  Let $k \in \N_0$.
  Let $\ResidueCondition$ be a $(k+\frac12)$--regular residue condition.
  If $\phi \in H_\loc^{k+1}\Gamma\paren{X\setminus Z, S\otimes\fl}$ satisfies $\phi, D\phi \in H_b^k\Gamma\paren{X\setminus Z,S\otimes\fl}$ and $\res[\phi] \in \ResidueCondition$,
  then $\phi \in H_a^{k+1}\Gamma\paren{X\setminus Z, S\otimes\fl}$ and
  \begin{equation*}
    \Abs{\phi}_{H_\adm^{k+1}}
    \asymp_{\ResidueCondition,k}
    \Abs{D\phi}_{H_b^k}
    + \Abs{\phi}_{L^2}.
  \end{equation*}
\end{theorem}

The proof requires the following preparation.

\begin{lemma}
  \label{Lem_Regularity_RestrictionMap_NegativePart}
  For every $k \in \N_0$ and $\phi\in \dom(\Dmax)$
  \begin{equation*}
    \Abs{\one_{(-\infty,0)}(A)\res[\phi]}_{H^{k+1/2}}
    \lesssim_k
    \Abs{D\phi}_{H_b^k}
    +
    \Abs{\phi}_{H_b^k}.
  \end{equation*}
\end{lemma}

\begin{proof}
  It suffices to prove this for $\phi \in \chi(r)\dom(\mathringDmax^{-1/2,\mu})$.
  In view of \autoref{Prop_HbkHakSpectralDescription} it suffices to prove this for $k = 0$.
  This case, however, is established in \autoref{Lem_UniformEstimatesForRestrictionAndExtension}.
\end{proof}

\begin{proof}[Proof of \autoref{Thm_EllipticRegularityAndEstimates_2}]  
  By \autoref{Lem_Regularity_RestrictionMap_NegativePart} and since $\ResidueCondition$ is $(k+\frac12)$--regular,
  for every $\phi \in \dom(D_\ResidueCondition)$
  \begin{equation*}
    \Abs{\res[\phi]}_{H^{k+1/2}}
    \lesssim_{\ResidueCondition,k}
    \Abs{\one_{(-\infty,0)}(A)\res[\phi]}_{H^{k+1/2}}
    +
    \Abs{\res[\phi]}_{\check H}\lesssim _k \Abs{D\phi}_{H_b^k}
    +
    \Abs{\phi}_{H_b^k}.
  \end{equation*}
  This together with \autoref{Thm_EllipticRegularityAndEstimates_1} implies the assertion.
\end{proof}

%%% Local Variables:
%%% mode: latex
%%% TeX-master: "DiracOperatorsTwistedByRamifiedLineBundles"
%%% ispell-local-dictionary: "british"
%%% End:

%% file: FredholmExtensionsInHigherRegularity.tex
\subsection{Fredholm extensions in higher regularity}
\label{Sec_K+1/2RegularResidueConditions}

The following is a consequence of \autoref{Prop_BH-1/2CompactLeftSemiFredholm}.

\begin{cor}[$\frac12$--regular implies left semi-Fredholm]
  \label{Cor_1/2RegularLeftSemiFredholm}
  Let $\ResidueCondition \subset \check H\Gamma(Z,\check S)$ be a residue condition.
  If $\ResidueCondition$ is $\tfrac12$--regular,
  then $D_\ResidueCondition$ is left semi-Fredholm.
\end{cor}

The discussion in \autoref{Sec_EllipticRegularityAndEstimates} leads to the following observation.

\begin{definition}
  Let $k \in \N_0$.
  Let $\ResidueCondition \subset \check H\Gamma\paren{Z,\check S}$ be a residue condition.
  Consider the closed subspaces
  \begin{equation*}
    H_a^{k+1}\Gamma\paren{X\setminus Z,S\otimes\fl;\ResidueCondition}
    \coloneq
    \set{ \phi \in H_a^{k+1}\Gamma\paren{X\setminus Z,S\otimes\fl} : \res[\phi] \in \ResidueCondition }
  \end{equation*}
  and the  restriction of $D$ to
  \begin{equation*}
    D_\ResidueCondition^{(k)} \co H_a^{k+1}\Gamma\paren{X\setminus Z,S\otimes\fl;\ResidueCondition} \to H_b^k\Gamma\paren{X\setminus Z,S\otimes\fl}.
    \qedhere
  \end{equation*}
\end{definition}

\begin{prop}
  \label{Prop_DBkFredholm}
  Let $k \in \N_0$.
  Let $\ResidueCondition \subset \check H\Gamma\paren{Z,\check S}$ be a $(k+\frac12)$--regular residue condition.
  The following hold:
  \begin{enumerate}
  \item
    \label{Prop_DBkFredholm_LeftSemiFredholm}
    $D_\ResidueCondition^{(k)}$ is left semi-Fredholm;
    in fact:
    \begin{equation*}
      \ker D_\ResidueCondition^{(k)} = \ker D_\ResidueCondition
    \end{equation*}
    and the canonical map
    \begin{equation*}
      \coker D_\ResidueCondition^{(k)} \to \coker D_\ResidueCondition
      \iso \paren{\ker D_{\ResidueCondition^G}}^*
    \end{equation*}
    is injective.
  \item
    \label{Prop_DBkFredholm_Fredholm}
    If $D_\ResidueCondition$ is Fredholm,
    then $D_\ResidueCondition^{(k)}$ is Fredholm and the canonical map $\coker D_\ResidueCondition^{(k)} \to \coker D_\ResidueCondition$ is an isomorphism.
  \item
    \label{Prop_DBkFredholm_HodgeDecomposition}
    If $\ker D_{\ResidueCondition^G} \subset H_b^k\Gamma\paren{X\setminus Z,S\otimes\fl}$,
    then the latter $L^2$ orthogonally decomposes as
    \begin{equation*}
      H_b^k\Gamma\paren{X\setminus Z,S\otimes\fl}
      =
      \im D_\ResidueCondition^{(k)}
      \oplus
      \ker D_{\ResidueCondition^G}.
    \end{equation*}
  \end{enumerate}
\end{prop}

\begin{proof}%[Proof of \autoref{Prop_DBkFredholm}]  
  The proof is identical to the one of \cite[Theorem 3.58]{Doan2024},
  but repeated here for the readers' convenience.
  By \autoref{Thm_EllipticRegularityAndEstimates_2},
  $D_\ResidueCondition^{(k)}$ is left semi-Fredholm,
  $\ker D_\ResidueCondition^{(k)} = \ker D_\ResidueCondition$;
  moreover: the linear map
  \begin{equation*}
    \frac{\dom(D_\ResidueCondition)}{H_a^{k+1}\Gamma\paren{X\setminus Z,S\otimes\fl;\ResidueCondition}}
    \to
    \frac{L^2\Gamma\paren{X\setminus Z,S\otimes\fl}}{H_b^k\Gamma\paren{X\setminus Z,S\otimes\fl}}
  \end{equation*}
  induced by $D_\ResidueCondition$ is injective.
  Therefore, by the Snake Lemma,
  the canonical map
  \begin{equation*}
    \coker D_\ResidueCondition^{(k)} \to \coker D_\ResidueCondition
  \end{equation*}
  is injective.
  This proves \autoref{Prop_DBkFredholm_LeftSemiFredholm}.

  $H_b^k\Gamma\paren{X\setminus Z,S\otimes\fl}$ is dense in $L^2\Gamma\paren{X\setminus Z,S\otimes\fl}$.
  Therefore, if $\ker D_{\ResidueCondition^G}$ is finite-dimensional,
  then  the map
  \begin{equation*}
    H_b^k\Gamma\paren{X\setminus Z,S\otimes\fl} \to \paren{\ker D_{\ResidueCondition^G}}^* \iso \coker D_\ResidueCondition
  \end{equation*}
  is surjective.
  Since it factors through $\coker D_\ResidueCondition^{(k)} \to \coker D_\ResidueCondition$,
  the latter must be surjective.
  This proves \autoref{Prop_DBkFredholm_Fredholm}.

  \autoref{Prop_DBkFredholm_HodgeDecomposition} is obvious.
\end{proof}

\begin{example}
  Let $\ResidueCondition_L \subset \check H\Gamma\paren{Z,\check S}$ be an $\infty$--regular Lagrangian residue condition and $\rho \in \check H\Gamma\paren{Z,\check S}$ but $\rho \notin H^{1/2}\Gamma\paren{Z,\check S}$.
  The residue condition $\ResidueCondition \coloneq \paren{\ResidueCondition_L + \R\rho}^G \subset \ResidueCondition_L$ is $\infty$--regular and $D_\ResidueCondition$ is Fredholm and, by \autoref{Lem_FredholmResidueConditionComparison}, $\ind D_\ResidueCondition = -1$;
  however, $\ResidueCondition^G = \ResidueCondition_L + \R\rho$ is not $\tfrac12$--regular.
  If $\rho = \res[\phi] \in \Lambda$, the Calderón subspace defined in \autoref{Ex_CalderonSubspace_Lagrangian},
  then $\phi \in \ker D_{\ResidueCondition^G}$ but $\phi \notin H_b^1\Gamma\paren{X\setminus Z,S\otimes\fl}$ by \autoref{Lem_ResidueMap_Higher_Regularity}.
\end{example}

\begin{remark}
  The applications of the Snake Lemma in \autoref{Sec_FredholmExtension} carry over to the higher regularity setting with minor cosmetic modifications.
\end{remark}

\autoref{Prop_DBkFredholm} can be employed to obtain higher regularity analogues of the $L^2$ orthogonal decompositions
\begin{equation*}
  L^2\Gamma\paren{X\setminus Z, S\otimes \fl} = \im \Dmax \oplus \ker \Dmin
  \qandq
  L^2\Gamma\paren{X\setminus Z, S\otimes \fl} = \im \Dmin \oplus \ker \Dmax.
\end{equation*}

\begin{prop}
  \label{Prop_DkHodgeDecomposition}
  For every $k \in \N_0$ the following hold:
  \begin{enumerate}
  \item
    \label{Prop_DkHodgeDecomposition_1}
    $H_b^k\Gamma\paren{X\setminus Z,S\otimes\fl}$
    $L^2$ orthogonally decomposes as
    \begin{equation*}
      H_b^k\Gamma\paren{X\setminus Z,S\otimes\fl}
      =
      \im D^{(k)}
      \oplus
      \ker \Dmin.
    \end{equation*}
  \item
    \label{Prop_DkHodgeDecomposition_2}
    $H_b^{k+1}\Gamma\paren{X\setminus Z,S\otimes\fl}$
    $L^2$ orthogonally decomposes as
    \begin{equation*}
      H_b^{k+1}\Gamma\paren{X\setminus Z,S\otimes\fl}
      =
      \im D_0^{(k+1)}
      \oplus
      \ker D^{(k)}.
    \end{equation*}
  \end{enumerate}
\end{prop}

\begin{proof}%[Proof of \autoref{Prop_DkHodgeDecomposition}]
  Since $\ker D_{\ResidueCondition_\APS}$ is finite-dimensional,
  there is a $\tau \leq 0$ such that the projection
  \begin{equation*}
    \res(\ker D_{\ResidueCondition_\APS}) \to \bigoplus_{\lambda \in [\tau,0)} \ker (A-\lambda\one)
  \end{equation*}
  is injective.
  The residue conditions
  \begin{equation*}
    \ResidueCondition_\tau
    \coloneq
    \one_{(-\infty,\tau)}(A)H^{1/2}\Gamma\paren{Z,\check S}
    \subset
    \check H\Gamma\paren{Z,\check S}
  \end{equation*}
  and $\ResidueCondition_\tau^G = \one_{(-\infty,-\tau]}(A)H^{1/2}\Gamma\paren{Z,\check S} $ are $\infty$--regular.
  Therefore,
  by \autoref{Prop_DBkFredholm},
  \begin{equation*}
    H_b^k\Gamma\paren{X\setminus Z,S\otimes\fl}
    =
    \im D_{\ResidueCondition_\tau^G}^{(k)}
    \oplus
    \ker D_{\ResidueCondition_\tau}.
  \end{equation*}
  By construction,
  $\ker D_{\ResidueCondition_\tau} = \ker \Dmin$.
  Moreover, since $\im D^{(k)}$ and $\ker \Dmin$ are $L^2$ orthogonal,
  $\im D_{\ResidueCondition_\tau^G}^{(k)} = \im D^{(k)}$.
  This proves \autoref{Prop_DkHodgeDecomposition_1}.

  Since $\im D_0^{(k)}$ and $\ker D^{(k)}$ are $L^2$ orthogonal,
  to prove \autoref{Prop_DkHodgeDecomposition_2} it suffices to prove that for every $\phi \in H_b^k\Gamma\paren{X\setminus Z,S\otimes\fl}$ which is $L^2$ orthogonal to $\ker D^{(k)}$, there is a $\psi \in H_a^{k+1}\Gamma\paren{X\setminus Z,S\otimes\fl;0}$ with $D\psi = \phi$.
  For every $\sigma \leq \tau$,
  by \autoref{Prop_DBkFredholm} and
  since $\phi$ is $L^2$ orthogonal to $\ker D^{(k)} \supset \ker D_{\ResidueCondition_\sigma^G}$,
  there is a unique $\psi_\sigma \in H_a^{k+1}\Gamma\paren{X\setminus Z,S\otimes\fl;\ResidueCondition_\sigma}$ such that $D\psi_\sigma = \phi$ and $\psi_\sigma$ is $L^2$ orthogonal to $\ker D_{\ResidueCondition_\sigma} = \ker \Dmin$.
  Since $\psi_\sigma \in H_a^{k+1}\Gamma\paren{X\setminus Z,S\otimes\fl;\ResidueCondition_\tau}$,  
  $\psi_\sigma = \psi_\tau \eqcolon \psi$ is independent of $\sigma \leq \tau$.
  By construction,
  $D\psi = \phi$
  and $\res[\psi] \in \bigcap_{\sigma \leq \tau} \ResidueCondition_\sigma = 0$.
\end{proof}

This has the following noteworthy consequence.

 \begin{cor}
  \label{Cor_L2KernelProjectionRegularity}
  Denote by $\Pi \co L^2\Gamma\paren{X\setminus Z,S\otimes \fl} \to \ker \Dmax$ the $L^2$ orthogonal projection onto $\ker \Dmax$.
  For every $k \in \N_0$
  the restriction of $\Pi$ to $H_a^{k+1}\Gamma\paren{X\setminus Z,S\otimes \fl}$ factors through the inclusion $\ker D^{(k)} \subset \ker \Dmax$.
\end{cor}

%%% Local Variables:
%%% mode: latex
%%% TeX-master: "DiracOperatorsTwistedByRamifiedLineBundles"
%%% ispell-local-dictionary: "british"
%%% End:

%% file: SobolevSpaces_2.tex
\subsection{Adapted Sobolev spaces, II: $L^\infty$ bound}
\label{Sec_SobolevSpaces_2}

The purpose of the upcoming two subsections is to further understand the scale of adapted Sobolev spaces $\paren{H_a^k\Gamma\paren{X\setminus Z,S\otimes \fl},\Abs{-}_{H_a^k}}_{k \in \N_0}$.

\begin{prop}
  \label{Prop_HakSubleadingTerm}
  For every $k \in \N$ with $k > n/2$
  \begin{equation*}
    H_a^{k+1}\Gamma\paren{X\setminus Z,S\otimes\fl} \subset r^{-1/2}L^\infty\Gamma\paren{X\setminus Z,S\otimes\fl};
  \end{equation*}
  in fact,
  every $\phi \in H_a^{k+1}\Gamma\paren{X\setminus Z,S\otimes\fl}$ is of the form
  \begin{equation*}
    \phi = \ext(\res(\phi)) + \psi
  \end{equation*}
  with $\psi \in L^\infty\Gamma\paren{X\setminus Z,S\otimes\fl}$
  and $\Abs{\psi}_{L^\infty} \lesssim_k \Abs{\phi}_{H_a^{k+1}}$.
\end{prop}

The proof relies on the following observations.

\begin{prop}
  \label{Prop_HbkBorderlineHardyInequality}
  For every $k \in \N_0$ and $\phi \in H_a^{k+1}\Gamma\paren{X\setminus Z,S\otimes\fl}$ with $\res(\phi) = 0$
  \begin{equation*}
    \Abs{\nabla \phi}_{H_b^k} + \Abs{r^{-1}\phi}_{H_b^k}
    \lesssim_k
    \Abs{D\phi}_{H_b^k} + \Abs{\phi}_{H_b^k}.
  \end{equation*}
\end{prop}

\begin{proof}
  For $k = 0$ this is an immediate consequence of \autoref{Hyp_BorderlineHardyInequality} and \autoref{Eq_D_EllipticEstimate}.
  The general case is a consequence of the case $k = 0$, \autoref{Lem_CommutationRelations} and  \autoref{Lem_ResDelB}.
\end{proof}

\begin{prop}
  \label{Prop_HbkLinfty}
  For every $k \in \N$ with $k > n/2$ and $\phi \in H_b^k\Gamma\paren{X\setminus Z,S\otimes\fl}$  
  \begin{equation*}
    \Abs{r \phi}_{L^\infty} \lesssim_k \Abs{\phi}_{H_b^k}.
  \end{equation*}
\end{prop}

\begin{proof}
  This is a consequence of the usual Sobolev embedding theorem and a scaling consideration.
\end{proof}

\begin{proof}[Proof of \autoref{Prop_HakSubleadingTerm}]
  Since $\psi \coloneq \phi - \ext(\res(\phi)) \in \ker \res$,
  this is a consequence of \autoref{Prop_HbkBorderlineHardyInequality} and \autoref{Prop_HbkLinfty}.
\end{proof}

The above observation leads to the following ``poor man's Weyl law''.

\begin{prop}[Growth of eigenvalues]
  \label{Prop_GrowthOfEigenValues}
  Let $k \in \N_0$ with $k > n/2$.
  If $\ResidueCondition \subset \check H\Gamma\paren{Z,\check S}$ is a $\paren{k+\tfrac12}$--regular Lagrangian residue condition,
  then the \defined{counting function} $N \co [0,\infty) \to \N_0$ defined by
  \begin{equation*}
    N(\Lambda) \coloneq \dim E_{\leq \Lambda} \qwithq E_{\leq \Lambda} \coloneq \bigoplus_{\lambda \in [-\Lambda,\Lambda]} \ker\paren{D_\ResidueCondition - \lambda\one}
  \end{equation*}
  satisfies
  \begin{equation*}
    N(\Lambda) \lesssim_{\ResidueCondition,k} \bracket{\Lambda}^{2(k+1)}.
  \end{equation*}
\end{prop}

\begin{proof}%[Proof of \autoref{Prop_GrowthOfEigenValues}]
  The following argument is due to \citet[Lemma 11]{Li1980}.
  Choose an $L^2$ orthonormal basis $(\phi_1,\ldots,\phi_{N(\Lambda)})$ of $E_{\leq \Lambda}$.
  The density $d \in C^\infty(X\setminus Z,[0,\infty))$ defined by
  \begin{equation*}
    d \coloneq \sum_{i=1}^{N(\Lambda)} \abs{\phi_i}^2
  \end{equation*}
  does not depend on the choice of $L^2$ orthonormal basis.
  By construction
  \begin{equation*}
    N(\Lambda) = \frac{1}{\vol(X)} \int_X d \lesssim \Abs{rd}_{L^\infty}.
  \end{equation*}
  
  Choose an $x \in X$ with $\Abs{rd}_{L^\infty} \leq 2\abs{rd}(x)$.
  Since $\ev_x \co E_{\leq \Lambda} \to \paren{S\otimes\fl}_x$ has rank at most $\rk S$,
  without loss of generality,
  \begin{equation*}
    \abs{rd}(x)
    =
    \sum_{i=1}^{\rk S} r\abs{\phi_i}^2(x).
  \end{equation*}
  By \autoref{Thm_EllipticRegularityAndEstimates_2}
  and \autoref{Prop_HakSubleadingTerm}
  \begin{equation*}
    \Abs{r^{1/2}\phi_i}_{L^\infty}
    \lesssim \Abs{\phi_i}_{H_a^{k+1}}
    \lesssim_{\ResidueCondition,k} 
    \Abs{D^{k+1}\phi_i}_{L^2} + \Abs{\phi_i}_{L^2}
    \lesssim
    \bracket{\Lambda}^{k+1}.
  \end{equation*}
  This implies the assertion.
\end{proof}

\begin{cor}
  \label{Cor_HeatKernelTraceClass}
  Let $k \in \N_0$ with $k > n/2 + 1$.
  If $\ResidueCondition \subset \check H\Gamma\paren{Z,\check S}$ is a $\paren{k+\tfrac12}$--regular Lagrangian residue condition,
  then for every $t \in (0,\infty)$ the heat operator $h_t \coloneq \exp(-tD_\ResidueCondition^2)$ is trace class.
\end{cor}

\begin{remark}
  \label{Rmk_APSIndexTheorem}
  Assume the situation of \autoref{Cor_HeatKernelTraceClass}.
  If $\epsilon$ is a chirality operator,
  then for every $t > 0$
  \begin{equation*}
    \ind D_\ResidueCondition^+
    = \dim \ker D_\ResidueCondition^+ - \dim \ker D_\ResidueCondition^-
    = \str_\epsilon h_t.
  \end{equation*}
  Here $\str_\epsilon$ denotes the super trace with respect to $\epsilon$ of the heat operator $h_t$; cf.~\cite[§1.3]{BerlineGetzlerVergne1992}.
  For suitable choices of $\ResidueCondition$ an analysis of the asymptotic behaviour of the kernel attached to $h_t$ as $t \downarrow 0$ should result in index formulae analogous to the one established by \citet[Theorems 3.10 and 4.2]{Atiyah1975a}.  
  It would be interesting to work this out in detail.
  Also, it should be mentioned that part of the unpublished PhD thesis \cite[Theorems 1.0.3 and 2.3.4]{Yang2007:MITPhDThesis} discusses such index formulae.
\end{remark}

%%% Local Variables:
%%% mode: latex
%%% TeX-master: "DiracOperatorsTwistedByRamifiedLineBundles"
%%% ispell-local-dictionary: "british"
%%% End:

%% file: SobolevSpaces_3.tex
\subsection{Adapted Sobolev spaces, III: tameness}
\label{Sec_SobolevSpaces_3}

This subsection establishes that the graded Fréchet spaces $\paren{H_b^\infty\Gamma\paren{X\setminus Z,S\otimes \fl},\paren{\Abs{-}_{H_b^k}}_{k \in \N_0}}$ and $\paren{H_a^\infty\Gamma\paren{X\setminus Z,S\otimes \fl},\paren{\Abs{-}_{H_a^k}}_{k \in \N_0}}$ are tame in the sense of \cite[Part II Definition 1.3.2]{Hamilton1982:NashMoser}.

\begin{prop}
  \label{Prop_TameFrechetHb}
  The graded Fréchet space
  $\paren{H_b^\infty\Gamma\paren{X\setminus Z,S\otimes \fl},\paren{\Abs{-}_{H_b^k}}_{k \in \N_0}}$
  is tame.
\end{prop}

The proof requires the following lemma.
By a slight abuse of notation,
also denote by $\mathring S \otimes \mathring\fl \to NZ \setminus Z$ the pullback of $\underline S \otimes \underline \fl \to F$ along the canonical projection $NZ \setminus Z \to F$.

\begin{lemma}
  \label{Lem_TameFrechetHbNZ}
  The graded Fréchet space
  $\paren{H_b^\infty\Gamma\paren{NZ \setminus Z, \mathring S \otimes \mathring\fl},\paren{\Abs{-}_{H_b^k}}_{k \in \N_0}}$
  is tame.
\end{lemma}

\begin{proof}[Proof of \autoref{Prop_TameFrechetHb}]
  The graded Fréchet space
  $\paren{H_b^\infty\Gamma\paren{X\setminus Z,S\otimes \fl},\paren{\Abs{-}_{H_b^k}}_{k \in \N_0}}$
  is a tame direct summand in the sense of \cite[Part II Definition 1.3.1]{Hamilton1982:NashMoser} of the direct sum of graded Fréchet spaces
  $\paren{H^\infty\paren{\hat X, \hat S \otimes \hat \fl},\paren{\Abs{-}_{H^k}}_{k \in \N_0}}$
  and 
  $\paren{H_b^\infty\Gamma\paren{NZ \setminus Z, \mathring S \otimes \mathring\fl},\paren{\Abs{-}_{H_b^k}}_{k \in \N_0}}$.
  Indeed, define
  $i \co H_b^\infty\Gamma\paren{X\setminus Z,S\otimes \fl} \to H^\infty\paren{\hat X, \hat S \otimes \hat \fl} \oplus H_b^\infty\Gamma\paren{NZ \setminus Z, \mathring S \otimes \mathring\fl}$
  and
  $p \co H^\infty\paren{\hat X, \hat S \otimes \hat \fl} \oplus H_b^\infty\Gamma\paren{NZ \setminus Z, \mathring S \otimes \mathring\fl} \to H_b^\infty\Gamma\paren{X\setminus Z,S\otimes \fl}$
  by
  \begin{equation*}
    i (\phi) \coloneq \paren{(1-\chi(2r))\phi, \chi (2r)\phi}
    \qandq 
    p(\hat \psi, \mathring \psi) \coloneq  (1-\chi(4r))\hat \psi + \chi (r) \mathring \psi.
  \end{equation*}
  Evidently, $i$ and $p$ are tame linear maps.
  A direct computation shows that $p \circ i = \one$; indeed, 
  \begin{equation*}
    p \circ i(\phi)=  (1-\chi(4r))(1-\chi(2r))\phi+ \chi (r) \chi (2r) \phi= (1-\chi(2r))\phi+ \chi (2r) \phi=\phi.
  \end{equation*}
  Therefore, by \cite[Part II, Corollary 1.3.7]{Hamilton1982:NashMoser} together with the usual Sobolev embedding theorems, and by \autoref{Lem_TameFrechetHbNZ}, 
  $\paren{H_b^\infty\Gamma\paren{X\setminus Z,S\otimes \fl},\paren{\Abs{-}_{H_b^k}}_{k \in \N_0}}$ is tame; cf.~\cite[Part II Lemmas 1.3.3 and 1.3.4]{Hamilton1982:NashMoser}.
\end{proof}

The proof of \autoref{Lem_TameFrechetHbNZ} relies on the following observation.

\begin{lemma}
  \label{Lem_TamenessCriterionSmoothingOperator}  
  Let $\paren{F,\paren{\Abs{-}_{k}}_{k \in \N_0}}$ be a graded Fréchet space.
  If there is a family of smoothing operators $\paren{\sS_\beta}_{\beta \in \N}$ on $F$ such that for every $\phi \in F$, $\beta \in \N$ and $k,\ell \in \N_0$ 
  \begin{equation}
    \label{Eq_SmoothingOperator}
    \Abs{\sS_\beta \phi }_{k+\ell}\lesssim_{k,\ell} e^{\ell \beta} \Abs{\phi }_{k},
    \quad
    \Abs{\phi-\sS_\beta \phi }_{k}\lesssim_{k,\ell} e^{-\ell \beta} \Abs{\phi }_{k+\ell}
  \end{equation}
  and
  \begin{equation}
    \label{Eq_ExtraConditionSmoothingOperator}
    \sS_ {\beta+1}\sS_{\beta}=\sS_{\beta},
  \end{equation}
  then $\paren{F,\paren{\Abs{-}_{k}}_{k \in \N_0}}$ is tame in the sense of \cite[Part II Definition 1.3.2]{Hamilton1982:NashMoser}.
\end{lemma}

\begin{proof}%[Proof of \autoref{Lem_TamenessCriterionSmoothingOperator}]
  As a consequence of \autoref{Eq_SmoothingOperator} the seminorms $\Abs{-}_{k}$ ($k \in \N_0$) are norms.
  Denote by
  $\paren{B, \Abs{-}_0}$,
  the Banach space obtained by completing $\paren{F, \Abs{-}_0}$.
  Consider the graded Fréchet space
  $\paren{\Sigma\paren{B},\paren{\Abs{-}_k}_{k \in \N_0}}$
  of exponentially decreasing sequences defined by  
  \begin{equation*}
    \Sigma\paren{B}
    \coloneq
    \set*{
      (\psi_\beta) \in B^{\N_0}
      :
      \Abs{\psi_\beta}_k < \infty
      ~\text{for every}~
      k \in \N_0
    }
    \qwithq
    \Abs{(\psi_\beta)}_k
    \coloneq
    \sup_\beta e^{k \beta} \Abs{\psi_\beta}_{0};
  \end{equation*}
  cf.~\citet[Part II Example 1.1.4(c)]{Hamilton1982:NashMoser}. 
  
  By \autoref{Eq_SmoothingOperator}, $\sS_\beta$ uniquely extends to a bounded linear operator $\sS_\beta \co B\to F \subseteq B$.
  Define $ i \co F \to \Sigma\paren{B}$  and $ p \co \Sigma\paren{B} \to F$ by
  \begin{equation*}
    i(\phi)_\beta \coloneq \paren{\sS_{\beta+1} -  \sS_{\beta}} \phi \qwithq \sS_{0}\coloneq 0,
    \qandq 
    p(\psi_\beta) \coloneq \sum_{\beta=0}^\infty \sS_{\beta+2} \psi_\beta.
  \end{equation*}
  The linear maps $ i$ and $ p$ are tame; indeed, by \autoref{Eq_SmoothingOperator}
  \begin{align*}
    \Abs{i(\phi)}_k
    \leq
    \sup_\beta e^{k \beta} \paren*{\Abs{\paren{\one - \sS_{\beta+1}} \phi}_{0} + \Abs{\paren{\one - \sS_{\beta}} \phi}_{0}}
    \lesssim_k
    \Abs{\phi}_k
  \end{align*}
  and 
  \begin{align*}
    \Abs{  p(\psi_\beta)}_k\leq \sum_{\beta=0}^\infty \Abs{ \sS_{{\beta+2}} \psi_\beta}_k \lesssim_k \sum_{\beta=0}^\infty e^{k \beta} \Abs{\psi_\beta}_{0} \lesssim \Abs{(\psi_\beta)}_{k+1}.
  \end{align*}
  A  direct computation shows that $ p \circ  i = \one$; indeed, by  \autoref{Eq_SmoothingOperator} and \autoref{Eq_ExtraConditionSmoothingOperator}
  \begin{equation*}
    p\circ i(\phi)
    =
    \sum_{\beta=0}^\infty
    \sS_{\beta+2}\paren{\sS_{\beta+1} -  \sS_{\beta}}\phi
    =
    \sum_{\beta=0}^\infty
    \paren{\sS_{\beta+1} - \sS_{\beta}}\phi
    =
    \lim_{\beta \to \infty} \sS_{\beta+1} \phi
    =
    \phi.
  \end{equation*}
  Therefore,
  $\paren{F,\paren{\Abs{-}_{k}}_{k \in \N_0}}$ is a tame direct summand in the sense of \cite[Part II Definition 1.3.1]{Hamilton1982:NashMoser} of $\paren{\Sigma\paren{B},\paren{\Abs{-}_k}_{k \in \N_0}}$;
  that is: it is tame.
\end{proof}

\begin{proof}[Proof of \autoref{Lem_TameFrechetHbNZ}]
  Because of \autoref{Lem_TamenessCriterionSmoothingOperator} it suffices to construct a suitable family of smoothing operators $\paren{\sS_\beta}_{\beta\in \N}$ on $H_b^\infty\Gamma\paren{NZ \setminus Z, \mathring S \otimes \mathring\fl}$.
  Here is some preparation for this construction.
  Choose a cut-off function $\rho \in C^\infty([0,\infty),[0,1])$ with $\rho|_{[0,1/e]} = 1$ and $\supp(\rho) \subseteq [0,1)$.
  Define the isometry $T \co L^2\paren{(0,\infty), r \rd r; \C} \to L^2\paren{\R;\C}$ by $T(f) (s) \coloneq e^sf(e^s)$ and denote by $\sF \co L^2\paren{\R;\C} \to L^2\paren{\R;\C}$ the Fourier transform.
  The Mellin transform $\sM \co L^2\paren{(0,\infty), r \rd r; \C}\to L^2\paren{\R;\C}$ is defined by $\sM \coloneq \sF \circ T$ and the Mellin convolution $*_{\sM}$ by $f *_{\sM} g \coloneq T^{-1}(Tf * Tg)$.
  In particular,
  for every $f,g \in L^2\paren{(0,\infty),r \rd r;\C}$
  \begin{equation*}
    \sM(f*_\sM g) = \sM(f)\sM(g).
  \end{equation*}

  Decompose every $\phi\in H_b^\infty\Gamma\paren{NZ \setminus Z, \mathring S \otimes \mathring\fl}$ as
  \begin{equation*}
    \phi = \sum_{(\lambda,\mu) \in \Spectrum} \phi_{\lambda,\mu} \in  
    \bigoplus_{(\lambda,\mu) \in \Spectrum}
    L^2\paren{(0,\infty),r\rd r;E_{\lambda,\mu}}.
  \end{equation*}
  as in \autoref{Sec_SpectralDecomposition} and for every $\beta \in \N$ define
  \begin{equation*}
    \sS_\beta \phi \coloneq \sum_{(\lambda,\mu) \in \Spectrum} \rho_\beta({\bracket{\lambda}+\bracket{\mu}}) \sM^{-1}(\rho_\beta)*_{\sM}  \phi_{\lambda,\mu}
  \end{equation*}
  with
  $\rho_\beta \coloneq \rho(e^{-\beta}\abs{\cdot})$.  

  Since $\rho_{\beta+1} \cdot \rho_{\beta}= \rho_{\beta}$,
  it follows that $\sS_ {\beta+1}\sS_{\beta}=\sS_{\beta}$; indeed:
  \begin{equation*}
    \sM(\sS_ {\beta+1}\sS_{\beta} \phi)
    =
    \sum_{(\lambda,\mu) \in \Spectrum} \rho_{\beta+1}\paren*{\bracket{\lambda}+\bracket{\mu}} \rho_{\beta+1} \cdot \rho_{\beta}\paren*{\bracket{\lambda}+\bracket{\mu}}  \rho_{\beta} \sM(\phi_{\lambda,\mu})
    =
    \sM(\sS_\beta\phi).
  \end{equation*}
  Furthermore, since $\bracket{s} \lesssim e^\beta$ for every $s\in \supp \rho_\beta$, and $T\paren{r\partial_r} T^{-1}=\partial_s-\one$
  \begin{align*}
    \Abs{\sS_\beta \phi }^2_{H_b^{k+\ell}} 
    &\asymp_{k,\ell}
      \sum_{(\lambda,\mu) \in \Spectrum}
      \int_\R \paren*{\bracket{s}+ \bracket{\lambda}+ \bracket{\mu}}^{2k+2\ell} \rho_\beta^2\paren*{{\bracket{\lambda}+\bracket{\mu}}}\rho_\beta(s)^2
      \abs{\sM(\phi_{\lambda,\mu})(s)}^2 \,\rd s \\
    & \lesssim_{k}
      e^{2\ell \beta} \Abs{\phi }^2_{H_b^{k}}.
  \end{align*}
  Similarly,
  \begin{align*}
    \Abs{\phi-\sS_\beta \phi }^2_{H_b^k} 
    &\asymp_{k}
      \sum_{(\lambda,\mu) \in \Spectrum}
      \int_\R \paren*{\bracket{s}+ \bracket{\lambda}+ \bracket{\mu}}^{2k} \paren*{1-\rho_\beta({\bracket{\lambda}+\bracket{\mu}})\rho_\beta(s)}^2 \abs{\sM(\phi_{\lambda,\mu})(s)}^2 \,\rd s \\
    &\lesssim_{k,\ell}
      e^{-2\ell \beta} \Abs{\phi}^2_{H_b^{k+\ell}}. 
  \end{align*}
  Therefore,
  by \autoref{Lem_TamenessCriterionSmoothingOperator},
  $\paren{H_b^\infty\Gamma\paren{NZ \setminus Z, \mathring S \otimes \mathring\fl},\paren{\Abs{-}_{H_b^k}}_{k \in \N_0}}$ is tame. 
\end{proof}

\begin{prop}
  \label{Prop_TameFrechet}
  The graded Fréchet space
  $\paren{H_a^\infty\Gamma\paren{X\setminus Z,S\otimes \fl},\paren{\Abs{-}_{H_a^{k+1}}}_{k \in \N_0}}$
  is tame.
\end{prop}

\begin{proof}
  Consider the $\infty$--regular Lagrangian residue condition $\ResidueCondition\coloneq \ResidueCondition_\APS  \oplus L$ with $L \subseteq \ker A$ a Lagrangian subspace as discussed in \autoref{Ex_APSResidueCondition_Fredholm}.
  The operator $D_\ResidueCondition$ is self-adjoint and Fredholm.
  By \autoref{Prop_DBkFredholm}, the graded Fréchet space $\paren{\im D_\ResidueCondition^{(\infty)}\coloneq\bigcap_k \im D^{(k)}_\ResidueCondition,\paren{\Abs{-}_{H_b^k}}_{k \in \N_0}}$
  is a tame direct summand of $\paren{H_b^\infty\Gamma\paren{X\setminus Z,S\otimes \fl},\paren{\Abs{-}_{H_b^k}}_{k \in \N_0}}$ in the sense of \cite[Part II Definition 1.3.1]{Hamilton1982:NashMoser}. Therefore, by \autoref{Prop_TameFrechetHb}, it is tame; cf.~\cite[Part II Lemma 1.3.3]{Hamilton1982:NashMoser}.   
  
  Denote by $\Pi_\ResidueCondition \co L^2\Gamma\paren{X\setminus Z,S\otimes \fl} \to \ker D_\ResidueCondition$ the $L^2$ orthogonal projection onto $\ker D_\ResidueCondition$. The above further implies that the graded Fréchet space
  $\paren{H_a^\infty\Gamma\paren{X\setminus Z,S\otimes \fl;\ResidueCondition},\paren{\Abs{-}_{H_a^k}}_{k \in \N_0}}$ is tamely isomorphic, via the restriction of the map $D_\ResidueCondition \oplus \Pi_\ResidueCondition$, to
  \begin{equation*}
    H_a^\infty\Gamma\paren{X\setminus Z,S\otimes \fl;\ResidueCondition}\cong \im D_\ResidueCondition^{(\infty)} \oplus \ker D_\ResidueCondition.
  \end{equation*}
  Therefore, by the above it is tame; cf.~\cite[Part II Lemma 1.3.4]{Hamilton1982:NashMoser}.
  
  Finally,
  by \autoref{Lem_ExtensionMap_Higher_Regularity}, \autoref{Lem_ResidueMap_Higher_Regularity} and \autoref{Thm_EllipticRegularityAndEstimates_1}, 
  $\paren{H_a^\infty\Gamma\paren{X\setminus Z,S\otimes \fl},\paren{\Abs{-}_{H_a^{k+1}}}_{k \in \N_0}}$
  is tamely isomorphic to
  \begin{equation*}
    H_a^\infty\Gamma\paren{X\setminus Z,S\otimes \fl} \iso H_a^\infty\Gamma\paren{X\setminus Z,S\otimes \fl;\ResidueCondition} \oplus \paren{J\ResidueCondition \cap \Gamma\paren{Z,\check S}}
  \end{equation*}
  and,
  therefore, it is tame, provided that the graded Fréchet space $\paren{J\ResidueCondition \cap \Gamma\paren{Z,\check S},\paren{\Abs{-}_{H^{k}}}_{k \in \N_0}}$ is tame.
  
  Consider the graded Fréchet space
  $\paren{\Sigma\paren{L^2\Gamma\paren{Z, \check S}},\paren{\Abs{-}_k}_{k \in \N_0}}$ of exponentially decreasing sequences defined by  
  \begin{gather*}
    \Sigma\paren{L^2\Gamma\paren{Z, \check S}}
    \coloneq
    \set*{
      (\psi_\beta) \in L^2\Gamma\paren{Z, \check S}^{\N_0}
      :
      \Abs{\psi_\beta}_k < \infty
      ~\text{for every}~
      k \in \N_0
    }
    \\
    \intertext{with}
    \Abs{(\psi_\beta)}_k^2
    \coloneq
    \sum_{\beta=0}^\infty e^{2k \beta} \Abs{\psi_\beta}_{L^2}^2;
  \end{gather*}
  cf.~\citet[Part II Example 1.1.4(b) with $q=2$]{Hamilton1982:NashMoser}.
  Let $(\phi_\alpha)_{\alpha \in \N_0}$ be an $L^2$ orthonormal basis of $J\ResidueCondition$ consisting of eigenspinors for $A$ and denote by $(\lambda_\alpha)_{\alpha \in \N_0}$ the corresponding sequence of eigenvalues.
  Define $i \co J\ResidueCondition \cap \Gamma\paren{Z,\check S} \to \Sigma\paren{L^2\Gamma\paren{Z, \check S}}$  and $p \co \Sigma\paren{L^2\Gamma\paren{Z, \check S}} \to J\ResidueCondition \cap \Gamma\paren{Z,\check S}$ by
  \begin{equation*}
    i (\phi)_\beta \coloneq \sum_{\alpha=0}^\infty \one_{[e^\beta,e^{\beta+1})}\paren{\bracket{\lambda_\alpha}} \Inner{\phi,\phi_\alpha}_{L^2} \phi_\alpha 
    \qandq 
    p(\psi_\beta) \coloneq \sum_{\alpha,\beta=0}^\infty \one_{[e^\beta,e^{\beta+1})}(\bracket{\lambda_\alpha}) \Inner{\psi_\beta,\phi_\alpha}_{L^2} \phi_\alpha.
  \end{equation*}
  A moment's thought shows that $p \circ i = \one$;
  moreover:
  \begin{equation*}
    \Abs{i(\phi)}_k^2 = \sum_{\beta=0}^\infty e^{2k\beta} \Abs{i(\phi)_\beta}_{L^2}^2
    \leq \sum_{\alpha=0}^\infty \bracket{\lambda_\alpha}^{2k} \Inner{\phi,\phi_\alpha}_{L^2}^2
    \lesssim_k \Abs{\phi}_{H^k}^2
  \end{equation*}
  and
  \begin{align*}
    \Abs{p(\psi_\beta)}_{H^k}^2
    &
      \lesssim_k
      \sum_{\alpha=0}^\infty
      \bracket{\lambda_\alpha}^{2k}
      \Inner{p(\psi_\beta),\phi_\alpha}^2 \\
    &
      =
      \sum_{\alpha,\beta=0}^\infty
      \bracket{\lambda_\alpha}^{2k}
      \one_{[e^\beta,e^{\beta+1})}(\bracket{\lambda_\alpha})
      \Inner{\psi_\beta,\phi_\alpha}^2
      \leq
      e^{2k} \sum_{\beta=0}^\infty
      e^{2k\beta} \Abs{\psi_\beta}_{L^2}^2. 
  \end{align*}
  Therefore,
  $\paren{J\ResidueCondition \cap \Gamma\paren{Z,\check S},\paren{\Abs{-}_{H^{k}}}_{k \in \N_0}}$ is a tame direct summand of $\paren{\Sigma\paren{L^2\Gamma\paren{Z,\check S}},\paren{\Abs{-}_k}_{k \in \N_0}}$; that is: it is tame in the sense of \cite[Part II Definition 1.3.2]{Hamilton1982:NashMoser}.
\end{proof}

\begin{prop}
  ~
  \begin{enumerate}
  \item 
    \label{Prop_TameFrechet0_1}
    The graded Fréchet space
    $\paren{H_a^\infty\Gamma\paren{X\setminus Z,S\otimes \fl;0},\paren{\Abs{-}_{H_a^k}}_{k \in \N_0}}$ is tame.
    
  \item 
    \label{Prop_TameFrechet0_2}
    $\paren{H_a^\infty\Gamma\paren{X\setminus Z,S\otimes \fl},\paren{\Abs{-}_{H_a^k}}_{k \in \N_0}}$
    is tamely isomorphic to
    \begin{equation*}
      H_a^\infty\Gamma\paren{X\setminus Z,S\otimes \fl} \iso H_a^\infty\Gamma\paren{X\setminus Z,S\otimes \fl;0} \oplus  \Gamma\paren{Z,\check S}. 
    \end{equation*}
  \end{enumerate}
\end{prop}

\begin{proof}
  Evidently the inclusion map  $i \co H_a^\infty\Gamma\paren{X\setminus Z,S\otimes \fl;0} \to  H_a^\infty\Gamma\paren{X\setminus Z,S\otimes \fl} $ is tame,
  and $p\co H_a^\infty\Gamma\paren{X\setminus Z,S\otimes \fl} \to  H_a^\infty\Gamma\paren{X\setminus Z,S\otimes \fl;0}$ defined by $p(\phi)\coloneq \phi-\ext \res [\phi]$ satisfies $p\circ i=\one$.
  By \autoref{Lem_ExtensionMap_Higher_Regularity}, \autoref{Lem_ResidueMap_Higher_Regularity} and \autoref{Thm_EllipticRegularityAndEstimates_1},
  the map $p\oplus \res[-]$ defines the tame isomorphism in \autoref{Prop_TameFrechet0_2}, and in particular, $p$ itself is tame.
  Therefore, $\paren{H_a^\infty\Gamma\paren{X\setminus Z,S\otimes \fl;0},\paren{\Abs{-}_{H_a^k}}_{k \in \N_0}}$ is a tame direct summand of $\paren{H_a^\infty\Gamma\paren{X\setminus Z,S\otimes \fl},\paren{\Abs{-}_{H_a^k}}_{k \in \N_0}}$, and hence by \autoref{Prop_TameFrechet} it is tame; cf.~\cite[Part II Lemma 1.3.3]{Hamilton1982:NashMoser}.
\end{proof}

%%% Local Variables:
%%% mode: latex
%%% TeX-master: "DiracOperatorsTwistedByRamifiedLineBundles"
%%% ispell-local-dictionary: "british"
%%% End:

%% file: SymbolicRegularityCriterion.tex
\subsection{Symbolic criterion for \texorpdfstring{$\infty$}{infty}--regularity}
\label{Sec_SymbolicRegularityCriterion}

The following criterion is useful for verifying $\infty$--regularity.

\begin{prop}[symbolic criterion for $\infty$--regularity]
  \label{Prop_SymbolicCriterionForInftyRegularity}
  Let $V \subseteq \check S$ be a subbundle.
  If
  for every $\xi \in T^*Z \setminus \set{0}$
  \begin{equation*}
    V \oplus J\gamma(\xi) V = \check S,
  \end{equation*}
  then the local residue conditions $\ResidueCondition_V$ defined in \autoref{Def_LocalBoundaryCondition} and $\ResidueCondition_V^G$ are $\infty$--regular.
\end{prop}

The proof relies on the following criterion due to \citeauthor{BaerBallmann2012:BVP}.

\begin{theorem}[{\cite[Theorem 7.20 and Proposition 7.24]{BaerBallmann2012:BVP}}]
  \label{Thm_SymbolicCriterionForInftyRegularity_Complex}
  Suppose that $(S,\gamma,\nabla,\Tor)$ is a complex Dirac bundle with skew torsion.
  Let $V \subset \check S$ be a complex subbundle.
  Denote by $\pr_V \co \check S \to \check S$ the orthogonal projection onto $V$.
  If
  for every $\xi \in T^*Z \setminus \set{0}$ the projection $\pr_V$ restricts to an isomorphism from the sum of the positive eigenspaces of $iJ\gamma(\xi)$ onto $V$,
  then the local residue condition $\ResidueCondition_V$ defined in \autoref{Def_LocalBoundaryCondition} is $\infty$--regular.
\end{theorem}

\begin{proof}[Proof sketch]
  The proof of \cite[Theorem 7.20 and Proposition 7.24]{BaerBallmann2012:BVP} carries over verbatim to the present situation.
  It might be helpful to remind the reader of the crucial point.
  Denote the symbol of the pseudo-differential operator $\one_{[0,\infty)}(A)$ by $i\sigma_A$.
  According to \cite[Proposition 14.2]{BoossBavnbekWojciechowski1993},
  for every $\xi \in T^*Z \setminus \set{0}$,
  $i\sigma_A(\xi)$ is the orthogonal projection onto the sum of the positive eigenspaces of $i\sigma_A(\xi) = -iJ\gamma(\xi)$.
  An exercise in linear algebra shows that $i\pr_V - i\sigma_A(\xi)$ is invertible (precisely) if the condition stated above holds.
  As a consequence, $\pr_V - \one_{[0,\infty)}(A)$ satisfies elliptic estimates
  that in turn imply that $\ResidueCondition_V$ is $\infty$--regular.
\end{proof}

\begin{proof}[Proof of \autoref{Prop_SymbolicCriterionForInftyRegularity}]
  Let $x \in Z$ and $\xi \in T_x^*Z \setminus \set{0}$.
  $T \coloneq \abs{\xi}^{-1} J\gamma(\xi)$ defines an orthogonal complex structure on $\check S_x$.
  The complexification $\check S_x \otimes \C$ is isomorphic to $\check S_x \oplus \overline{\check S_x}$.
  The sum of the positive eigenspaces of $iT$ is precisely $\overline{\check S_x}$.
  The orthogonal projection induces an isomorphism $\overline{\check S_x} \to V\otimes \C$ precisely if $V \subset \check S_x$ is totally real;
  that is $V \oplus TV = \check S$.
  \qedhere  
\end{proof}

The following simpler version of \autoref{Prop_SymbolicCriterionForInftyRegularity} can be proved without pseudo-differential techniques.
This observation is already implicit in \cite[Lemma 2]{FarinelliSchwarz1998}.

\begin{prop}
  \label{Prop_SymbolicCriterionForInftyRegularity_Weak}
  Let $V \subset \check S$ be a subbundle.
  If
  \begin{equation*}
    J\gamma(\xi)V \subset V^\perp
  \end{equation*}
  for every $\xi \in T^*Z \setminus\set{0}$,
  then the local residue condition $\ResidueCondition_V$ defined in \autoref{Def_LocalBoundaryCondition} is $\infty$--regular.
\end{prop}

\begin{proof}%[Proof of \autoref{Prop_SymbolicCriterionForInftyRegularity_Weak}]
  Let $k \in \N_0$.  
  Denote by $\pr_V \co \check S \to \check S$ the orthogonal projection onto $V$.
  A priori $\pr_V A^{2k+1} \pr_V$ appears to be a differential operator of order at most $2k+1$.
  As such its symbol satisfies
  \begin{equation*}
    \sigma_{\pr_V A^{2k+1} \pr_V}(\xi)
    = \pr_V \sigma_A(\xi)^{2k+1} \pr_V
    = (-1)^k \abs{\xi}^{2k} \pr_V J\gamma(\xi) \pr_V
    = 0.
  \end{equation*}
  Consequently, $\pr_V A^{2k+1} \pr_V$ is a differential operator of order at most $2k$.
  
  Therefore, for every $\phi \in \Gamma(Z,V)$ and $\epsilon > 0$, by interpolation and Cauchy--Schwarz,
  \begin{align*}
    \Inner{A^{2k+1}\phi,\phi}_{L^2} 
    =
    \Inner{\pr_V A^{2k+1} \pr_V \phi,\phi}_{L^2}
    &\lesssim_{\ResidueCondition_V,k}    
      \Abs{\phi}_{H^k}^2 \\
    &
      \lesssim
      \Abs{\phi}_{H^{k-1/2}}\Abs{\phi}_{H^{k+1/2}}
      \lesssim
      \epsilon^{-1}\Abs{\phi}_{H^{k-1/2}}^2 + \epsilon\Abs{\phi}_{H^{k+1/2}}^2;
  \end{align*}
  moreover, by direct inspection,
  \begin{equation*}
    \Inner{A^{2k+1}\phi,\phi}_{L^2}
    =
    -\Abs{\abs{A}^{k+1/2}\one_{(-\infty,0)}(A)\phi}_{L^2}^2
    +
    \Abs{\abs{A}^{k+1/2}\one_{[0,\infty)}(A)\phi}_{L^2}^2.
  \end{equation*}
  As a consequence,
  for every $\phi \in \ResidueCondition_V$,
  \begin{align*}
    \Abs{\phi}_{H^{k+1/2}}^2
    &\lesssim
      \Abs{\abs{A}^{k+1/2}\one_{(-\infty,0)}(A)\phi}_{L^2}^2
      + \Abs{\abs{A}^{k+1/2}\one_{[0,\infty)}(A)\phi}_{L^2}^2 \\
    &\lesssim_{\ResidueCondition_V,k}
      \Abs{\one_{(-\infty,0)}(A)\phi}_{H^{k+1/2}}^2
      + \Abs{\phi}_{H^{k-1/2}}^2.
  \end{align*}
  By induction, $\ResidueCondition_V$ is $\paren{k+1/2}$--regular for every $k \in \N_0$;
  hence: $\infty$--regular.
\end{proof}

\begin{example}
  The local residue conditions associated with $\check S_N$ and $\check S_D$ defined in \autoref{Ex_NeumannDirichlet} obviously satisfy the criterion in \autoref{Prop_SymbolicCriterionForInftyRegularity}.
\end{example}

\begin{example}
  The MIT bag residue conditions $\ResidueCondition_\bag^\pm$ defined in \autoref{Ex_MITBag} satisfy the criterion in \autoref{Prop_SymbolicCriterionForInftyRegularity} because $\gamma(\xi)$ and $J$ anti-commute.
\end{example}

\begin{example}
  \label{Ex_DeformZ}
  If
  \begin{equation*}
    \bL \in \Gamma\paren{Z,\Hom_\C(\overline{NZ},\check S)}
  \end{equation*}
  is nowhere vanishing,
  then  
  \begin{equation*}
    V \coloneq \im \bL \subset \check S
  \end{equation*}
  is a rank one complex subbundle.
  Therefore and since $J\gamma(\xi)$ and $IJ\gamma(\xi)$ are skew-adjoint,
  $\gamma(\xi)V \subseteq JV^\perp$ for every $\xi \in T^*Z \setminus \set{0}$;
  that is:
  $V$ satisfies the criterion in \autoref{Prop_SymbolicCriterionForInftyRegularity}.  
  Moreover:
  if $\rk_\C \check S = 2$, then $J\gamma(\xi)V = V^\perp$ and $\ResidueCondition_V$ is Lagrangian.
\end{example}

\begin{remark}
  \label{Rmk_DeformZ}
  Suppose that $(Z,\fl;\Phi)$ is a $\Z/2\Z$ harmonic spinor whose branching locus $Z$ satisfies \autoref{Hyp_CodimensionTwoCooriented}.
  $\Phi$ defines an $\bL$ as in \autoref{Ex_DeformZ} as follows. 
  By \autoref{Thm_EllipticRegularityAndEstimates_1}, $\Phi\in H_a^\infty\Gamma\paren{X\setminus Z,S\otimes \fl}$. 
  Since $D\Phi=0$,
  by \cite[Lemma 1.2]{KimFriedrich2000},
  for every $v \in \Vect(X)$
  \begin{equation*}
    D(\nabla_v\Phi)
    =
    -[\nabla_v,D]\Phi = -\sum_{i=1}^n \gamma(e_i) \nabla_{\nabla_{e_i} v}\Phi - \frac12 \gamma(\Ric(v,-)^\sharp) \Phi + \sum_{i=1}^n \gamma(e_i)F_\nabla^\tw(v,e_i) \Phi
  \end{equation*}
  with $F_\nabla^{\tw} \in \Omega^2\paren{X,\fo_{\Cl}(S)}$ denoting the twisting curvature.
  In particular, by \autoref{Prop_HbkBorderlineHardyInequality}, $\nabla_v\Phi \in H_a^\infty\Gamma\paren{X\setminus Z,S\otimes \fl}$.
  A brief computation, together with \autoref{Lem_ResDelB}, shows that $\res(\nabla_v\Phi)$ depends only on $v|_Z$ and is tensorial;
  that is: it defines a section $\bL_\Phi \in \Gamma\paren{Z,\Hom(NZ,\check S)}$.
  Since $D\Phi = 0$, again by \autoref{Lem_ResDelB}, $\bL_\Phi$ is complex anti-linear.
  This is the residue condition behind the scenes in \cite{Takahashi2015,Parker2023:Deformation}. 
\end{remark}

\begin{remark}
\label{Rmk_ChiralityDeformZ}
  In the presence of a chirality operator $\epsilon$ the above discussion refines as follows:
  \begin{enumerate}
  \item 
    Let $V^+ \subset \check S^+$ be a subbundle.
    If
    \begin{equation*}
      \gamma(\xi)V^+ = J(V^+)^\perp \cap \check S^-
    \end{equation*}
    for every $\xi \in T^*Z \setminus\set{0}$,
    then
    $V \coloneq V^+ \oplus V^-$ with $V^- \coloneq J(V^+)^\perp \cap \check S^-$ satisfies the condition in \autoref{Prop_SymbolicCriterionForInftyRegularity}.
  \item
    If $\rk S^+ = 4$ and
    \begin{equation*}
      \bL \in \Gamma\paren{Z,\Hom_\C(\overline{NZ},\check S^+)}
    \end{equation*}
    is nowhere vanishing,
    then $V^+ \coloneq \im \bL$ satisfies the above condition.
  \item
    As in \autoref{Rmk_DeformZ},
    a positive $\Z/2\Z$ harmonic spinor $\Phi$ determines an $\bL_\Phi \in \Gamma\paren{Z,\Hom_\C(\overline{NZ},\check S^+)}$.
    \qedhere
  \end{enumerate}
\end{remark}

%%% Local Variables:
%%% mode: latex
%%% TeX-master: "DiracOperatorsTwistedByRamifiedLineBundles"
%%% ispell-local-dictionary: "british"
%%% End:

%% file: NoncoorientableBranchingLoci.tex
\section{Non-coorientable branching loci}
\label{Sec_NonCoorientableBranchingLoci}

The following discussion explains what changes need to be made in \autoref{Sec_GeometricRealisation} and \autoref{Sec_RegularityTheory} if $Z \subset X$ is a closed submanifold of codimension two,
but not cooriented or even coorientable.

The \defined{coorientation bundle}
\begin{equation*}
  \fo \coloneq \Wedge^2 NZ \to Z
\end{equation*}
is a Euclidean line bundle and its unique orthogonal connection is flat.
The Euclidean metric on $NZ$ identifies $\fo$ with the bundle of skew-adjoint endomorphisms of $NZ$.
The evaluation map defines an isometry
\begin{equation*}
  I \co \fo \otimes NZ \to NZ.
\end{equation*}
A moment's thought shows that the diagram
\begin{equation*}
  \begin{tikzcd}
    NZ \ar{r}{\iso} \ar[swap,bend right]{rrr}{-\one} & \fo^{\otimes 2} \otimes NZ \ar{r}{\one\otimes I} & \fo \otimes NZ \ar{r}{I} & NZ
  \end{tikzcd}
\end{equation*}
commutes.
A trivialisation $\fo \iso \ubR$ enhances $I$ to an orthogonal almost complex structure on $NZ$;
that is: a coorientation of $Z$ enhances $NZ$ to a Hermitian line bundle.

The canonical isomorphism $\fo^{\otimes 2} \iso \R$ produces the flat bundle
\begin{equation*}
  \bA \coloneq \R \oplus i \fo \to Z
\end{equation*}
of normed $\R$--algebras whose fibres are isomorphic to $\C$,
canonically up to complex conjugation.
The above discussion reveals $NZ$ to be a bundle of Euclidean $\bA$--modules of rank one.
Systematically replacing $\C$ by $\bA$ and tracking the use of $\fo$ in \autoref{Sec_GeometricRealisation} and \autoref{Sec_RegularityTheory} removes the need for a coorientation of $Z$:
\begin{itemize}
\item
  The frame bundle $\pi \co F \to Z$ defined in \autoref{Def_FrameBundle} is not $\U(1)$--principal.
  Its vertical tangent bundle $\ker T\pi$ is canonically isomorphic to $i \pi^*\fo$.
  Therefore,
  the Levi-Civita connection defines $i\theta \in \Omega^1(F,i \pi^*\fo)$;
  moreover, $\del_\alpha \in \Gamma\paren{F,\pi^*\paren{\fo \otimes NZ}}$.
\item
  \autoref{Def_US} reveals $S|_Z$ to be an $\bA$--module and
  defines $J \in \Gamma\paren{F,\End(\underline S)}$ and $I,K=IJ \in \Gamma\paren{F,\pi^*\fo \otimes \End(\underline S)}$.
  The sign ambiguities in the term $I\mathring\nabla_{\del_\alpha}$ appearing in \autoref{Rmk_ModelDiracOperator} and
  \autoref{Prop_SpectralDecomposition} cancel.
\item
  \autoref{Def_NZLambda} constructs $NZ^\lambda$ as an $\bA$--module.
  \autoref{Rmk_NZLambda}, \autoref{Prop_NZLambda}, \autoref{Prop_SpectralDecomposition} hold with $\C$ replaced by $\bA$.
  This can be seen, e.g., by passing to the double cover $\tilde Z \to Z$ defined by $\fo$. 
\item
  In the definition of the residue bundle $\check S$ and the branching locus operator $A$ in  \autoref{Def_AssemblyOfResidueMap} the appearances of $\C$ need to be replaced by $\bA$.
  $\check S$ inherits $J \in \Gamma\paren{Z,\End(\check S)}$ and $I,K=IJ \in \Gamma\paren{Z,\fo\otimes\End(\check S)}$.
\item
  $\C$ needs to be replaced by $\bA$ in \autoref{Ex_NeumannDirichlet}.
\item
  In \autoref{Ex_DeformZ}, $\C$ needs to be replaced by $\bA$, and $V$ is a rank-one $\bA$--submodule.
  In \autoref{Rmk_DeformZ}, $L_\Phi$ is $\bA$--anti-linear.
  Moreover,
  $\C$ needs to be replaced by $\bA$ in \autoref{Rmk_ChiralityDeformZ}.
\end{itemize}

%%% Local Variables:
%%% mode: latex
%%% TeX-master: "DiracOperatorsTwistedByRamifiedLineBundles"
%%% ispell-local-dictionary: "british"
%%% End: